\documentclass[11pt,a4paper,leqno]{article}
\usepackage[latin1]{inputenc}
\usepackage{amsmath,amsthm}
\usepackage{amsfonts,amssymb,latexsym}
\usepackage{graphicx}
\usepackage[french]{babel}

\usepackage[all]{xy}

\DeclareMathOperator{\ou}{ou}
\DeclareMathOperator{\Gal}{Gal}
\DeclareMathOperator{\codim}{codim}

\DeclareMathOperator{\premier}{premier}

\DeclareMathOperator{\eff}{eff}

\DeclareMathOperator{\supp}{supp}

\DeclareMathOperator{\PGCD}{pgcd}

\DeclareMathOperator{\primitifs}{primitifs}

\DeclareMathOperator{\Pic}{Pic}

\DeclareMathOperator{\card}{card}

\DeclareMathOperator{\mauvais}{mauvais}
\DeclareMathOperator{\bon}{bon}

\DeclareMathOperator{\lle}{le}
\DeclareMathOperator{\dapres}{d'apr\grave{e}s}
\DeclareMathOperator{\lemme}{lemme}

\DeclareMathOperator{\Vol}{Vol}
\newcommand{\CC}{\mathbf{C}}
\newcommand{\RR}{\mathbf{R}}
\newcommand{\ZZ}{\mathbf{Z}}
\newcommand{\NN}{\mathbf{N}}
\newcommand{\QQ}{\mathbf{Q}}
\newcommand{\xx}{\boldsymbol{x}}
\renewcommand{\ss}{\boldsymbol{s}}

\newcommand{\y}{\mathbf{y}}

\newcommand{\yy}{\boldsymbol{y}}
\newcommand{\zz}{\boldsymbol{z}}

\newcommand{\ww}{\boldsymbol{w}}
\renewcommand{\tt}{\boldsymbol{t}}
\newcommand{\jj}{\boldsymbol{j}}
\newcommand{\omeg}{\underline{\omega}_{\varepsilon}}
\newcommand{\poids}{\underline{w}_{\varepsilon}}
\newcommand{\poidsun}{\underline{w}_{\varepsilon,1}}
\newcommand{\poidsdeux}{\underline{w}_{\varepsilon,2}}
\newcommand{\cc}{\boldsymbol{c}}
\renewcommand{\aa}{\boldsymbol{a}}
\newcommand{\0}{\boldsymbol{0}}

\newcommand{\uu}{\boldsymbol{u}}
\newcommand{\vv}{\boldsymbol{v}}
\newcommand{\bb}{\boldsymbol{b}}
\newcommand{\ra}{\rightarrow}
\newcommand{\mt}{\mapsto}

\newcommand{\PP}{\mathbf{P}}

\newcommand{\OO}{\mathcal{O}}

 \newtheorem{thm}{Th\'eor\`eme}[section]
 \newtheorem{prop}[thm]{Proposition}
 \newtheorem{lemma}[thm]{Lemme}
 
 \newtheorem{Def}[thm]{D\'efinition}
  \newtheorem{rem}[thm]{Remarque}

 \usepackage[refpage]{nomencl}

\makenomenclature

\usepackage{makeidx}
\makeindex

\begin{document}

\title{\underline{POINTS DE HAUTEUR BORN\'EE} \\ \underline{ SUR LES HYPERSURFACES } \\
\underline{DE $ \PP^{n}\times \PP^{n}\times\PP^{n} $}}

\author{Teddy Mignot}

\maketitle

\begin{abstract}
Nous \'etablissons ici une formule asymptotique pour le nombre de points de hauteur born\'ee sur l'hypersurface de l'espace triprojectif $ \PP^{n}_{\QQ}\times \PP^{n}_{\QQ}\times \PP^{n}_{\QQ} $ d\'efinie par l'\'equation $ x_{0}y_{0}z_{0}+...+x_{n}y_{n}z_{n}=0 $. Nous verrons que le r\'esultat obtenu est en accord avec la conjecture de Batyrev-Manin. La m\'ethode utilis\'ee est essentiellement celle d\'evelopp\'ee par Heath-Brown dans \cite{HB1} et reprise par Robbiani dans \cite{R} pour \'evaluer le nombre de points rationnels sur une hypersurface de l'espace biprojectif.

\end{abstract}

\tableofcontents

\section{Introduction}

Suite \`a la parution de la pr\'ec\'edente version de cet article sous le num\'ero arXiv : 1402.5531v1, j'ai appris que le r\'esultat principal (\`a savoir la proposition \ref{conclusion}) avait d\'ej\`a \'et\'e obtenu auparavant dans un cadre beaucoup plus g\'en\'eral par Blomer et Br\"{u}dern dans \cite{BB}. En effet, dans cet article, Blomer et Br\"{u}dern \'etablissent une formule asymptotique pour le nombre depoints de hauteur born\'ee pour les hypersurfaces d'espaces multiprojectifs d\'efinies par des \'equations diagonales, donc en particulier pour l'hypersurface consid\'er\'ee ici. Aussi ai-je d\'ecid\'e de pr\'esenter cette nouvelle version de mon article dans laquelle je cite cette nouvelle r\'ef\'erence.\\
   
Etant donn\'e un entier $ n\geqslant 2 $, on consid\`ere l'hypersurface $ X $ de l'espace $ \PP^{n}_{\QQ}\times \PP^{n}_{\QQ}\times\PP^{n}_{\QQ}  $ d\'efinie par l'\'equation $ F(\xx,\yy,\zz)=0 $ o\`u  \[ F(\xx,\yy,\zz)=x_{0}y_{0}z_{0}+...+x_{n}y_{n}z_{n}, \] avec $ (\xx,\yy,\zz)=((x_{0}:...:x_{n}),(y_{0}:...:y_{n}),(z_{0}:...:z_{n})) \in \PP^{n}_{\QQ}\times \PP^{n}_{\QQ}\times \PP^{n}_{\QQ} $.  
On dira que $ (x_{0},...,x_{n})\in \ZZ^{n+1} $ est \emph{primitif} si $ \PGCD(x_{0},...,x_{n})=1 $
Dans tout ce qui va suivre, on note pour tout $ (\xx,\yy,\zz) \in \PP^{n}_{\ZZ}\times \PP^{n}_{\ZZ}\times \PP^{n}_{\ZZ} $ : \nomenclature{$ H(\xx,\yy,\zz) $ }{hauteur} \[ H(\xx,\yy,\zz)=\max_{0\leqslant i \leqslant n}|x_{i}| \max_{ 0\leqslant j \leqslant n} |y_{j}|\max_{0\leqslant k\leqslant n}|z_{k}|,  \]o\`u $ (x_{0},...,x_{n}), \; (y_{0},...,y_{n}),(z_{0},...,z_{n}) \in \ZZ^{n+1} $ sont primitifs et tels que \linebreak  $ (\xx,\yy,\zz)=((x_{0}:...:x_{n}),(y_{0}:...:y_{n}), \;(z_{0}:...:z_{n})) $.
On souhaite d\'eterminer une formule asymptotique pour le nombre de points de l'hypersurface $ X $ de l'espace triprojectif $ \PP^{n}_{\QQ}\times \PP^{n}_{\QQ}\times\PP^{n}_{\QQ} $  de hauteur born\'ee par $ B $, ce qui revient \`a \'evaluer \nomenclature{$ \mathcal{N}(B) $}{nombre de points primitifs} \begin{multline*} \mathcal{N}(B) =\card\{(\xx,\yy,\zz)\in \ZZ_{0}^{n+1}\times \ZZ_{0}^{n+1}\times \ZZ_{0}^{n+1} \; | \; \\ (x_{0},...,x_{n}), (y_{0},...,y_{n}), (z_{0},...,z_{n}) \; \primitifs \;,  \xx.\yy.\zz =0, \; H(\xx,\yy,\zz)\leqslant B\}, \end{multline*} o\`u l'on a not\'e \nomenclature{ $ \xx.\yy.\zz $ }{ $ =\sum_{i=0}^{n}x_{i}y_{i}z_{i} $ } $ \xx.\yy.\zz=x_{0}y_{0}z_{0}+...+x_{n}y_{n}z_{n} $ et $ \ZZ_{0}=\ZZ\setminus \{0\} $. Par une inversion de M\"{o}bius, on se ram\`ene au calcul de \nomenclature{$ N(B) $}{nombre de points}\begin{multline*} N(B)= \card\{(\xx,\yy,\zz)\in \ZZ_{0}^{n+1}\times \ZZ_{0}^{n+1}\times \ZZ_{0}^{n+1} \\ \; | \; x_{\max}>0, \; y_{\max}>0, \;  \xx.\yy.\zz=0, \;H(\xx,\yy,\zz) \leqslant B\},  \end{multline*} o\`u l'indice \og $ \max $ \fg \ indique la coordonn\'ee dont la valeur absolue est maximale. Nous allons \'evaluer ce nombre $ N(B) $ en suivant la m\'ethode d\'ecrite par Robbiani (cf. \cite{R}), qui elle m\^eme s'inspire de la m\'ethode de Heath-Brown (cf. \cite{HB1}), pour finalement arriver au r\'esultat principal, \`a savoir la proposition \ref{conclusion}, qui correspond exactement \`a la conjecture de Batyrev-Manin raffin\'ee dans ce cas. 

\section{La m\'ethode de Heath-Brown}

On introduit les notations suivantes : \nomenclature{$ \omega_{0} $}{fonction poids}
\[ \omega_{0}(x)=\left\{ \begin{array}{lcr} \exp(-(1-x^{2})^{-1}) & \mbox{si} & |x|<1 \\ 0 & \mbox{si} & |x| \geqslant 1 \end{array} \right. \]
\[ c_{0}=\int_{-\infty}^{+\infty}\omega_{0}(x)dx,\] \[ \omega(x)=4c_{0}^{-1}\omega_{0}(4x-3) \] \nomenclature{$ h(x,y) $}{fonction d'Iwaniec}
\[h(x,y)=\sum_{j>0}\frac{1}{xj}\left(\omega(xj)-\omega\left(\frac{|y|}{xj}\right)\right),\]  pour tout $ (x,y)\in \mathopen]0,+\infty\mathopen[ \times \RR $
\begin{rem}\label{rem2}
La fonction r\'eelle $ h $ est infiniment diff\'erentiable et $ h(x,y) $ est nul lorsque $ x> \max(1,2|y|) $. On a par ailleurs, pour tout $ y \in \RR $ (cf. \textup{\cite[lemme 4]{HB1}}): \[h(x,y) \ll \frac{1}{x} \] 
\end{rem}
La m\'ethode de Heath-Brown inspir\'ee de la m\'ethode du cercle classique, permet de donner une approximation du nombre de points $ \xx \in \ZZ^{n+1} $ de hauteur born\'ee v\'erifiant une \'equation $ G(\xx)=0 $, $ G $ \'etant un polyn\^ome en $ n+1 $ variables \`a coefficients entiers. Cette m\'ethode consiste, dans un premier temps, \`a consid\'erer une fonction lisse \`a support inclus dans un compact $ w : \RR^{n+1} \ra \RR $ bien choisie et \`a  estimer  \[ N(G,w)=\sum_{\xx \in \ZZ^{n+1}, \;  G(\xx)=0}w(\xx)=\sum_{\xx \in \ZZ^{n+1}}w(\xx)\delta_{G(\xx)},\] o\`u l'on a pos\'e $ \delta_{l}=\left\{ \begin{array}{lcr} 1 & \mbox{si} & l=0 \\ 0 & \mbox{si} & l\neq 0 \end{array} \right. $ pour tout entier $ l $. 
On utilise alors le th\'eor\`eme suivant : 
\begin{thm}[Duke, Friedlander, Iwaniec]
Pour tout $ l \in \ZZ $ et pour tout $ Q>1 $, il existe une constante \nomenclature{$ c_{Q} $}{$ =1+O_{N}(Q^{-N}) $} $ c_{Q}=1+O_{N}(Q^{-N}) $ pour tout $ N>0 $, telle que : \[ \delta_{l}=c_{Q}Q^{-2}\sum_{q=1}^{\infty}\sum_{d\in (\ZZ/q\ZZ)^{\ast}}e_{q}(dl)h\left(\frac{q}{Q},\frac{l}{Q^{2}}\right), \] o\`u l'on note $ e(x)=\exp(2i\pi x) $ et $ e_{q}(x)=e\left(\frac{ x}{q}\right) $.
\end{thm}
Ce th\'eor\`eme appliqu\'e \`a $ \delta_{G(\xx)} $ donne : \[ N(G,w)=\sum_{\xx\in \ZZ^{n+1}}c_{Q}Q^{-2}\sum_{q=1}^{\infty}\sum_{d\in (\ZZ/q\ZZ)^{\ast}}w(\xx)e_{q}(dG(\xx))h\left(\frac{q}{Q},\frac{G(\xx)}{Q^{2}}\right), \] pour toute constante $ Q>1 $. En utilisant une formule de Poisson en dimension $ n+1 $, Heath-Brown en d\'eduit la relation ci-dessous (cf. \cite[Th\'eor\`eme 2]{HB1}) :
\begin{equation}\label{equation}
N(G,w)=\sum_{\cc \in \ZZ^{n+1}}\sum_{q=1}^{\infty}q^{-n-1}S_{q}(\cc)I_{q}(\cc), 
\end{equation} 
o\`u \nomenclature{$ S_{q}(\cc) $}{Somme singuli\`ere}\begin{equation} S_{q}(\cc)=\sum_{d\in (\ZZ/q\ZZ)^{\ast}}\sum_{\bb\in (\ZZ/q\ZZ)^{n+1}}e_{q}(dG(\bb)+\cc.\bb), \end{equation}
et \nomenclature{$ I_{q}(\cc) $}{Int\'egrale singuli\`ere}\begin{equation} I_{q}(\cc)=\int_{\RR^{n+1}}c_{Q}Q^{-2}w(\xx)h\left(\frac{q}{Q},\frac{G(\xx)}{Q^{2}}\right)e_{q}(-\cc.\xx)d\xx. \end{equation}

Nous allons appliquer la m\'ethode de Heath-Brown \`a l'\'etude du nombre de points $ (\xx,\yy,\zz)\in \ZZ^{n+1}\times \ZZ^{n+1} \times \ZZ^{n+1} $ tels que $ \xx.\yy.\zz=0 $ et $ H(\xx,\yy,\zz)\leqslant B $. Nous verrons alors que $ N(B) $ est du type $ CB^{n}\log(B)^{2}+o(B^{n}\log(B)^{2}) $, o\`u $ C $ est une constante que nous pr\'eciserons. \\

Remarquons dans un premier temps que la condition $ H(\xx,\yy,\zz)\leqslant B $ implique que l'on a $ \max_{i}|x_{i}|\max_{j}|y_{j}|\leqslant B^{\frac{2}{3}} $ ou $ \max_{j}|y_{j}|\max_{k}|z_{k}|\leqslant B^{\frac{2}{3}} $ ou $ \max_{i}|x_{i}|\max_{k}|z_{k}|\leqslant B^{\frac{2}{3}} $ (en effet si ce n'\'etait pas le cas, la condition $ H(\xx,\yy,\zz)\leqslant B $ donnerait $ \max_{i}|x_{i}|\leqslant B^{\frac{1}{3}} $, $ \max_{j}|y_{j}|\leqslant B^{\frac{1}{3}} $ et $ \max_{k}|z_{k}|\leqslant B^{\frac{1}{3}} $ et on obtient alors une contradiction). Nous nous int\'eresserons donc exclusivement aux points $ (\xx,\yy,\zz) \in X $ v\'erifiant l'une (au moins) de ces trois conditions. Nous montrerons par la suite, dans la section \ref{dernieresection} que le nombre de points contenus dans l'intersection des domaines ainsi d\'efinis est du type $ O(B\log(B)) $, donc n\'egligeable par rapport au terme principal $ CB\log(B)^{2} $. Remarquons par ailleurs qu'une telle condition implique n\'ecessairement que l'une (au moins) des trois coordonn\'ees $ \max_{i}x_{i} $, $ \max_{j}y_{j} $, $ \max_{k}z_{k} $ est inf\'erieure ou \'egale \`a $ B^{\frac{1}{3}} $. Pour simplifier les notations, nous consid\'ererons dans un premier temps les points tels que $ \max_{i}|x_{i}|\max_{j}|y_{j}|\leqslant B^{\frac{2}{3}}$ et $ \max_{i}|x_{i}|\leqslant B^{\frac{1}{3}} $ et on supposera par ailleurs\nomenclature{$ i_{0},j_{0} $}{ indices des maxima pour $ \xx,\yy $} $ x_{i_{0}}=\max_{0\leqslant i\leqslant n}|x_{i}| $ et $ y_{j_{0}}=\max_{0\leqslant j \leqslant n}|y_{j}| $, pour $ i_{0},j_{0} \in \{0,...,n\} $ fix\'es. 
Nous allons \`a pr\'esent introduire des fonctions fonctions particuli\`eres dites fonctions poids traduisant ces conditions : \\
\begin{Def}
On appelle \emph{fonction poids sur $ \RR^{d} $} une application $ w : \RR^{d} \ra \RR $ lisse \`a support compact telle que pour tout $ \jj=(j_{1},...,j_{d}) \in \NN^{d} $ il existe $ A_{\jj}>0 $ tel que \[ \forall \xx\in \RR^{d}, \; \; \left|\frac{\partial^{j_{1}+...+j_{d}}w}{\partial x_{1}^{j_{1}}...\partial x_{d}^{j_{d}}}(\xx) \right| \leqslant A_{\jj}. \]
\end{Def}
Pour tout $ 0<\varepsilon<1 $ fix\'e, et pour tout r\'eel $ x $, on pose 
\[ \underline{\omega}_{\varepsilon}(x)=c_{0}^{-1}\int_{-\infty}^{\frac{x}{\varepsilon}-1}\omega_{0}(y)dy. \]
Remarquons que cette fonction est nulle pour tout $ x\leqslant 0 $ et vaut $ 1 $ lorsque $ x \geqslant 2\varepsilon $. On d\'efinit de m\^eme :

\[ \overline{\omega}_{\varepsilon}(x)=c_{0}^{-1}\int_{-\infty}^{\frac{x}{\varepsilon}+1}\omega_{0}(y)dy, \]
qui vaut $ 0 $ pour $ x\leqslant -2\varepsilon $ et $ 1 $ pour $ x\geqslant 0 $. On d\'efinit ensuite les fonctions poids

\nomenclature{$ \underline{w}_{\varepsilon}(x_{i_{0}},y_{j_{0}}) $}{ \eqref{(4)} fonction poids en $ x_{i_{0}},y_{j_{0}} $}
 \begin{equation}\label{(4)}
 \underline{w}_{\varepsilon,B}(x_{i_{0}},y_{j_{0}})=\underline{w}_{\varepsilon}(x_{i_{0}},y_{j_{0}})=\omeg(x_{i_{0}}-1)\omeg(B^{\frac{1}{3}}-x_{i_{0}})\underline{\omega}_{\varepsilon}(y_{j_{0}}-1)\underline{\omega}_{\varepsilon}(B^{\frac{2}{3}}-x_{i_{0}}y_{j_{0}}), 
 \end{equation} 
(d'apr\`es les remarques pr\'ec\'edentes, cette fonction vaut $ 0 $ lorsque $ x_{i_{0}}<1 $ ou $ y_{j_{0}}<1 $ ou $ x_{i_{0}}y_{j_{0}}>B^{\frac{2}{3}} $ ou $ x_{i_{0}}>B^{\frac{1}{3}} $, cette fonction traduit donc les conditions $ 1\leqslant x_{i_{0}} \leqslant B^{\frac{1}{3}} $, $  x_{i_{0}}y_{j_{0}}\leqslant B^{\frac{2}{3}} $ et $ y_{j_{0}}\geqslant 1 $ que l'on souhaite imposer \`a $ x_{i_{0}} $ et $ y_{j_{0}} $)
\nomenclature{$ \poids(\xx,\yy) $}{ \eqref{(5)} fonction poids en $ \xx,\yy $}
\begin{equation}\label{(5)}
\poids(\xx,\yy)=\poids(\xx)\poids(\yy)
\end{equation} 
avec 
\begin{equation}
\poids(\xx)=\prod_{i\neq i_{0}}\underline{\omega}_{\varepsilon}\left(1-\frac{|x_{i}|}{x_{i_{0}}}\right)\omeg\left(\frac{|x_{i}|}{x_{i_{0}}}-\frac{1}{x_{i_{0}}}\right),
\end{equation}
(ce qui traduit le fait que $ \max_{i}|x_{i}|=x_{i_{0}} $ et que $ |x_{i}|\geqslant 1 $ pour tout $ i $)
\begin{equation}
\underline{w}_{\epsilon}
(\yy)=\prod_{j\neq j_{0}}\underline{\omega}_{\varepsilon}\left(1-\frac{|y_{j}|}{y_{j_{0}}}\right)\omeg\left(\frac{|y_{j}|}{y_{j_{0}}}-\frac{1}{y_{j_{0}}}\right)
\end{equation}
 (ce qui traduit le fait que $ \max_{j}|y_{j}|=y_{j_{0}} $ et que $ |y_{j}|\geqslant 1 $ pour tout $ j $)
 \nomenclature{$ \poids(\zz) $}{\eqref{(8)} fonction poids en $ \zz $}
\begin{equation}\label{(8)}
\poids(\zz, x_{i_{0}},y_{j_{0}})=\poids(\zz)=\prod_{k=0}^{n}\underline{\omega}_{\varepsilon}\left(1-\frac{|z_{k}|}{P}\right)\underline{\omega}_{\varepsilon}\left(\frac{|z_{k}|}{P}-\frac{1}{P}\right).
\end{equation}
o\`u l'on a pos\'e $ P=B/x_{i_{0}}y_{j_{0}} $ (cette fonction traduit donc la condition $ H(\xx,\yy,\zz)\leqslant B $ et $ |z_{k}|\geqslant 1 $ pour tout $ k $). 
On pose enfin \nomenclature{$ \underline{w}_{\varepsilon,B, i_{0},j_{0}} $}{\eqref{sommeponderee} somme pond\'er\'ee}
\begin{equation}\label{sommeponderee}
\underline{w}_{\varepsilon,B, i_{0},j_{0}}(\xx,\yy,\zz)=\underline{w}_{\varepsilon,B}(x_{i_{0}},y_{j_{0}})\underline{w}_{\varepsilon}(\xx,\yy) \poids(\zz).
\end{equation} 
\begin{rem}
Il s'av\'erera souvent assez pratique par la suite d'\'ecrire les fonctions poids $ \poids(x_{i_{0}},y_{j_{0}}) $, $ \poids(\xx) $ et $ \poids(\yy) $ sous la forme : 
\begin{equation}
\underline{w}_{\varepsilon}(x_{i_{0}},y_{j_{0}})=\poids(x_{i_{0}})\poids(y_{j_{0}}),
\end{equation}
avec \begin{align}
\poids(x_{i_{0}})=\omeg(x_{i_{0}}-1)\omeg(B^{\frac{1}{3}}-x_{i_{0}}), \\ \poids(y_{j_{0}})=\underline{\omega}_{\varepsilon}(y_{j_{0}}-1)\underline{\omega}_{\varepsilon}(B^{\frac{2}{3}}-x_{i_{0}}y_{j_{0}})
\end{align}
ainsi que :
\begin{equation}
 \poids(\xx)=\prod_{i\neq i_{0}}\poids(x_{i})
\end{equation}
avec \begin{equation}
\poids(x_{i})=\underline{\omega}_{\varepsilon}\left(1-\frac{|x_{i}|}{x_{i_{0}}}\right)\omeg\left(\frac{|x_{i}|}{x_{i_{0}}}-\frac{1}{x_{i_{0}}}\right),
\end{equation}
et 
\begin{equation}
 \poids(\yy)=\prod_{j\neq j_{0}}\poids(y_{j})
\end{equation}
avec \begin{equation}
\poids(y_{j})=\underline{\omega}_{\varepsilon}\left(1-\frac{|y_{j}|}{y_{j_{0}}}\right)\omeg\left(\frac{|y_{j}|}{y_{j_{0}}}-\frac{1}{y_{j_{0}}}\right).
\end{equation}
\end{rem}
On d\'efinit de m\^eme des fonctions poids $  \overline{w}_{\varepsilon,B,i_{0},j_{0}} $ (en rempla\c{c}ant les $ \omeg $ par $ \overline{\omega}_{\varepsilon} $). On note \[ N(\underline{w}_{\varepsilon,B,i_{0},j_{0}},B) =\sum_{\xx, \yy,\zz\in \ZZ^{n+1}}\underline{w}_{\varepsilon,B,i_{0},j_{0}}(\xx,\yy,\zz)\delta_{\xx.\yy.\zz}, \] et \[ N(\overline{w}_{\varepsilon,B,i_{0},j_{0}},B) =\sum_{\xx, \yy,\zz\in \ZZ^{n+1}}\overline{w}_{\varepsilon,B,i_{0},j_{0}}(\xx,\yy,\zz)\delta_{\xx.\yy.\zz}.\] Il existe alors des constantes $ C_{1},C_{2} $ telles que  : \begin{equation}\label{encadrement} \sum_{i_{0},j_{0}}N(\underline{w}_{\varepsilon,B,i_{0},j_{0}},B)+C_{1}B\log(B)\leqslant N'(B)\leqslant\sum_{i_{0},j_{0}}N(\overline{w}_{\varepsilon,B,i_{0},j_{0}},B)+C_{2}B\log(B),\end{equation} o\`u \nomenclature{$ N'(B) $}{\eqref{N'B}} \begin{multline}\label{N'B} N'(B)=\card\{ (\xx,\yy,\zz)\in \ZZ_{0}^{n+1}\times \ZZ_{0}^{n+1}\times \ZZ_{0}^{n+1}\; |  \; x_{\max} >0, \; y_{\max}>0,\; \\ \xx.\yy.\zz=0 ,\;   H(\xx,\yy,\zz)\leqslant B, \; \max_{i}|x_{i}|\max_{j}|y_{j}|\leqslant B^{\frac{2}{3}}, \; \max_{i}|x_{i}|\leqslant B^{\frac{1}{3}} \} \end{multline} (nous verrons en effet dans la partie \ref{dernieresection} que le nombre de points contenus dans les intersections d\'efinies par $ x_{\max}=x_{i_{0}}=x_{i_{0}'} $, avec $ i_{0}\neq i_{0}' $, ou $ y_{\max}=y_{j_{0}}=y_{j_{0}'} $,avec $ j_{0}\neq j_{0}' $ est du type $ O(B\log(B)) $). Nous allons montrer que pour tous $ i_{0},j_{0} $ et pour $ \varepsilon $ assez petit, il existe une constante $ C_{i_{0},j_{0}} $ telle que \[ N(\underline{w}_{\varepsilon,B,i_{0},j_{0}},B)\sim_{B\ra \infty}N(\overline{w}_{\varepsilon,B,i_{0},j_{0}},B)\sim_{B\ra \infty}C_{i_{0},j_{0}}B^{n}\log(B)^{2}, \] ce qui nous permettra de conclure quant \`a la valeur asymptotique de $ N'(B) $.\\

On cherche \`a \'evaluer : 
\begin{align*}
N(\underline{w}_{\varepsilon,B,i_{0},j_{0}},B) & =\sum_{\xx, \yy,\zz\in \ZZ^{n+1}}\underline{w}_{\varepsilon,B,i_{0},j_{0}}(\xx,\yy,\zz)\delta_{\xx.\yy.\zz} \\ 
& = \sum_{x_{i_{0}},y_{j_{0}}}\poids(x_{i_{0}},y_{j_{0}})\sum_{\substack{(x_{i})_{i\neq i_{0}} \\ (y_{j})_{j\neq j_{0}} }}\poids(\xx,\yy)\sum_{\zz}\poids(\zz)\delta_{\xx.\yy.\zz}. \end{align*} 
o\`u $ \xx.\yy.\zz=x_{0}y_{0}z_{0}+...+x_{n}y_{n}z_{n} $. On applique alors la formule \eqref{equation}, et on obtient  $ N(\underline{w}_{\varepsilon,B,i_{0},j_{0}},B) $ sous la forme :

\begin{equation}\label{generalite2}
\sum_{x_{i_{0}},y_{j_{0}}}\poids(x_{i_{0}},y_{j_{0}})\sum_{\substack{(x_{i})_{i\neq i_{0}} \\ (y_{j})_{j\neq j_{0}} }}\poids(\xx,\yy)\sum_{\cc\in \ZZ^{n+1}}\sum_{q=1}^{\infty}q^{-n-1}S_{q}(\xx,\yy)(\cc)I_{q}(\xx,\yy)(\cc), \end{equation} o\`u l'on a not\'e :
\nomenclature{ $ S_{q}(\xx,\yy)(\cc) $ } {\eqref{sommesing} somme singuli\`ere en $ \xx,\yy $}
 \begin{equation}\label{sommesing} S_{q}(\xx,\yy)(\cc)=\sum_{d\in (\ZZ/q\ZZ)^{\ast}}\sum_{\bb \in (\ZZ/q\ZZ)^{n+1}}e_{q}(d\xx.\yy.\bb+\cc.\bb) \end{equation}\nomenclature{ $ I_{q}(\xx,\yy)(\cc) $ } {\eqref{intsing} somme singuli\`ere en $ \xx,\yy $}
 \begin{equation}\label{intsing} I_{q}(\xx,\yy)(\cc)=c_{Q}\int_{\RR^{n+1}}\frac{\poids(\zz)}{Q^{2}}h\left(\frac{q}{Q},\frac{\xx.\yy.\zz}{Q^{2}}\right)e_{q}(-\cc.\zz)d\zz,       \end{equation}
avec $ c_{Q}=1+O_{N}(Q^{-N}) $ pour tout $ N>0 $.
\begin{rem}
D'apr\`es la remarque \ref{rem2}, $ h\left(\frac{q}{Q},\frac{\xx.\yy.\zz }{Q^{2}}\right) $ est nul si $ \frac{q}{Q}\leqslant\frac{\xx.\yy.\zz }{Q^{2}} $. Comme $ \frac{\xx.\yy.\zz }{Q^{2}}\ll 1 $ sur le support de $ \poids $, $ I_{q}(\xx,\yy)(\cc) $ vaut $ 0 $ si l'on a pas $ \frac{q}{Q} \ll 1 $. Par cons\'equent, on peut restreindre la somme sur $ q $ aux entiers $ q\ll Q $ lorsque $ Q\geqslant  \sqrt{B} $ (En fait, nous fixerons $ Q=\sqrt{B} $).
\end{rem}

Nous allons voir par la suite que les $ \cc\neq \0 $ fourniront un terme d'erreur, et le terme principal sera obtenu \`a partir de la partie $ \cc=\0 $. Auparavant, nous allons effectuer un changement de variables $ \zz=P\uu=\frac{B}{x_{i_{0}}y_{j_{0}}}\uu $ dans l'int\'egrale $ I_{q}(\xx,\yy)(\cc) $. On obtient alors : 
\begin{align*}
 I_{q}(\xx,\yy)(\cc) & = c_{Q}\int_{\RR^{n+1}}\frac{\poids(\zz)}{Q^{2}}h\left(\frac{q}{Q},\frac{\xx.\yy.\zz}{Q^{2}}\right)e_{q}(-\cc.\zz)d\zz \\ & =c_{Q}\frac{P^{n+1}}{Q^{2}}\int_{\RR^{n+1}}\poids(\uu)h\left(\frac{q}{Q},\frac{P}{Q^{2}}\xx.\yy.\uu\right)e_{q}(-P\cc.\uu)d\uu,
 \end{align*}
o\`u, par abus de langage, on note \'egalement \nomenclature{$ \poids(\uu) $}{fonction poids en $ \uu $} $ \poids(\uu) $ la nouvelle fonction poids en la variable $ \uu $. Remarquons que sur $ \supp(\poids) $ on a alors $ \max_{k}|u_{k}|\leqslant 1 $. On pose \`a pr\'esent $ Q=\sqrt{B} $ de sorte que \[ \frac{P}{Q^{2}}\xx.\yy.\uu=\frac{\xx.\yy.\uu}{x_{i_{0}}y_{j_{0}}}=\sum_{k=0}^{n} \frac{x_{k}y_{k}}{x_{i_{0}}y_{j_{0}}}u_{k} \] qui est une forme lin\'eaire en $ \uu $ \`a coefficients born\'es. On note par la suite : \nomenclature{$ I(\cc) $}{ \eqref{oubli} $ =c_{Q}\frac{P^{n+1}}{Q^{2}} I(\cc) $} \begin{equation}\label{oubli}
I(\cc)=I(\cc,\xx,\yy)=\int_{\RR^{n+1}}\poids(\uu)h\left(\frac{q}{Q},\frac{\xx.\yy.\uu}{x_{i_{0}}y_{j_{0}}}\right)e_{q}(-P\cc.\uu)d\uu.
\end{equation}
et on a donc \begin{align}\label{Iqc}
 I_{q}(\xx,\yy)(\cc) & =c_{Q}\frac{P^{n+1}}{Q^{2}} I(\cc) \\  & =c_{Q}\frac{B^{n}}{x_{i_{0}}^{n+1}y_{j_{0}}^{n+1}}I(\cc). 
 \end{align}
Ceci nous permet de r\'e\'ecrire la formule \eqref{generalite2} sous la forme : 
\begin{multline}\label{nouvformule}
N(\underline{w}_{\varepsilon,B,i_{0},j_{0}},B)= c_{Q}B^{n}\sum_{x_{i_{0}},y_{j_{0}}}\frac{\poids(x_{i_{0}},y_{j_{0}})}{x_{i_{0}}^{n+1}y_{j_{0}}^{n+1}}\sum_{\substack{(x_{i})_{i\neq i_{0} }\\ (y_{j})_{j\neq j_{0}} }}\poids(\xx,\yy)\\ \sum_{\cc\in \ZZ^{n+1}}\sum_{q\ll Q}q^{-n-1}S_{q}(\xx,\yy)(\cc)I(\cc).
\end{multline}

\section{Le terme d'erreur}\label{PartieTE}

 Nous allons d\'emontrer que la somme 
\begin{equation}\label{termerreur2}
c_{Q}B^{n}\sum_{x_{i_{0}},y_{j_{0}}}\frac{\poids(x_{i_{0}},y_{j_{0}})}{x_{i_{0}}^{n+1}y_{j_{0}}^{n+1}}\sum_{\substack{(x_{i})_{i\neq i_{0} }\\ (y_{j})_{j\neq j_{0}} }}\poids(\xx,\yy)\sum_{\cc\neq \0}\sum_{q\ll Q}q^{-n-1}S_{q}(\xx,\yy)(\cc)I(\cc)
\end{equation}
est du type $ O_{\epsilon}(B^{n}) $.
Nous allons pour cela commencer par majorer les int\'egrales $ I(\cc) $ par les m\'ethodes d\'ecrites par Heath-Brown dans \cite[Chapitres 7 et 8]{HB1}. 
 
 \subsection{Estimations simples de $ I(\cc) $ } \label{estimationsimples}

Nous allons \`a pr\'esent construire une nouvelle fonction poids $ \psi(\uu) $. Pour cela nous allons utiliser le lemme suivant (voir \cite[Lemme 3]{HB1}) : 
\begin{lemma}
Soit $ w $ une fonction poids sur $ \RR^{n+1} $ et $ K_{1},...,K_{n},... $ une suite de r\'eels strictement positifs fix\'es. Soit $ g : \RR^{n+1} \ra \CC $ d\'efinie sur $ \supp(w) $ et satisfaisant $ |g|\gg 1 $ et \[ \left|\frac{\partial^{j_{0}+...+j_{n}}g(\xx)}{\partial x_{0}^{j_{0}}...\partial x_{n}^{j_{n}}}\right| \leqslant K_{j} \; \; \; \; ( \mbox{pour} \; j=\sum_{k=0}^{n} j_{k} \geqslant 1), \] sur $ \supp(w) $. Alors la fonction $ \frac{w}{g} $ est une fonction poids. 
\end{lemma}
On pose \nomenclature{$ \omega_{1}(s) $}{$ =\omega_{0}\left(\frac{s}{2n}\right) $}$ \omega_{1}(s)=\omega_{0}\left(\frac{s}{2n}\right) $. Puisque, sur $ \supp(\poids(\uu)) $, $ \max_{k}|u_{k}|\leqslant 1 $ on a  $ \left|\frac{\xx.\yy.\uu}{x_{i_{0}}y_{j_{0}}}\right| \leqslant n  $ et donc $ \omega_{1}\left(\frac{ \xx.\yy.\uu}{x_{i_{0}}y_{j_{0}}}\right) \gg 1 $ sur $ \supp(\poids(\uu)) $. Par ailleurs, $ \omega_{0} $ \'etant une fonction poids, toutes les d\'eriv\'ees partielles en $ \uu $ de $ \omega_{1}\left(\frac{ \xx.\yy.\uu}{x_{i_{0}}y_{j_{0}}}\right) $ sont born\'ees. Par cons\'equent, par application du lemme ci-dessus, la fonction \nomenclature{$ \psi(\uu) $}{$ =\frac{\poids(\uu)}{\omega_{1}\left(\frac{ \xx.\yy.\uu}{x_{i_{0}}y_{j_{0}}}\right)} $}
\[ \psi(\uu)=\frac{\poids(\uu)}{\omega_{1}\left(\frac{ \xx.\yy.\uu}{x_{i_{0}}y_{j_{0}}}\right)} \] est une fonction poids. On peut alors exprimer $ I(\cc) $ sous la forme : 
\[ I(\cc)=\int_{\RR^{n+1}}\psi(\uu)\omega_{1}\left(\frac{ \xx.\yy.\uu}{x_{i_{0}}y_{j_{0}}}\right)h\left(\frac{q}{Q},\frac{\xx.\yy.\uu}{x_{i_{0}}y_{j_{0}}}\right)e_{q}(-P\cc.\uu)d\uu. \] 
Par inversion de fourier, on a que : 
\[ \omega_{1}\left( \frac{\xx.\yy.\uu}{x_{i_{0}}y_{j_{0}}}\right)h\left(\frac{q}{Q},\frac{\xx.\yy.\uu}{x_{i_{0}}y_{j_{0}}}\right)=\int_{-\infty}^{\infty}p(t)e\left(t\frac{\xx.\yy.\uu}{x_{i_{0}}y_{j_{0}}}\right)dt \]
o\`u  \nomenclature{$ p(t) $}{$ =\int_{-\infty}^{\infty}\omega_{1}(s)h\left(\frac{q}{Q},s\right)e(-ts)ds $}\[ p(t)=\int_{-\infty}^{\infty}\omega_{1}(s)h\left(\frac{q}{Q},s\right)e(-ts)ds. \]
On obtient alors \begin{equation}\label{I(c)} I(\cc)=\int_{-\infty}^{\infty}p(t)\int_{\RR^{n+1}}\psi(\uu)e(\phi(\uu,t))d\uu dt, \end{equation}
o\`u l'on a pos\'e \nomenclature{$ \phi(\uu,t) $}{$ =t\frac{\xx.\yy.\uu}{x_{i_{0}}y_{j_{0}}}-\frac{P\cc}{q}.\uu $}\[ \phi(\uu,t)=t\frac{\xx.\yy.\uu}{x_{i_{0}}y_{j_{0}}}-\frac{P\cc}{q}.\uu.\]

D'apr\`es les propri\'et\'es de la fonction $ h $ \'enonc\'ees par Heath-Brown, on a (voir \cite[Lemme 5]{HB1}) : 
\[ \frac{d^{N}}{ds^{N}}\left(\omega_{1}(s)h\left(\frac{q}{Q},s\right)\right) \ll \left(\frac{q}{Q}\right)^{-N-1}\min\{1,\left(\frac{q}{Q|s|}\right)^{2}\} \] pour tout entier $ N> 0 $. Par int\'egrations par parties on a alors : \begin{align*} p(t) & =\int_{-\infty}^{\infty}  \omega_{1}(s)h\left(\frac{q}{Q},s\right)e(-ts)ds \\ & =(-t)^{-N}\int_{-\infty}^{\infty} \frac{d^{N}}{ds^{N}}\left( \omega_{1}(s)h\left(\frac{q}{Q},s\right)\right)e(-ts)ds \\ & \ll |t|^{-N}\int_{-\infty}^{\infty} \left(\frac{q}{Q}\right)^{-N-1}\min\{1,\left(\frac{q}{Q|s|}\right)^{2}\}ds \\ & \ll  \left(\frac{q}{Q}\right)^{-N-1}|t|^{-N}\left(\int_{|s|\leqslant \frac{q}{Q}}ds + \left(\frac{q}{Q}\right)^{2}\int_{|s|>\frac{q}{Q}}\frac{1}{|s|^{2}}ds\right) \\ & \ll   \left(\frac{q}{Q}\right)^{-N}|t|^{-N} , \end{align*}
pour tout entier $ N > 0 $.  On a donc \'etabli : \begin{equation}\label{p(t)}
p(t) \ll \left(\frac{q}{Q}|t|\right)^{-N} \end{equation} pour tout $ t \in \RR $ et $ N > 0 $. Pour le cas $ N=0 $, \cite[Lemme 5]{HB1} donne $ \left|\omega_{1}(s)h\left(\frac{q}{Q},s\right)\right|\ll \frac{Q}{q} $, et donc on trouve \begin{equation}\label{p(t)2}
p(t)\ll \frac{Q}{q}, 
\end{equation}\'etant donn\'e que $ \supp(\omega_{1}) $ est de taille $ O(1) $. Ces majorations nous seront utiles dans ce qui va suivre.\\

Pour majorer l'int\'egrale $ I(\cc) $, nous allons consid\'erer la formule \eqref{I(c)}, et s\'eparer l'int\'egrale sur $ t $ en deux parties. Lorsque $ |t|\leqslant \frac{P|\cc|}{3q} $, on remarque que, si $ |c_{l}|=\max_{k}|c_{k}| $, alors on a que \begin{align*} \left|\frac{\partial\phi}{\partial u_{l}}(\uu,t)\right| & =\left| t\frac{x_{l}y_{l}}{x_{i_{0}}y_{j_{0}}}-\frac{Pc_{l}}{q}\right| \\ & \geqslant \frac{Pc_{l}}{q}-|t|\left|\frac{x_{l}y_{l}}{x_{i_{0}}y_{j_{0}}}\right| \\ & \geqslant \frac{P|\cc|}{3q}, \end{align*} car $ \left|\frac{x_{l}y_{l}}{x_{i_{0}}y_{j_{0}}}\right|\leqslant (1+\varepsilon) $ sur $ \supp(\poids(\xx,\yy)) $. Nous allons alors utiliser le lemme suivant (voir \cite[Lemme 10]{HB1}) :

  \begin{lemma}\label{lemme10}
  Soit $ w : \RR^{n+1} \ra \RR $  une fonction poids et $ f(\xx) $ une fonction $ C^{\infty} $ \`a valeurs r\'eelles d\'efinie sur $ \supp(w) $. On suppose qu'il existe $ \lambda>0 $ et un ensemble de r\'eels strictement positifs $ \{ A_{2},A_{3},A_{4},...\} $ tels que pour tout $ \xx \in \supp(w) $ on a : 
  \[ |\nabla f | \geqslant \lambda, \] et
  \[ \left| \frac{\partial^{j_{0}+...+j_{n}} f(\xx)}{\partial^{j_{0}} x_{0}...\partial^{j_{n}} x_{n}}\right| \leqslant A_{j}\lambda, \; \; \; \;  (j=j_{0}+...+j_{n}\geqslant 2 ).\]
  Alors pour tout $ N>0 $, on a : 
  \[ \int_{\RR^{n+1}} w(\xx)e(f(\xx))d\xx \ll \lambda^{-N}. \]
  \end{lemma}

Puisque l'on a $ |\nabla \phi(\uu,t)| \gg \frac{P|\cc|}{q} $ et que les d\'eriv\'ees partielles sup\'erieures en $ \uu $ de $ \phi(\uu,t) $ sont toutes nulles, on a par le lemme \ref{lemme10} pour $ |t|\leqslant \frac{P|\cc|}{3q}  $: \begin{equation}\label{tpetit} \left| \int_{\RR^{n+1}} \psi(\uu)e(\phi(\uu,t))d\uu \right| \ll_{M} \left( \frac{P|\cc|}{q}\right)^{-M} \end{equation}
pour tout entier $ M>0 $. \\

Lorsque $ |t|\geqslant \frac{P|\cc|}{3q}  $, nous utilisons la majoration triviale  \begin{equation}\label{majotriviale} \left| \int_{\RR^{n+1}} \psi(\uu)e(\phi(\uu,t))d\uu \right| \ll 1,  \end{equation} (due au fait que $ |\psi(\uu)|\ll 1 $ sur $ \supp(\poids(\uu)) $ qui est de mesure $ O(1) $). En utilisant par ailleurs la majoration \eqref{p(t)2} pour $ |t|\leqslant \frac{P|\cc|}{3q}  $ et \eqref{p(t)} avec $ N=M $ pour $ |t|\geqslant \frac{P|\cc|}{3q}  $, on obtient l'estimation suivante de $ I(\cc) $ : 
\begin{align*} I(\cc) & \ll \frac{Q}{q}\left( \frac{P|\cc|}{q}\right)^{1-M} + \left( \frac{P|\cc|}{q}\right)^{1-M}\left(\frac{q}{Q}\right)^{-M} \\ &  \ll  \left( \frac{P|\cc|}{q}\right)^{1-M}\left(\frac{q}{Q}\right)^{-M} \end{align*}
(rappelons que l'on a $ q\ll Q $). On a ainsi \'etabli le lemme suivant : 

\begin{lemma}\label{lemmemajo1}
Pour tout entier $ N \geqslant 0 $, on a : \[ I(\cc) \ll \frac{Q}{q}\left(\frac{Q}{P|\cc|}\right)^{N}=\frac{Q}{q}\left(\frac{x_{i_{0}}y_{j_{0}}}{\sqrt{B}|\cc|}\right)^{N}. \]
\end{lemma}  
Remarquons dans un premier temps que lorsque $ \sqrt{B}^{1+\nu} \geqslant x_{i_{0}}y_{j_{0}} $ \nomenclature{$ \nu $}{ r\'eel strictement positif arbitrairement petit}(avec $ \nu>0 $ arbitrairement petit), le lemme ci-dessus donne $ I(\cc) \ll \frac{Q}{q}|\cc|^{-N}B^{\frac{\nu N}{2}} $ pour tout $ N\geqslant 0 $, et la somme sur les $ \cc $ tels que $ |\cc|\geqslant B^{\frac{1}{6}} $ donnera une contribution n\'egligeable : en effet, si $ |\cc|\geqslant  B^{\frac{1}{6}} $, alors pour $ N $ assez grand et $ \nu  $ assez petit (par exemple $ N=2(n+2) $ et $ \nu\leqslant \frac{1}{12} $), on aura \[ |\cc|^{N}\geqslant B^{\frac{n+2}{6}}|\cc|^{n+2}\geqslant B^{\frac{n+2}{12}+\frac{\nu N}{2}}|\cc|^{n+2}\geqslant B^{\frac{3}{2}+\frac{\nu N}{2}}|\cc|^{n+2} \](pour $ n\geqslant 16 $), donc $ |I(\cc)|\ll \frac{Q}{q}|\cc|^{-N}B^{\frac{\nu N}{2}}\ll \frac{B^{-1}}{q|\cc|^{-n-2}} $ et par cons\'equent, en restreignant la somme \eqref{termerreur2} aux $ |\cc|\geqslant B^{\frac{1}{6}} $ on a la majoration (en utilisant la majoration triviale $ S_{q}(\xx,\yy)(\cc) \ll q^{n+2} $) : \begin{multline*}
c_{Q}B^{n}\sum_{x_{i_{0}},y_{j_{0}}}\frac{\poids(x_{i_{0}},y_{j_{0}})}{x_{i_{0}}^{n+1}y_{j_{0}}^{n+1}}\sum_{\substack{(x_{i})_{i\neq i_{0} }\\ (y_{j})_{j\neq j_{0}} }}\poids(\xx,\yy)\sum_{\substack{\cc\neq \0 \\ |\cc|\geqslant B^{\frac{1}{6}}}}\sum_{q\ll Q}q^{-n-1}S_{q}(\xx,\yy)(\cc)I(\cc) \\ \ll B^{n-1}\sum_{x_{i_{0}},y_{j_{0}}}\frac{\poids(x_{i_{0}},y_{j_{0}})}{x_{i_{0}}^{n+1}y_{j_{0}}^{n+1}}\sum_{\substack{(x_{i})_{i\neq i_{0} }\\ (y_{j})_{j\neq j_{0}} }}\poids(\xx,\yy)\left(\sum_{\substack{\cc\neq \0 \\ |\cc|\geqslant B^{\frac{1}{6}}}}\frac{1}{|\cc|^{n+2}}\right)\left(\sum_{q\ll Q}1\right) \\ \ll B^{n-\frac{1}{2}}\sum_{x_{i_{0}},y_{j_{0}}}\frac{\poids(x_{i_{0}},y_{j_{0}})}{x_{i_{0}}^{n+1}y_{j_{0}}^{n+1}}\sum_{\substack{(x_{i})_{i\neq i_{0} }\\ (y_{j})_{j\neq j_{0}} }}\poids(\xx,\yy) \\ \ll B^{n-\frac{1}{2}}\sum_{x_{i_{0}},y_{j_{0}}}\frac{\poids(x_{i_{0}},y_{j_{0}})}{x_{i_{0}}y_{j_{0}}}  \ll B^{n-\frac{1}{2}}\log(B)^{2}.
\end{multline*}
Cette partie est donc n\'egligeable par rapport au terme principal, qui rappelons-le sera du type $ O(B^{n}\log(B)^{2}) $.\\

Consid\'erons \`a pr\'esent le cas o\`u $ \sqrt{B}^{1+\nu} \leqslant x_{i_{0}}y_{j_{0}} $. On suppose dans un premier temps que $ |\cc|\geqslant B $, le lemme \ref{lemmemajo1} donne alors (en rappelant que $ x_{i_{0}}y_{j_{0}} \leqslant B^{\frac{2}{3}} $, d'\`apr\`es \eqref{(4)}) :  \[ I(\cc) \ll \frac{Q}{q}\left(\frac{x_{i_{0}}y_{j_{0}}}{\sqrt{B}|\cc|}\right)^{N} \ll \frac{Q}{q}\left(\frac{x_{i_{0}}y_{j_{0}}}{B}\right)^{N} |\cc|^{-\frac{N}{2}} \ll \frac{Q}{q}B^{-\frac{N}{3}}|\cc|^{-\frac{N}{2}}, \] et en choisissant $ N=2(n+2) $, puis en majorant $ I(\cc) $ par cette expression dans \eqref{termerreur2}, on obtient une contribution : \begin{multline*}
 B^{\frac{n}{3}-\frac{4}{3}+\frac{1}{2}}\sum_{x_{i_{0}},y_{j_{0}}}\frac{\poids(x_{i_{0}},y_{j_{0}})}{x_{i_{0}}^{n+1}y_{j_{0}}^{n+1}}\sum_{\substack{(x_{i})_{i\neq i_{0} }\\ (y_{j})_{j\neq j_{0}} }}\poids(\xx,\yy)\left(\sum_{\substack{\cc\neq \0 \\ |\cc|\geqslant B}}\frac{1}{|\cc|^{n+2}}\right)\left(\sum_{q\ll Q}1\right) \\ \ll B^{\frac{n}{3}-\frac{1}{3}}\sum_{x_{i_{0}},y_{j_{0}}}\frac{\poids(x_{i_{0}},y_{j_{0}})}{x_{i_{0}}y_{j_{0}}} \ll B^{\frac{n}{3}-\frac{1}{3}}\log(B)^{2},
\end{multline*} donc \`a nouveau un terme n\'egligeable par rapport au terme principal. On suppose donc \`a pr\'esent $ |\cc| \leqslant B $ et $ |\cc|\gg \left(\frac{x_{i_{0}}y_{j_{0}}}{\sqrt{B}}\right)^{1+\delta} $, avec $ \delta>0 $ arbitrairement petit. On a alors \[ I(\cc) \ll \frac{Q}{q}\left(\frac{x_{i_{0}}y_{j_{0}}}{\sqrt{B}}\right)^{-N\delta} \leqslant \frac{Q}{q}B^{-\frac{N\delta\nu}{2}}, \] ($ \nu $ et $ \delta $ \'etant fix\'es) et donc en prenant $ N $ assez grand, on voit que la somme sur les $ \cc $ tels que $ |\cc|\gg \left(\frac{x_{i_{0}}y_{j_{0}}}{\sqrt{B}}\right)^{1+\delta} $ sera n\'egligeable par rapport \`a $ B^{n}\log(B)^{2} $ : en effet, si $ N $ est assez grand on a $ B^{-\frac{N\delta\nu}{2}} \ll B^{-n-\frac{7}{2}}\ll B^{-\frac{3}{2}}|\cc|^{-n-2} $, et donc $ I(\cc) \ll \frac{B^{-1}}{q}|\cc|^{-n-2} $ et on retrouve un terme d'erreur n\'egigeable par rapport \`a $ B^{n}\log(B)^{2} $, comme pr\'ec\'edemment. On supposera donc d\'esormais que $ |\cc|\ll B^{\frac{1}{6}} $ pour les $ x_{i_{0}},y_{j_{0}} $ tels que $ \sqrt{B}^{1+\nu} \geqslant x_{i_{0}}y_{j_{0}} $ et que  $ |\cc|\ll \left(\frac{x_{i_{0}}y_{j_{0}}}{\sqrt{B}}\right)^{1+\delta} $ pour $ \sqrt{B}^{1+\nu} \leqslant x_{i_{0}}y_{j_{0}} $. Remarquons que puisque l'on a suppos\'e $ x_{i_{0}}y_{j_{0}}\leqslant B^{\frac{2}{3}} $, si l'on a $ |\cc|\ll \left(\frac{x_{i_{0}}y_{j_{0}}}{\sqrt{B}}\right)^{1+\delta} $ alors $  |\cc|\ll B^{(\frac{2}{3}-\frac{1}{2})(1+\delta)}=B^{\frac{1}{6}+\frac{1}{6}\delta} $. On \'etudiera donc plus g\'en\'eralement le cas des $ \cc  $ tels que \begin{equation}\label{paramc}
|\cc|\ll B^{\frac{1}{6}+\delta} 
\end{equation}  (avec $ \delta>0 $ arbitrairement petit), et ce quels que soient $ x_{i_{0}} $ et $ y_{j_{0}}  $. \\

Pour pouvoir traiter cette partie de la somme \eqref{termerreur2}, nous allons devoir utiliser des estimations plus fines de $ I(\cc) $, inspir\'ees de celles d\'evelopp\'ees par Heath-Brown dans \cite[Chapitre 8]{HB1}.

\subsection{Estimations plus fines de $ I(\cc) $ }

Remarquons avant tout que, pour l'\'evaluation du terme d'erreur \eqref{termerreur2}, il est possible de restreindre la somme sur $ q $ aux entiers $ q\leqslant x_{i_{0}}y_{j_{0}} $. En effet, si l'on ne consid\`ere que les entiers $ q>x_{i_{0}}y_{j_{0}} $, on utilise alors la majoration du lemme \ref{lemmemajo1} en prenant $ N=n+2 $, et on obtient alors, dans \eqref{termerreur2}, un terme d'erreur 
\begin{multline}\label{majosimple}c_{Q}B^{n}\sum_{x_{i_{0}},y_{j_{0}}}\frac{\poids(x_{i_{0}},y_{j_{0}})}{x_{i_{0}}^{n+1}y_{j_{0}}^{n+1}}\sum_{\substack{(x_{i})_{i\neq i_{0}} \\ (y_{j})_{j\neq j_{0}} }}\poids(\xx,\yy)\sum_{\cc\neq \0} \\ \sum_{x_{i_{0}}y_{j_{0}}<q\leqslant Q}q^{-n-1}|S_{q}(\xx,\yy)(\cc)|\frac{Q}{q}\left(\frac{x_{i_{0}}y_{j_{0}}}{\sqrt{B}|\cc|}\right)^{n+2} \\ =c_{Q}B^{\frac{n}{2}-\frac{1}{2}} \sum_{x_{i_{0}},y_{j_{0}}}\poids(x_{i_{0}},y_{j_{0}})x_{i_{0}}y_{j_{0}}\sum_{\cc\neq \0}\frac{1}{|\cc|^{n+2} } \\ \sum_{x_{i_{0}}y_{j_{0}}<q\leqslant Q}q^{-n-2}\sum_{\substack{(x_{i})_{i\neq i_{0}} \\ (y_{j})_{j\neq j_{0}}}}\poids(\xx,\yy)|S_{q}(\xx,\yy)(\cc)|. \end{multline}

Remarquons \`a pr\'esent que l'on a : \begin{multline}\label{qpetit} \sum_{\substack{(x_{i})_{i\neq i_{0}} \\ (y_{j})_{j\neq j_{0}}}}\poids(\xx,\yy)|S_{q}(\xx,\yy)(\cc)| \\  =\sum_{\substack{(x_{i})_{i\neq i_{0}} \\ (y_{j})_{j\neq j_{0}}}}\poids(\xx,\yy)\left|\sum_{d\in (\ZZ/q\ZZ)^{\ast}}\sum_{\bb\in (\ZZ/q\ZZ)^{n+1}}e_{q}(d\xx.\yy.\bb+\cc.\bb)\right| \\  \leqslant\sum_{d\in (\ZZ/q\ZZ)^{\ast}}\sum_{\substack{(x_{i})_{i\neq i_{0}} \\ (y_{j})_{j\neq j_{0}}}}\poids(\xx,\yy)\left|\sum_{\bb\in (\ZZ/q\ZZ)^{n+1}}e_{q}(d\xx.\yy.\bb+\cc.\bb)\right|\\  \leqslant\sum_{d\in (\ZZ/q\ZZ)^{\ast}}\sum_{\substack{(x_{i})_{i\neq i_{0}} \\ (y_{j})_{j\neq j_{0}}}}\poids(\xx,\yy)\prod_{k=0}^{n}\sum_{b\in \ZZ/q\ZZ}e_{q}(dx_{k}y_{k}b+c_{k}b) .\end{multline} Pour $ d $ fix\'e, la somme $ \sum_{b\in \ZZ/q\ZZ}e_{q}(dx_{k}y_{k}b+c_{k}b) $ vaut $ q $ si $ dx_{k}y_{k}+c_{k}\equiv 0 (q) $ (c'est-\`a dire si  $ x_{k}y_{k}\equiv -d^{-1}c_{k} (q)  $) et $ 0 $ sinon. Or, l'\'equation $ x_{k}y_{k}\equiv -d^{-1}c_{k} (q)  $ admet un nombre $ O_{\varepsilon}(q^{\epsilon}) $ de solutions (avec $ \epsilon>0 $ arbitrairement petit) $ (x_{k},y_{k}) $ sur $ \supp\poids(x_{k},y_{k}) $, \'etant donn\'e que, $ |x_{k}y_{k}|\leqslant (1+\varepsilon)^{2}x_{i_{0}}y_{j_{0}} < (1+\varepsilon)^{2}q $ pour tout $ k\in \{0,...,n\} $ et tout $ (x_{k},y_{k}) \in \supp\poids(x_{k},y_{k}) $. Par cons\'equent, $ \prod_{k=0}^{n}\sum_{b\in \ZZ/q\ZZ}e_{q}(dx_{k}y_{k}b+c_{k}b)  $ peut \^etre major\'ee par \[ \left|\prod_{k=0}^{n} \sum_{b\in \ZZ/q\ZZ}e_{q}(dx_{k}y_{k}b+c_{k}b)\right| \ll\prod_{k=0}^{n}q = q^{n+1} \] pour un nombre $ O_{\varepsilon}(q^{\epsilon}) $ de $ (\xx,\yy) \in \supp \poids(\xx,\yy) $ et vaut $ 0 $ pour tous les autres valeurs de $ (\xx,\yy) $. On peut donc \'ecrire la majoration : 
\[  \sum_{\substack{(x_{i})_{i\neq i_{0}} \\ (y_{j})_{j\neq j_{0}} }}\poids(\xx,\yy)\left|S_{q}(\xx,\yy)(\cc)\right| \ll \sum_{d\in (\ZZ/q\ZZ)^{\ast}}q^{n+1+\epsilon}\ll q^{n+2+\epsilon}. \] On obtient donc 
\[ \eqref{majosimple} \ll B^{\frac{n}{2}-\frac{1}{2}}\sum_{x_{i_{0}},y_{j_{0}}}\poids(x_{i_{0}},y_{j_{0}})x_{i_{0}}y_{j_{0}}\sum_{\cc\neq \0}\frac{1}{|\cc|^{n+2} }  \sum_{x_{i_{0}}y_{j_{0}}<q\leqslant Q}q^{\epsilon} \]
Puis en remarquant que $ \sum_{\cc\neq \0}\frac{1}{|\cc|^{n+1+\nu} } =O(1) $, que $ \sum_{x_{i_{0}}y_{j_{0}}<q\leqslant Q}q^{\epsilon}\ll Q^{1+\epsilon}= B^{\frac{1}{2}+\frac{\epsilon}{2}} $, et que $ \sum_{x_{i_{0}},y_{j_{0}}}\poids(x_{i_{0}},y_{j_{0}})x_{i_{0}}y_{j_{0}}\ll B^{\frac{4}{3}}\log(B) $ (rappelons que sur $ \supp\poids(x_{i_{0}},y_{j_{0}}) $, on a $ x_{i_{0}}y_{j_{0}}\leqslant B^{\frac{2}{3}} $), on obtient finalement \[  \eqref{majosimple} \ll B^{\alpha} \] avec \[ \alpha=\frac{n}{2}-\frac{1}{2}+\frac{1}{2}+\frac{4}{3}+\epsilon=\frac{n}{2}+\frac{4}{3}+\epsilon, \] qui est du type $ O(B^{n}) $ pour tout $ n\geqslant 3 $, et donc la somme sur $ q>x_{i_{0}}y_{j_{0}} $ donne un terme n\'egligeable par rapport au terme principal qui, rappelons-le, sera du type $ O(B^{n}\log(B)^{2}) $. On pourra donc restreindre la somme sur $ q $ aux entiers $ q\leqslant x_{i_{0}}y_{j_{0}} $.\\

Pour simplifier les notations, on notera \`a partir d'ici \nomenclature{$ \ww $}{\eqref{ww} $ =\frac{P\cc}{q}=\frac{B\cc}{x_{i_{0}}y_{j_{0}}q} $}\begin{equation}\label{ww} \ww=\frac{P\cc}{q}=\frac{B\cc}{x_{i_{0}}y_{j_{0}}q}, \end{equation} de sorte que $ \phi(\uu,t)=t\frac{\xx.\yy.\uu}{x_{i_{0}}y_{j_{0}}}-\ww.\uu $. D'apr\`es la formule \eqref{tpetit} appliqu\'ee \`a $ M=N+1 $, et la majoration $ |p(t)| \ll \frac{Q}{q} $, on a  \begin{equation} \int_{|t|\leqslant\frac{|\ww|}{3}}p(t)\int_{\RR^{n+1}}\psi(\uu)e(\phi(\uu,t))d\uu dt \ll \frac{Q}{q}|\ww|^{-N}=\frac{Q}{q}\left(\frac{q}{P|\cc|}\right)^{N}, \end{equation} pour tout $ N\geqslant 0 $. Donc en prenant $ N=n+2 $, on remarque que cette partie de l'int\'egrale $ I(\cc) $ fournit un terme d'erreur :

\begin{multline}\label{integ1} c_{Q}B^{n}\sum_{x_{i_{0}},y_{j_{0}}}\frac{\poids(x_{i_{0}},y_{j_{0}})}{x_{i_{0}}^{n+1}y_{j_{0}}^{n+1}}\sum_{\substack{(x_{i})_{i\neq i_{0}} \\ (y_{j})_{j\neq j_{0}} }}\poids(\xx,\yy)\sum_{\cc\neq \0} \\ \sum_{q\leqslant Q}q^{-n-1}|S_{q}(\xx,\yy)(\cc)\frac{Q}{q}|\left(\frac{q}{P|\cc|}\right)^{n+2} \\ =c_{Q}B^{-\frac{3}{2}} \sum_{x_{i_{0}},y_{j_{0}}}\frac{\poids(x_{i_{0}},y_{j_{0}})}{x_{i_{0}}^{-1}y_{j_{0}}^{-1}}\sum_{\cc\neq \0}\frac{1}{|\cc|^{n+2} } \\ \sum_{q\leqslant Q}\sum_{\substack{ (x_{i})_{i\neq i_{0}} \\ (y_{j})_{j\neq j_{0}}}}\poids(\xx,\yy)|S_{q}(\xx,\yy)(\cc)|. \end{multline}

Or on remarque que si $ q $ est fix\'e, suppos\'e inf\'erieur \`a $ x_{i_{0}} $ et $ y_{j_{0}} $ dans un premier temps, on a que 
\begin{multline}\label{somplus}
\sum_{\substack{ (x_{i})_{i\neq i_{0}} \\ (y_{j})_{j\neq j_{0}}}}\poids(\xx,\yy)|S_{q}(\xx,\yy)(\cc)| \\ =\sum_{\substack{(a_{i})_{i\neq i_{0}}\in (\ZZ/q\ZZ)^{n} \\ (a_{j}')_{j\neq j_{0}}\in (\ZZ/q\ZZ)^{n} }}|S_{q}(\aa,\aa')(\cc)|\sum_{\substack{ x_{i}\equiv a_{i} (q), i\neq i_{0}\\ y_{j}\equiv a_{j}'(q), j\neq j_{0}}}\poids(\xx,\yy) 
\end{multline}
o\`u $ \aa $ (resp. $ \aa' $) d\'esigne le vecteur de $ (\ZZ/q\ZZ)^{n+1} $ dont les coordonn\'ees sont les $ a_{i} $ pour $ i\neq i_{0} $ (resp. $ a_{j}' $ pour $ j\neq j_{0} $) et $ x_{i_{0}}\mod q $ pour $ i=i_{0} $ (resp. $ y_{j_{0}}\mod q $ pour $ j=j_{0} $). En consid\'erant la majoration \[ \sum_{\substack{ x_{i}\equiv a_{i} (q), i\neq i_{0}\\ y_{j}\equiv a_{j}'(q), j\neq j_{0}} }\poids(\xx,\yy) \ll \frac{x_{i_{0}}^{n}y_{j_{0}}^{n}}{q^{2n}} \] due \`a la d\'efinition de $ \poids(\xx,\yy) $, on se ram\`ene au calcul de \begin{equation}\label{somplusplus} \frac{x_{i_{0}}^{n}y_{j_{0}}^{n}}{q^{2n}}\sum_{\substack{(a_{i})_{i\neq i_{0}}\in (\ZZ/q\ZZ)^{n} \\ (a_{j}')_{j\neq j_{0}}\in (\ZZ/q\ZZ)^{n} }}|S_{q}(\aa,\aa')(\cc)| \end{equation} pour tout $ q $ fix\'e. Rappelons avant tout que \begin{align*} S_{q}(\aa,\aa')(\cc) & =\sum_{d\in (\ZZ/q\ZZ)^{\ast}}\sum_{\bb \in (\ZZ/q\ZZ)^{n+1}}e_{q}(d\aa.\aa'.\bb+\cc.\bb) \\ & =\sum_{d\in (\ZZ/q\ZZ)^{\ast}}\prod_{k=0}^{n}\left(\sum_{b\in \ZZ/q\ZZ}e_{q}(da_{k}a_{k}'b+c_{k}b)\right).  \end{align*}

Or on a : \begin{equation}\label{equivalence} \sum_{b\in \ZZ/q\ZZ}e_{q}(da_{k}a_{k}'b+c_{k}b) =\left\{\begin{array}{lcr}
q & \mbox{si} & da_{k}a_{k}'+c_{k} \equiv 0 (q) \\ 
0 & \mbox{sinon}. & 
\end{array}\right. \end{equation}

Pour $ k=i_{0} $ ou $ j_{0} $ on majore trivialement  $ \sum_{b\in \ZZ/q\ZZ}e_{q}(da_{k}a_{k}'b+c_{k}b) $ par $ q $ et d'apr\`es \eqref{equivalence}, on a (dans le cas o\`u $ i_{0}\neq j_{0} $) \begin{multline}\label{somme} \sum_{\substack{(a_{i})_{i\neq i_{0}}\in (\ZZ/q\ZZ)^{n}  \\ (a_{j}')_{j\neq j_{0}}\in (\ZZ/q\ZZ)^{n} }}|S_{q}(\aa,\aa')(\cc)| \\ \ll q^{4}\sum_{d\in (\ZZ/q\ZZ)^{\ast}}\prod_{k\neq i_{0},j_{0}}q.\card \left\{ (a_{k},a_{k}')\in (\ZZ/q\ZZ)^{2} \; | \; a_{k}a_{k}' \equiv -d^{-1}c_{k} (q) \right\} \\ =q^{n+3}\sum_{d\in (\ZZ/q\ZZ)^{\ast}}\prod_{k\neq i_{0},j_{0}}\card \left\{ (a_{k},a_{k}')\in (\ZZ/q\ZZ)^{2} \; | \; a_{k}a_{k}' \equiv -d^{-1}c_{k} (q) \right\}. \end{multline}  

Nous allons donc chercher \`a \'evaluer \[ A_{k}(q)=\card \left\{ (a_{k},a_{k}')\in (\ZZ/q\ZZ)^{2} \; | \; a_{k}a_{k}' \equiv d^{-1}c_{k} (q) \right\}. \] Notons $ q=p_{1}^{e_{1}}...p_{r}^{e_{r}} $ la d\'ecomposition de $ q $ en produit de facteurs premiers. On a par le th\'eor\`eme chinois : \[ A_{k}(q)=\prod_{l=1}^{r}A_{k}(p_{l}^{e_{l}}). \]  

On \'evalue alors $ A_{k}(p^{e}) $ pour tout $ p $ premier et $ e\in \NN $. Supposons que $ v_{p}(a_{k}')=f_{k} $ (avec $ 0\leqslant f_{k} \leqslant e) $. Alors on a que $ a_{k}'=m_{k}p^{f_{k}} $ avec $ m_{k}\in (\ZZ/p^{e}\ZZ)^{\ast} $. Donc, si $ a_{k}a_{k}' \equiv -d^{-1}c_{k} (p^{e}) $ alors $ a_{k}p^{f_{k}}\equiv -m_{k}^{-1}d^{-1}c_{k} (p^{e}) $. On a alors au plus $ p^{f_{k}} $ valeurs possibles pour $ a_{k} $ : en effet $ a_{k}p^{f_{k}}\equiv -m_{k}^{-1}d^{-1}c_{k} (p^{e}) $ admet une solution lorsque $ c_{k}\mod q =c_{k}'p^{f_{k}} $ (pour un certain $ c_{k}' \in \ZZ/p^{e}\ZZ $), et on aura alors $ (a_{k}+m_{k}^{-1}d^{-1}c_{k}')p^{f_{k}}\equiv 0 (p^{e}) $, donc $ p^{e-f_{k}} $ divise $ a_{k}-m_{k}^{-1}d^{-1}c_{k}'$, et on a donc au plus $ p^{f_{k}} $ valeurs possibles pour $ a_{k}\in \ZZ/p^{e}\ZZ $. D'autre part, puisque dans $ \ZZ/p^{e}\ZZ $ il y a $ p^{e-f_{k}} $ \'el\'ements $ a_{k}' $ tels que $ v_{p}(a_{k}')\geqslant f_{k} $, on conclut que :\[ A_{k}(p^{l})\ll\sum_{f_{k}=0}^{e}p^{e-f_{k}}p^{f_{k}}=(e+1)p^{e}. \]
Par cons\'equent, on a que \[ A_{k}(q)=\prod_{l=1}^{r}A_{k}(p_{l}^{e_{l}}) \ll \underbrace{\prod_{l=1}^{r}p_{l}^{e_{l}}}_{=q}\prod_{l=1}^{r}(e_{l}+1). \] Or $ \prod_{l=1}^{r}(e_{l}+1) $ est le nombre de diviseurs de $ q $ qui est du type $ O(q^{\epsilon}) $ pour tout $ \epsilon>0 $. On a donc finalement $ A_{k}(q)=O(q^{1+\epsilon}) $. En utilisant ce r\'esultat dans \eqref{somme}, on obtient le lemme ci-dessous :

\begin{lemma}\label{lemmemoi} On a pour tout $ \epsilon>0 $ :
 \[ \sum_{\substack{(a_{i})_{i\neq i_{0}}\in (\ZZ/q\ZZ)^{n}  \\ (a_{j}')_{j\neq j_{0}}\in (\ZZ/q\ZZ)^{n} }}|S_{q}(\aa,\aa')(\cc)| \ll q^{3+2n+\epsilon}. \]
\end{lemma}
\begin{rem}
ce r\'esultat est vrai m\^eme lorsque $ i_{0}=j_{0} $ : en effet, dans ce cas, on majore trivialement la somme $ \sum_{b\in \ZZ/q\ZZ}e_{q}(dx_{i_{0}}y_{i_{0}}b+c_{i_{0}}b) $ par $ q $ et on obtient comme dans \eqref{somme} : \begin{multline*} \sum_{\substack{ (a_{i})_{i\neq i_{0}}\in (\ZZ/q\ZZ)^{n}  \\ (a_{j}')_{j\neq j_{0}}\in (\ZZ/q\ZZ)^{n}}}|S_{q}(\aa,\aa')(\cc)| \\   \ll q^{n+1}\sum_{d\in (\ZZ/q\ZZ)^{\ast}}\prod_{k\neq i_{0}}\card \left\{ (a_{k},a_{k}')\in (\ZZ/q\ZZ)^{2} \; | \; a_{k}a_{k}' \equiv -d^{-1}c_{k} (q) \right\} \\  = q^{n+1}\sum_{d\in (\ZZ/q\ZZ)^{\ast}}\prod_{k\neq i_{0}}A_{k}(q) \\   \ll q^{n+1}\sum_{d\in (\ZZ/q\ZZ)^{\ast}}q^{n+\epsilon} \ll q^{2+2n+\epsilon}.\end{multline*}
\end{rem}

En appliquant ce lemme ainsi que \eqref{somplusplus} dans la formule \eqref{integ1}, on obtient \begin{multline*} \eqref{integ1} \ll c_{Q}B^{-\frac{3}{2}}\sum_{x_{i_{0}},y_{j_{0}}}\frac{\poids(x_{i_{0}},y_{j_{0}})}{x_{i_{0}}^{-1}y_{j_{0}}^{-1}}\sum_{\cc\neq \0}\frac{1}{|\cc|^{n+2} } \sum_{q\ll Q}\frac{x_{i}^{n}y_{j}^{n}}{q^{2n}}q^{3+2n+\epsilon} \\ = c_{Q}B^{-\frac{3}{2}}\sum_{x_{i_{0}},y_{j_{0}}}\poids(x_{i_{0}},y_{j_{0}})x_{i_{0}}^{n+1}y_{j_{0}}^{n+1}\sum_{\cc\neq \0}\frac{1}{|\cc|^{n+2} } \sum_{q\leqslant Q}q^{3+\epsilon} \\ \ll  B^{\frac{1}{2}+\frac{\epsilon}{2}}\sum_{x_{i_{0}},y_{j_{0}}}\poids(x_{i_{0}},y_{j_{0}})x_{i_{0}}^{n+1}y_{j_{0}}^{n+1} \end{multline*}
(en majorant $ \sum_{q\leqslant Q}q^{3+\epsilon} $ par $ Q^{4+\epsilon}=B^{2+\frac{\epsilon}{2}} $). Puis en majorant \linebreak  $ \sum_{x_{i_{0}},y_{j_{0}}}\poids(x_{i_{0}},y_{j_{0}})x_{i_{0}}^{n+1}y_{j_{0}}^{n+1} $ par $ B^{\frac{2}{3}(n+2)}\log(B) $ (car $ x_{i_{0}}y_{j_{0}}\leqslant B^{\frac{2}{3}} $ sur \linebreak $ \supp\poids(x_{i_{0}},y_{j_{0}}) $), on trouve finalement que la partie $ |t|\leqslant \frac{|\ww|}{3} $ de l'int\'egrale $ I(\cc) $ fournit un terme d'erreur du type $ O(B^{\frac{2}{3}n+\frac{11}{6}+\frac{\epsilon}{2}}\log(B)) $ (en restreignant la somme \eqref{termerreur2} aux $ q\leqslant \max\{x_{i_{0}},y_{j_{0}}\} $) qui est bien un $ O(B^{n}) $ pour tout $ n\geqslant 6 $ et $ \epsilon $ assez petit. \\

Lorsque $ q $ est tel que $ x_{i_{0}}\leqslant q $ et $ y_{j_{0}}>q $, on a que : 
\[ \sum_{\substack{(x_{i})_{i\neq i_{0}} \\ (y_{j})_{j\neq j_{0}} }}\poids(\xx,\yy)|S_{q}(\xx,\yy)(\cc)|=\sum_{\substack{(x_{i})_{i\neq i_{0}} \\ (a_{j}')_{j\neq j_{0}}\in (\ZZ/q\ZZ)^{n} }}\poids(\xx)|S_{q}(\xx,\aa')(\cc)|\sum_{y_{j}\equiv a_{j}'(q)}\poids(\yy), \] o\`u $ \aa'\in (\ZZ/q\ZZ)^{n+1} $ de $ j_{0}^{\grave{e}me} $ terme \'egal \`a $ y_{j_{0}} \mod q $. On majore dans ce cas la deuxi\`eme somme par $ \frac{y_{j_{0}}^{n}}{q^{n}} $. Il reste alors \`a estimer \[ \sum_{\substack{(x_{i})_{i\neq i_{0}} \\ (a_{j}')_{j\neq j_{0}}\in (\ZZ/q\ZZ)^{n}  }}\poids(\xx)|S_{q}(\xx,\aa')(\cc)|. \]

Pour cela, on r\'e\'ecrit cette somme sous la forme : \begin{equation}\label{deuxeq} \sum_{d\in (\ZZ/q\ZZ)^{\ast}}\prod_{k=0}^{n}\sum_{\substack{x_{k} \in \supp(\poids(x_{k})) \\ a_{k}'\in \ZZ/q\ZZ}}\poids(x_{k})\left|\sum_{b\in \ZZ/q\ZZ }e_{q}(dx_{k}a_{k}'b+c_{k}b)\right|. \end{equation}
En majorant le facteur correspondant \`a $ k=i_{0} $ par $ q^{2} $ (la variable $ x_{k} $ \'etant alors fix\'ee \'egale \`a $ x_{i_{0}} $) et le facteur $ k=j_{0} $ par $ qx_{i_{0}} $ ($ a_{k}' $ \'etant alors fix\'e \'egal \`a $ y_{j_{0}} $), on trouve (pour le cas o\`u $ i_{0}\neq j_{0} $) que : 
\begin{multline*} \eqref{deuxeq} \ll q^{n+2}x_{i_{0}} \sum_{d\in (\ZZ/q\ZZ)^{\ast}}\prod_{k\neq i_{0},j_{0}} \card\{ (x_{k},a_{k}'
)\in \supp(\poids(x_{k}))\times \ZZ/q\ZZ \\ \; | \; x_{k}a_{k}'\equiv -d^{-1}c_{k} (q)\}.  \end{multline*}
On souhaite donc \'evaluer \[ B_{k}(q)=\card\{ (x_{k},a_{k}'
)\in \supp(\poids(x_{k}))\times \ZZ/q\ZZ \\ \; | \; x_{k}a_{k}'\equiv -d^{-1}c_{k} (q) \}. \]
On consid\`ere alors la d\'ecomposition $ q=p_{1}^{e_{1}}...p_{r}^{e_{r}} $ de $ q $ en produit de facteurs premiers. Si l'on fixe $ x_{k}\in \supp(\poids(x_{k})) $, et si $ f_{l}=\min\{v_{p_{l}}(x_{k}),e_{l}\} \in \{0,1,...,e_{l}\} $, pour tout $ l \in \{ 1,...,r\} $ alors l'\'equation $ x_{k}a_{k}'\equiv -d^{-1}c_{k} (q) $ admet au plus $ p_{1}^{f_{1}}...p_{r}^{f_{r}} $ solutions $ a_{k}' \in \ZZ/q\ZZ $ (d'apr\`es le th\'eor\`eme des restes chinois). Par ailleurs, si $ f_{l}=\min\{v_{p_{l}}(x_{k}),e_{l}\} $ pour tout $ l $, on a alors $ p_{1}^{f_{1}}...p_{r}^{f_{r}}|x_{k} $. Donc pour chaque $ r $-uplet $ (f_{1},...,f_{r})\in \{1,...,e_{1}\}\times...\times  \{1,...,e_{r}\} $ on aura au plus $ \frac{x_{i_{0}}}{ p_{1}^{f_{1}}...p_{r}^{f_{r}}} $ entiers $ x_{k} $ v\'erifiant la condition $ f_{l}=\min\{v_{p_{l}}(x_{k}),e_{l}\} $. On peut donc donner l'estimation de $ B_{k}(q) $ ci-dessous :

 \begin{align*} B_{k}(q) & \ll \sum_{\substack{f_{l}\leqslant e_{l} \\ l\in \{ 1,...,r\} }}\sum_{\substack{x_{k}\in \supp(\poids(x_{k}))\\ p_{1}^{f_{1}}...p_{r}^{f_{r}}|x_{k} }} p_{1}^{f_{1}}...p_{r}^{f_{r}} \\ & \ll  \sum_{\substack{f_{l}\leqslant e_{l} \\ l\in  \{ 1,...,r\}}}\frac{x_{i_{0}}}{ p_{1}^{f_{1}}...p_{r}^{f_{r}}}.p_{1}^{f_{1}}...p_{r}^{f_{r}} \\ & = x_{i_{0}}(e_{1}+1)...(e_{r}+1) \ll q^{\epsilon_{0}}x_{i_{0}}. \end{align*}
 
 En reprenant la majoration \eqref{deuxeq}, on trouve : \begin{equation}\label{inegmoi}  \sum_{\substack{(x_{i})_{i\neq i_{0}} \\ (a_{j}')_{j\neq j_{0}}\in (\ZZ/q\ZZ)^{n}  }}\poids(\xx)|S_{q}(\xx,\aa')(\cc)| \ll q^{n+2}x_{i_{0}}\sum_{d \in (\ZZ/q\ZZ)^{\ast}} x_{i_{0}}^{n-1}q^{\epsilon} \ll q^{n+3+\epsilon}x_{i_{0}}^{n}, \end{equation} avec $ \epsilon=(n-1)\epsilon_{0} >0 $ arbitrairement petit. On obtient donc finalement la majoration (pour $ i_{0}\neq j_{0} $):  

\begin{equation}\label{lemmemoi2} \sum_{\substack{(x_{i})_{i\neq i_{0}} \\ (y_{j})_{j\neq j_{0}} }}\poids(\xx,\yy)|S_{q}(\xx,\yy)(\cc)|\ll q^{3+\epsilon}x_{i_{0}}^{n}y_{j_{0}}^{n} \end{equation} lorsque $ q\geqslant x_{i_{0}} $ et $ q<y_{j_{0}} $. \begin{rem} Il n'est pas difficile de voir, en effectuant les m\^emes calculs, que cette majoration reste valable dans le cas $ i_{0}=j_{0} $. En fait pour $ k\neq i_{0} $ on a la m\^eme majoration des $ B_{k}(q) $, et pour $ k=i_{0}=j_{0} $, on prend la majoration triviale $ q $, et on trouve alors \[ \sum_{\substack{(x_{i})_{i\neq i_{0}} \\ (y_{j})_{j\neq i_{0}} }}\poids(\xx,\yy)|S_{q}(\xx,\yy)(\cc)|\ll q^{2+\epsilon}x_{i_{0}}^{n}y_{j_{0}}^{n}. \]\end{rem}

Remarquons que le r\'esultat \eqref{inegmoi} et le lemme \ref{lemmemoi} permettent d'\'etablir une majoration vraie pour toute valeur de $ x_{i_{0}} $ :
\begin{lemma}\label{lemmegeneral}
Pour tout $ \epsilon>0 $ :
 \[  \sum_{\substack{(x_{i})_{i\neq i_{0}} \\ (a_{j}')_{j\neq j_{0}}\in (\ZZ/q\ZZ)^{n} }}\poids(\xx)|S_{q}(\xx,\aa')(\cc)|  \ll q^{n+3+\epsilon}x_{i_{0}}^{n}.\]
\end{lemma}

\begin{rem}
Par les m\^emes calculs, en intervertissant les r\^oles de $ \xx $ et $ \yy $, on obtient exactement le m\^eme type de majoration pour le cas o\`u $ x_{i_{0}}>q $ et $ y_{j_{0}}\leqslant q $. \end{rem}
\begin{rem}\label{remlemme} On a suppos\'e ici que $ \cc\neq \0 $, mais il est clair (voir la d\'emonstration) que le r\'esultat reste vrai pour $ \cc=\0 $. Cela nous sera utile pour traiter le terme principal. \end{rem}

La contribution de la somme sur les $ q<y_{j_{0}} $ et $ q\geqslant x_{i_{0}} $ est donc :   

\begin{multline*}c_{Q}B^{-\frac{3}{2}}\sum_{x_{i_{0}},y_{j_{0}}}\frac{\poids(x_{i_{0}},y_{j_{0}})}{x_{i_{0}}^{-1}y_{j_{0}}^{-1}}\sum_{\cc\neq \0}\frac{1}{|\cc|^{n+2} } \sum_{\substack{  q\leqslant Q \\ x_{i_{0}}\leqslant q<y_{j_{0}}}}x_{i_{0}}^{n}y_{j_{0}}^{n}q^{3+\epsilon} \\ = c_{Q}B^{-\frac{3}{2}}\sum_{x_{i_{0}},y_{j_{0}}}\poids(x_{i_{0}},y_{j_{0}})x_{i_{0}}^{n+1}y_{j_{0}}^{n+1}\sum_{\cc\neq \0}\frac{1}{|\cc|^{n+2} } \sum_{ \substack{q\leqslant Q \\ x_{i_{0}}\leqslant q<y_{j_{0}}}}q^{3+\epsilon} \end{multline*}
(voir \eqref{integ1}) qui est du type $ O(B^{n}) $ pour tout $ n \geqslant 6 $ d'apr\`es les  calculs effectu\'es pour le cas $ q\geqslant x_{i_{0}},y_{j_{0}} $. On obtient alors le m\^eme r\'esultat pour la partie $ x_{i_{0}}>q $ et $ y_{j_{0}}\leqslant q $.\\

Traitons enfin le cas o\`u $ x_{i_{0}},y_{j_{0}}<q $. On consid\`ere \ a nouveau l'\'egalit\'e :

\[ \sum_{\substack{(x_{i})_{i\neq i_{0}} \\ (y_{j})_{j\neq j_{0}}  }}\poids(\xx,\yy)|S_{q}(\xx,\yy)(\cc)|=\sum_{\substack{(x_{i})_{i\neq i_{0}} \\ (a_{j}')_{j\neq j_{0}}\in (\ZZ/q\ZZ)^{n} }}\poids(\xx)|S_{q}(\xx,\aa')(\cc)|\sum_{y_{k}\equiv a_{k}'(q)}\poids(\yy), \]  et on majore la deuxi\`eme somme par $ 1 $ (car $ y_{j_{0}}<q $). On reprend alors la majoration \eqref{inegmoi} (bas\'ee sur la formule \eqref{deuxeq}) pour la premi\`ere somme, et on trouve alors l'estimation : \begin{equation}
\sum_{\substack{(x_{i})_{i\neq i_{0}} \\ (y_{j})_{j\neq j_{0}} }}\poids(\xx,\yy)|S_{q}(\xx,\yy)(\cc)|\ll q^{n+3+\epsilon}x_{i_{0}}^{n}. 
\end{equation}

En appliquant ce r\'esultat \`a la formule \eqref{integ1}, on trouve : \begin{multline*}c_{Q}B^{-\frac{3}{2}}\sum_{x_{i_{0}},y_{j_{0}}}\frac{\poids(x_{i_{0}},y_{j_{0}})}{x_{i_{0}}^{-1}y_{j_{0}}^{-1}}\sum_{\cc\neq \0}\frac{1}{|\cc|^{n+2} } \sum_{\substack{ y_{j_{0}}< q\ll Q \\ x_{i_{0}}< q }}x_{i_{0}}^{n}q^{n+3+\epsilon} \\ \ll B^{-\frac{3}{2}}\sum_{x_{i_{0}},y_{j_{0}}}\poids(x_{i_{0}},y_{j_{0}})x_{i_{0}}^{n+1}y_{j_{0}} \sum_{\substack{ y_{j_{0}}< q\ll Q \\ x_{i_{0}}< q }}q^{n+3+\epsilon}  . \end{multline*}

On effectue alors les majorations \[ \sum_{q \ll Q}q^{n+3+\epsilon}\ll Q^{n+4+\epsilon}\ll B^{\frac{n}{2}+2+\frac{\epsilon}{2}},  \] et \[ \sum_{x_{i_{0}},y_{j_{0}}}\poids(x_{i_{0}},y_{j_{0}})x_{i_{0}}^{n+1}y_{j_{0}} \ll B^{\frac{4}{3}}\sum_{x_{i_{0}}}\poids(x_{i_{0}})x_{i_{0}}^{n} \ll B^{\frac{n+1}{3}+\frac{4}{3}}, \] (en rappelant que $ x_{i_{0}}y_{j_{0}}\leqslant B^{\frac{2}{3}} $ et $ x_{i_{0}}\leqslant B^{\frac{1}{3}} $ d'\`apr\`es \eqref{(4)}). On obtient finalement un terme d'erreur du type $ O(B^{\alpha}) $ avec \[ \alpha=\frac{n}{2}+2+\frac{n+1}{3}+\frac{1}{2}+\frac{4}{3}+\frac{\epsilon}{2}=\frac{5n}{6}+\frac{13}{6}+\frac{\epsilon}{2}, \] qui est alors bien un $ O(B^{n}) $ pour tout $ n\geqslant 14 $.\\

 La contribution totale de la somme $ \eqref{integ1} $ est donc du type $ O(B^{n}) $ pour tout $ n\geqslant 21 $\label{page1}. \\

Nous allons \`a pr\'esent \'etudier la contribution de la partie $ |t|\geqslant\frac{|\ww|}{3} $ de l'int\'egrale $ I(\cc) $ que l'on notera $ I'(\cc) $.

\begin{rem}\label{restriction}
 Rappelons que d'apr\`es les remarques formul\'ees au d\'ebut de cette section (cf calculs \eqref{majosimple}, \eqref{qpetit}), on peut restreindre la somme sur $ q $ aux entiers $ q\leqslant x_{i_{0}}y_{j_{0}} $, et on a donc $ x_{i_{0}}\geqslant \sqrt{q} $ ou $ y_{j_{0}}\geqslant \sqrt{q} $. Par sym\'etrie, on restreint la somme aux entiers $ x_{i_{0}},y_{j_{0}} $ tels que $ x_{i_{0}}\leqslant y_{j_{0}} $, et on a alors $ y_{j_{0}}\geqslant \sqrt{q} $.
 \end{rem}
 
  Nous allons donc chercher \`a \'evaluer (en reprenant la majoration \eqref{termerreur2}) : \begin{equation}\label{nouveq} c_{Q}B^{n}\sum_{x_{i_{0}},y_{j_{0}}}\frac{\poids(x_{i_{0}},y_{j_{0}})}{x_{i_{0}}^{n+1}y_{j_{0}}^{n+1}}\sum_{\substack{(x_{i})_{i\neq i_{0}} \\ (y_{j})_{j\neq j_{0}}} }\poids(\xx,\yy)\sum_{\substack{\cc\neq \0 \\ |\cc|\ll B^{\frac{1}{6}+\delta}}}\sum_{q\ll Q}q^{-n-1}S_{q}(\xx,\yy)(\cc)I'(\cc) \end{equation}(rappelons que l'on ne s'int\'eresse qu'aux $ \cc\in \ZZ^{n+1} $ tels que $ |\cc|\ll B^{\frac{1}{6}+\delta} $, avec $ \delta $ arbitrairement petit (cf. \eqref{paramc}).\\

 On introduit un nouveau param\`etre $ R\geqslant 1 $ que l'on fixera par la suite. Nous allons \`a pr\'esent r\'eexprimer l'int\'egrale  $ I'(\cc) $ \`a l'aide du lemme suivant (voir \cite[Lemme.2]{HB1}) :

\begin{lemma}
Soit $w : \RR^{n+1} \ra \RR $ une fonction poids en $ n+1 $ variables. Alors pour tout $ \delta \in \mathopen] 0,1] $, il existe une fonction poids $ w_{\delta} : \RR^{n+1} \times \RR^{n+1} \ra \RR $ telle que : \[  w(\xx)=\delta^{-n-1}\int_{\RR^{n+1}}w_{\delta}\left( \frac{\xx-\yy}{\delta},\yy\right)d\yy. \]
On a de plus, si $ F(\xx)= w_{\delta}\left( \frac{\xx-\yy}{\delta},\yy\right) $, alors $ \supp(F)\subset \supp(w) $ quel que soit $ \yy $. 
\end{lemma}

On peut appliquer ce lemme \`a la fonction poids $ \psi(\uu) $. On a alors \[ \psi(\uu)=\delta^{-n-1}\int_{\RR^{n+1}}w_{\delta}\left(\frac{\uu-\vv}{\delta},\vv\right)d\vv \] o\`u $ 0< \delta <1 $ et $ w_{\delta} $ est une nouvelle fonction poids avec un support en $ \vv $ de taille $ O(1) $. Par un changement de variable \nomenclature{$ \uu=\vv+\delta\ss $}{ changement de variable}$ \uu=\vv+\delta\ss $, on trouve \begin{multline*} |I'(\cc)|=\left| \int_{|t|\geqslant\frac{|\ww|}{3}}p(t)\int_{\RR^{n+1}}\psi(\uu)e(\phi(\uu,t))d\uu dt \right| \\ \ll \int_{|t|\geqslant\frac{|\ww|}{3}}|p(t)|\int_{\RR^{n+1}}\left| \int_{\RR^{n+1}}w_{\delta}(\ss,\vv)e(\phi(\vv+\delta\ss,t))d\ss \right| d\vv dt . \end{multline*}

On fixe dans un premier temps les variables $ \vv $ et $ t $. Nous allons d\'emontrer que pour la plupart des valeurs de $ \yy \in \ZZ^{n+1} $, l'int\'egrale \[ \int_{\RR^{n+1}}w_{\delta}(\ss,\vv)e(\phi(\vv+\delta\ss,t))d\ss  \] est n\'egligeable. On pose \nomenclature{$ f(\ss) $}{ $ =\phi(\vv+\delta\ss,t) $}\[ f(\ss)=\phi(\vv+\delta\ss,t)=  t\frac{\xx.\yy.(\vv+\delta\ss)}{x_{i_{0}}y_{j_{0}}}-\ww. (\vv+\delta\ss). \]
Les d\'eriv\'ees partielles secondes de $ f $ sont nulles et on a par ailleurs 
\begin{equation}\label{f(s)} |\nabla f(\ss) |=|\nabla f(\0) |=\delta \max_{0\leqslant k \leqslant n} \left| t\frac{x_{k}y_{k}}{x_{i_{0}}y_{j_{0}}}- \frac{Pc_{k}}{q} \right| \end{equation} (la notation \og $  |.|  $\fg \ d\'esignant ici le max des coordonn\'ees).
D'apr\`es le lemme \ref{lemme10}, on a alors \[ \left|\int_{\RR^{n+1}}w_{\delta}(\ss,\vv)e(\phi(\vv+\delta\ss,t))d\ss\right| \ll R^{-N} \] lorsque $ |\nabla f(\0) |\geqslant R $, pour tout entier $ N\geqslant 0 $. A partir d'ici on fixe $ \delta=|\ww|^{-\frac{1}{2}} $. Pour un $ \xx \in \ZZ^{n+1} $ fix\'e, on dit que $ \yy \in \ZZ^{n+1} $ est un \og bon point \fg \ lorsque $ |\nabla f(\0) |\geqslant R $, c'est-\`a-dire, d'apr\`es \eqref{f(s)}, lorsque : \[ \left| t\frac{x_{k}y_{k}}{x_{i_{0}}y_{j_{0}}}- \frac{Pc_{k}}{q} \right| \geqslant R|\ww|^{\frac{1}{2}} \] pour un certain $ k\in \{ 0,1,...,n\} $ et pour tout $ |t|\geqslant \frac{|\ww|}{3} $.  Pour ces points en consid\'erant $ |p(t)|\ll \left(\frac{q}{Q}|t|\right)^{-2} $ (cf. \eqref{p(t)}) et en rappelant que le support de $ w_{\delta}(\ss,\vv) $ en $ \vv $ est de taille $ O(1) $ on a \[ I'(\cc) \ll \left(\frac{Q}{q}\right)^{2}|\ww|^{-1}R^{-N}=\frac{x_{i_{0}}y_{j_{0}}}{q|\cc|}R^{-N}. \] Les \og mauvais points \fg \ sont les $ \yy $ tels qu'il existe $ t $ tel que $ |t|\geqslant  \frac{|\ww|}{3} $ et tel que \[ \left| t\frac{x_{k}y_{k}}{x_{i_{0}}y_{j_{0}}}- \frac{Pc_{k}}{q} \right| \leqslant R|\ww|^{\frac{1}{2}} \] pour tout $ k $. Pour ces points on majore $ I'(\cc) $ par $ \frac{Q}{q}\left(\frac{x_{i_{0}}y_{j_{0}}}{\sqrt{B}|\cc|}\right)^{M} $ (\`a l'aide du lemme \ref{lemmemajo1} avec un entier $ M\geqslant 0 $ que nous pr\'eciserons). On a alors : \begin{multline}\label{(46)} \eqref{nouveq} \ll \underbrace{B^{n}\sum_{x_{i_{0}},y_{j_{0}}}\frac{\poids(x_{i_{0}},y_{j_{0}})}{x_{i_{0}}^{n}y_{j_{0}}^{n}}\sum_{\substack{\cc\neq \0 \\ |\cc|\ll B^{\frac{1}{6}+\delta}}}\frac{1}{|\cc|}\sum_{(x_{i})_{i\neq i_{0}}}\poids(\xx) \sum_{\substack{(y_{j})_{j\neq j_{0}} \\ \yy \;  \bon }} \poids(\yy)\sum_{q\ll Q}q^{-n-2}|S_{q}(\xx,\yy)(\cc)|R^{-N}}_{(\alpha)} \\ + \underbrace{B^{n+\frac{1}{2}}\sum_{x_{i_{0}},y_{j_{0}}}\frac{\poids(x_{i_{0}},y_{j_{0}})}{x_{i_{0}}^{n+1}y_{j_{0}}^{n+1}}\sum_{\substack{\cc\neq \0 \\ |\cc|\ll B^{\frac{1}{6}+\delta}}}\sum_{(x_{i})_{i\neq i_{0}}}\poids(\xx) \sum_{\substack{ (y_{j})_{j\neq j_{0}}\\ \yy \; \mauvais}} \poids(\yy)\sum_{q\ll Q}q^{-n-2}|S_{q}(\xx,\yy)(\cc)|\left(\frac{x_{i_{0}}y_{j_{0}}}{\sqrt{B}|\cc|}\right)^{M}}_{(\beta)}. \end{multline}

Si l'on choisit $ R=B^{\lambda} $ (avec $ \lambda $ que nous fixerons ult\'erieurement), on remarque qu'en prenant $ N $ assez grand, la partie $ (\alpha) $ de la somme est n\'egligeable (rappelons que nous effectuons la somme sur les $ x_{i}\ll x_{i_{0}} $, $  y_{j} \ll y_{j_{0}} $ pour tous $ i,j $, avec $ x_{i_{0}}y_{j_{0}} \leqslant B^{\frac{2}{3}} $). On s'int\'eressera donc exclusivement \`a la somme $ (\beta) $ dans ce qui va suivre. On peut r\'e\'ecrire cette somme sous la forme :
\begin{multline}\label{mauvais points} B^{n+\frac{1}{2}}\sum_{x_{i_{0}},y_{j_{0}}}\frac{\poids(x_{i_{0}},y_{j_{0}})}{x_{i_{0}}^{n+1}y_{j_{0}}^{n+1}}\sum_{\substack{\cc\neq \0 \\ |\cc|\ll B^{\frac{1}{6}+\delta}}}\sum_{q\ll Q}q^{-n-2}\sum_{\substack{(x_{i})_{i\neq i_{0}}\; ,i\neq i_{0} \\ (a_{j}')_{j\neq j_{0}}\in (\ZZ/q\ZZ)^{n}  }}\poids(\xx)|S_{q}(\xx,\aa')(\cc)| \\ \sum_{ \substack{y_{j}\equiv a_{j}'(q) ,j\neq j_{0} \\ \yy \; \mauvais}} \poids(\yy)\left(\frac{x_{i_{0}}y_{j_{0}}}{\sqrt{B}|\cc|}\right)^{M}. \end{multline}

Nous allons \`a pr\'esent estimer, pour $ \xx $ et $ q $ fix\'es, la somme \[  \sum_{\substack{y_{j}\equiv a_{j}'(q) ,j\neq j_{0} \\ \yy \; \mauvais}} \poids(\yy) \ll \card \left\{ (y_{j})_{j\neq j_{0}}  \; |\; \yy \; \mauvais, \;  y_{j}\equiv a_{j}'(q) \; \forall \;  j\neq j_{0} \right\}.  \]

Rappelons que $ \yy $ est un mauvais point lorsque \[ \left| t\frac{x_{k}y_{k}}{x_{i_{0}}y_{j_{0}}}- w_{k}\right| \leqslant R|\ww|^{\frac{1}{2}} \] pour tout $ k $ et pour un certain $ t $ tel que $ |t|\geqslant\frac{|\ww|}{3} $, c'est-\`a-dire \[ \left| y_{k}- \frac{x_{i_{0}}y_{j_{0}}}{t}\frac{w_{k}}{x_{k}} \right| \leqslant R\frac{x_{i_{0}}y_{j_{0}}}{|t||x_{k}|}|\ww|^{\frac{1}{2}}  \]  
(rappelons que $ |x_{k}|\neq 0 $ pour tout $ k $ sur $ \supp\poids(\xx) $).\\

Lorsque $ \yy $ v\'erifie en plus que $ y_{k}\equiv a_{k}' (q) $ pour tout $ k\neq j_{0} $, On peut \'ecrire $ y_{k}=qm_{k} +a_{k}' $, avec $ m_{k}  \in \ZZ $, et la condition pour que $ \yy $ soit un mauvais point s'\'ecrit alors  
\[ \left| m_{k} + \frac{a_{k}'}{q}- \frac{x_{i_{0}}y_{j_{0}}}{t}\frac{w_{k}}{qx_{k}} \right| \leqslant R\frac{x_{i_{0}}y_{j_{0}}}{|t|q|x_{k}| }|\ww|^{\frac{1}{2}}.  \] 

On a donc \begin{multline*} \card \left\{ (y_{k})_{k\neq j_{0}}  \; |\; \yy \; \mauvais, \;  y_{k}\equiv a_{k}'(q), \; \forall \;  k\neq j_{0} \right\} \\  \leqslant \card \{ (m_{k})_{k\neq j_{0}}\in \ZZ^{n} \; | \;  \exists \;  |t|\geqslant \frac{|\ww|}{3},\\ \;  \forall k\neq j_{0} \; \left| m_{k} - (M_{k}+\frac{1}{t}N_{k}(\xx))\right|\leqslant \frac{R_{k}(\xx)}{|t|}  \} \end{multline*}

o\`u l'on a pos\'e \nomenclature{$ M_{k} $}{\eqref{Mk} $ =-\frac{a_{k}'}{q } $}\begin{equation}\label{Mk}
 M_{k}=-\frac{a_{k}'}{q },
\end{equation} \nomenclature{$ N_{k}(\xx) $}{\eqref{Nk} $ =\frac{x_{i_{0}}y_{j_{0}}w_{k}}{qx_{k}} $} \begin{equation}\label{Nk}
N_{k}(\xx)= \frac{x_{i_{0}}y_{j_{0}}w_{k}}{qx_{k}},
\end{equation} \nomenclature{$ R_{k}(\xx) $}{\eqref{Rk} $ R\frac{x_{i_{0}}y_{j_{0}}}{q|x_{k}| }|\ww|^{\frac{1}{2}} $} \begin{equation}\label{Rk}
R_{k}(\xx)=R\frac{x_{i_{0}}y_{j_{0}}}{q|x_{k}| }|\ww|^{\frac{1}{2}}
\end{equation} 

On a par ailleurs que l'ensemble \[ E= \{ (m_{k})_{k\neq j_{0}}\in \ZZ^{n} \; | \;  \exists \;  |t|\geqslant \frac{|\ww|}{3},\\ \; \left| m_{k} - (M_{k}+\frac{1}{t}N_{k}(\xx))\right|\leqslant \frac{R_{k}(\xx)}{|t|}, \;  \forall k\neq j_{0}  \} \] est inclus dans \[ \left\{ (M_{k})_{\substack{k\neq j_{0}}} + \frac{1}{t}( N_{k}(\xx) + [-R_{k}(\xx),R_{k}(\xx)])_{\substack{k\neq j_{0}}}, \; |t|\geqslant  \frac{|\ww|}{3} \right\}\subset \RR^{n}. \]Remarquons que pour tout $ |t|\geqslant  \frac{|\ww|}{3} $,
\begin{multline*}  \left(M_{k}\right)_{\substack{k\neq j_{0} }} + \frac{1}{t}\left( N_{k}(\xx) + [-R_{k}(\xx),R_{k}(\xx)]\right)_{\substack{k\neq j_{0} }} \\ \subset \left(M_{k}+\frac{1}{t}N_{k}(\xx)\right)_{\substack{k\neq j_{0} }} + \left( \frac{3}{|\ww|}[-R_{k}(\xx),R_{k}(\xx)]\right)_{\substack{k\neq j_{0} }}  \\ \subset  \left(M_{k}+\frac{1}{t}N_{k}(\xx)\right)_{\substack{k\neq j_{0}}} + \left(\left[-\max\left(\frac{3R_{k}(\xx)}{|\ww|},1\right),\max\left(\frac{3R_{k}(\xx)}{|\ww|},1\right) \right]\right)_{\substack{k\neq j_{0} }}=E(t).   \end{multline*}

 On a alors que \[ E\subset \bigcup_{ |t|\geqslant  \frac{|\ww|}{3}}E(t). \] Plus pr\'ecis\'ement, $ \card(E) $ peut \^etre major\'e par le nombre de points entiers contenus dans  le domaine $ \bigcup_{|t|\geqslant  \frac{|\ww|}{3}}E(t) $. Or, puisque $ \max\left(\frac{3R_{k}(\xx)}{|\ww|},1\right)\geqslant 1 $, ce nombre de points entiers peut \^etre major\'e par $ \Vol\left(\bigcup_{ |t|\geqslant  \frac{|\ww|}{3}}E(t)\right) $. On a d'une part que la longueur du segment \[ \bigcup_{|t|\geqslant  \frac{|\ww|}{3}}\left(M_{k}+\frac{1}{t}N_{k}(\xx)\right)_{\substack{k\neq j_{0} }} \] est du type \[ O\left(\frac{\max_{k\neq j_{0}}|N_{k}(\xx)|}{|\ww|}\right)=O\left(\frac{x_{i_{0}}y_{j_{0}}}{q}\right) \] d'apr\`es la d\'efinition de $ N_{k}(\xx) $, et d'autre part pour tout $ t $ : \begin{align*}  \Vol\left(E(t)\right) & = \Vol\left(  (M_{k}+ \frac{1}{t} N_{k}(\xx))_{\substack{k\neq j_{0} }}  + \left(\left[-\max\left(\frac{3R_{k}(\xx)}{|\ww|},1\right),\max\left(\frac{3R_{k}(\xx)}{|\ww|},1\right) \right]\right)_{\substack{k\neq j_{0} }} \right) \\ & = \Vol\left( \left(\left[-\max\left(\frac{3R_{k}(\xx)}{|\ww|},1\right),\max\left(\frac{3R_{k}(\xx)}{|\ww|},1\right) \right]\right)_{\substack{k\neq j_{0} }}  \right)\\ & \ll \prod_{\substack{k\neq j_{0} \\ k \notin \mathcal{I}(\xx)}}\frac{R_{k}(\xx)}{|\ww|}, \end{align*}
avec \begin{equation}\label{I(x)} \mathcal{I}(\xx)=\left\{k\neq j_{0} \; \left| \; \frac{3R_{k}(\xx)}{|\ww|}<1 \right.\right\}. \end{equation}

Par cons\'equent, on peut majorer le volume de $ \bigcup_{ |t|\geqslant  \frac{|\ww|}{3}}E(t) $ et donc le cardinal de $ E $ par (rappelons que l'on a suppos\'e $ q\leqslant x_{i_{0}}y_{j_{0}} $): \begin{align*} \frac{x_{i_{0}}y_{j_{0}}}{q}\prod_{\substack{k\neq j_{0} \\ k \notin  \mathcal{I}(\xx)}}\frac{R_{k}(\xx)}{|\ww|} \end{align*}

 Quitte \`a permuter les variables, on se ram\`ene au cas o\`u \[ \mathcal{I}(\xx)=\{0,1,...,r-1\}=I ,\] avec $ r<j_{0}\leqslant n $. On cherche donc \`a pr\'esent \`a \'evaluer (cf. \eqref{(46)} avec $ M=r $), pour tout $ r $ : \begin{multline}\label{nouvexpr}
B^{n+\frac{1}{2}}\sum_{x_{i_{0}}, y_{j_{0}}}\frac{\poids(x_{i_{0}},y_{j_{0}})}{x_{i_{0}}^{n+1}y_{j_{0}}^{n+1}}\sum_{\substack{\cc\neq \0 \\ |\cc|\ll B^{\frac{1}{6}+\delta}}}\sum_{q=1}^{\infty}q^{-n-2}\sum_{\substack{(x_{i})_{i\neq i_{0}}\\ I(\xx)=I}} \\  \sum_{(a_{j}')_{j\neq j_{0}} \in (\ZZ/q\ZZ)^{n}}\poids(\xx)|S_{q}(\xx,\aa')(\cc)|\card(E)\left(\frac{x_{i_{0}}y_{j_{0}}}{\sqrt{B}|\cc|}\right)^{r} \\ \ll B^{n+\frac{1}{2}}\sum_{x_{i_{0}}, y_{j_{0}}}\frac{\poids(x_{i_{0}},y_{j_{0}})}{x_{i_{0}}^{n+1}y_{j_{0}}^{n+1}}\sum_{\substack{\cc\neq \0 \\ |\cc|\ll B^{\frac{1}{6}+\delta}}}\sum_{q=1}^{\infty}q^{-n-2}\sum_{\substack{(x_{i})_{i\neq i_{0}}\\ I(\xx)=I}} \\  \sum_{(a_{j}')_{j\neq j_{0}} \in (\ZZ/q\ZZ)^{n}}\poids(\xx)|S_{q}(\xx,\aa')(\cc)|\frac{x_{i_{0}}y_{j_{0}}}{q}\prod_{\substack{k\neq j_{0} \\ k \notin \mathcal{I}(\xx)}}\frac{R_{k}(\xx)}{|\ww|}\left(\frac{x_{i_{0}}y_{j_{0}}}{\sqrt{B}|\cc|}\right)^{r}.
\end{multline}

Par le calcul, en rappelant l'expression de $ R_{k}(\xx) $, \eqref{Rk} et celle de $ |\ww| $, \eqref{ww} :

\begin{multline}\label{nouvnom} \frac{x_{i_{0}}y_{j_{0}}}{q}\prod_{\substack{k\neq j_{0} \\ k \notin I}}\frac{R_{k}(\xx)}{|\ww|}  =\frac{x_{i_{0}}y_{j_{0}}}{q}\prod_{\substack{k\neq j_{0} \\ k \notin I}}R\frac{x_{i_{0}}y_{j_{0}}}{q|x_{k}| }|\ww|^{-\frac{1}{2}} \\ =\frac{R^{n-r}x_{i_{0}}^{\frac{3(n-r)}{2}+1}y_{j_{0}}^{\frac{3(n-r)}{2}+1}}{q^{\frac{n-r}{2}+1}|\cc|^{\frac{n-r}{2}}B^{\frac{n-r}{2}}\prod_{\substack{k\neq j_{0} \\ k \notin I}}|x_{k}|}. \end{multline}

Remarquons \`a pr\'esent que dans le cas o\`u $ I $ n'est pas vide (i.e. $ r>0 $), alors, par d\'efinition de $ \mathcal{I}(\xx) $ (cf. \eqref{I(x)}), pour tout $ \xx $ tel que $ \mathcal{I}(\xx)=I $, il existe un certain $ k  $ tel que $ \frac{3R_{k}(\xx)}{|\ww|}\leqslant 1 $, et on a donc dans ce cas, en rappelant que $ \ww=\frac{P|\cc|}{q}=\frac{B|\cc|}{qx_{i_{0}}y_{j_{0}}} $ ) : \[ \frac{3R_{k}(\xx)}{|\ww|}\leqslant \frac{3Rx_{i_{0}}^{\frac{3}{2}}y_{j_{0}}^{\frac{3}{2}}}{B^{\frac{1}{2}}q^{\frac{1}{2}}|x_{k}||\cc|^{\frac{1}{2}}}\leqslant 1, \] et donc en utilisant la majoration triviale $ x_{k}\ll x_{i_{0}} $ : \[ y_{j_{0}}\ll\frac{q^{\frac{1}{3}}|\cc|^{\frac{1}{3}}B^{\frac{1}{3}}}{R^{\frac{2}{3}}x_{i_{0}}^{\frac{1}{3}}}, \]
ce qui va nous permettre de majorer le facteur $ \left(\frac{x_{i_{0}}y_{j_{0}}}{\sqrt{B}|\cc|}\right)^{r}  $ de la somme \eqref{nouvexpr} par \begin{equation}\label{nouvfact} \frac{x_{i_{0}}^{\frac{2r}{3}}q^{\frac{r}{3}}}{R^{\frac{2r}{3}}B^{\frac{r}{6}}|\cc|^{\frac{2r}{3}}}. \end{equation}
On obtient alors, avec \eqref{nouvfact} et \eqref{nouvnom} la majoration suivante de \eqref{nouvexpr}, (qui sera vraie m\^eme pour $ I $ vide, \'etant donn\'e que ce cas correspond \`a $ r=0 $, et donc $ \eqref{nouvfact}= \left(\frac{x_{i_{0}}y_{j_{0}}}{\sqrt{B}|\cc|}\right)^{r} =1 $) :

\begin{multline}\label{nouvexpr2}
\eqref{nouvexpr} \ll  B^{n+\frac{1}{2}}\sum_{x_{i_{0}}, y_{j_{0}}}\frac{\poids(x_{i_{0}},y_{j_{0}})}{x_{i_{0}}^{n+1}y_{j_{0}}^{n+1}}\sum_{\substack{\cc\neq \0 \\ |\cc|\ll B^{\frac{1}{6}+\delta}}}\sum_{q\ll Q}q^{-n-2}\sum_{\substack{(x_{i})_{i\neq i_{0}}\\ \mathcal{I}(\xx)=I}}\sum_{(a_{j}')_{j\neq j_{0}} \in (\ZZ/q\ZZ)^{n}} \\  \poids(\xx)|S_{q}(\xx,\aa')(\cc)| \left(\frac{R^{n-r}x_{i_{0}}^{\frac{3(n-r)}{2}+1}y_{j_{0}}^{\frac{3(n-r)}{2}+1}}{q^{\frac{n-r}{2}+1}|\cc|^{\frac{n-r}{2}}B^{\frac{n-r}{2}}\prod_{\substack{k\neq j_{0} \\ k \notin I}}|x_{k}|}\right)\left(\frac{x_{i_{0}}^{\frac{2r}{3}}q^{\frac{r}{3}}}{R^{\frac{2r}{3}}B^{\frac{r}{6}}|\cc|^{\frac{2r}{3}}}\right)\end{multline} \begin{multline*} = B^{n+\frac{1}{2}}\sum_{x_{i_{0}}, y_{j_{0}}}\frac{\poids(x_{i_{0}},y_{j_{0}})}{x_{i_{0}}^{n+1}y_{j_{0}}^{n+1}}\sum_{\substack{\cc\neq \0 \\ |\cc|\ll B^{\frac{1}{6}+\delta}}}\sum_{q=1}^{\infty}q^{-n-2}\sum_{\substack{(x_{i})_{i\neq i_{0}}\\ \mathcal{I}(\xx)=I}}\sum_{(a_{j}')_{j\neq j_{0}} \in (\ZZ/q\ZZ)^{n}} \\ \frac{\poids(\xx)}{\prod_{\substack{k\neq j_{0} \\ k \notin I}}|x_{k}|}|S_{q}(\xx,\aa')(\cc)|  \left(\frac{R^{n-r}x_{i_{0}}^{\frac{3(n-r)}{2}+1}y_{j_{0}}^{\frac{3(n-r)}{2}+1}}{q^{\frac{n-r}{2}+1}|\cc|^{\frac{n-r}{2}}B^{\frac{n-r}{2}}}\right)\left(\frac{x_{i_{0}}^{\frac{2r}{3}}q^{\frac{r}{3}}}{R^{\frac{2r}{3}}B^{\frac{r}{6}}|\cc|^{\frac{2r}{3}}}\right).
\end{multline*}


Nous allons \`a pr\'esent majorer la somme \[ \sum_{\substack{(x_{i})_{i\neq i_{0}}\\ \mathcal{I}(\xx)=I}}\sum_{(a_{j}')_{j\neq j_{0}} \in (\ZZ/q\ZZ)^{n}} \\ \frac{\poids(\xx)}{\prod_{\substack{k\neq j_{0} \\ k \notin I}}|x_{k}|}|S_{q}(\xx,\aa')(\cc)|, \] (pour $ i_{0}\neq j_{0} $ dans un premier temps) par des m\'ethodes analogues \`a celles que nous avons utilis\'ees pour \'etablir les majorations des lemmes \ref{lemmemoi} et \ref{lemmegeneral}. Supposons dans un premier temps que l'entier $ i_{0} $ n'appartient pas \`a $ I $. On peut alors majorer la somme ci-dessus par \[ \frac{1}{x_{i_{0}}}\sum_{(x_{i})_{i\neq i_{0}}}\sum_{(a_{j}')_{j\neq j_{0}} \in (\ZZ/q\ZZ)^{n}} \\ \frac{\poids(\xx)}{\prod_{\substack{k\neq i_{0},j_{0} \\ k \notin I}}|x_{k}|}|S_{q}(\xx,\aa')(\cc)|, \](on ne tient plus compte de la condition $ \mathcal{I}(\xx)=I $ sur les $ \xx $). En remarquant que cette somme est major\'ee par \begin{multline}\label{preced}\sum_{d\in(\ZZ/q\ZZ)^{\ast}}\left(\prod_{\substack{k\neq i_{0},j_{0} \\ k \notin I}}\sum_{x_{k}}\frac{\poids(x_{k})}{|x_{k}|}\sum_{a_{k}'\in\ZZ/q\ZZ}\left|\sum_{b\in \ZZ/q\ZZ}e_{q}((dx_{k}a_{k}'+c_{k})b)\right|\right) \\ \times  \left(\prod_{k\in I}\sum_{x_{k}}\poids(x_{k})\sum_{a_{k}'\in\ZZ/q\ZZ}\left|\sum_{b\in \ZZ/q\ZZ}e_{q}((dx_{k}a_{k}'+c_{k})b)\right|\right) \\ \times \left(\frac{1}{x_{i_{0}}}\sum_{a_{i_{0}}'\in\ZZ/q\ZZ}\left|\sum_{b\in \ZZ/q\ZZ}e_{q}((dx_{i_{0}}a_{i_{0}}'+c_{i_{0}})b)\right|\right) \\ \times \left(\sum_{x_{j_{0}}}\poids(x_{j_{0}})\left|\sum_{b\in \ZZ/q\ZZ}e_{q}((dx_{j_{0}}y_{j_{0}}+c_{j_{0}})b)\right|\right), \end{multline} puis, en majorant trivialement le troisi\`eme facteur par $ \frac{1}{x_{i_{0}}}q^{2} $, le quatri\`eme par $ x_{i_{0}}q $ et en remarquant que les sommes sur $ b $ valent $ q $ lorsque $ x_{k}a_{k}'\equiv -d^{-1}c_{k} (q) $ et $ 0 $ sinon, on obtient une majoration 

\begin{multline}\label{deuxpreced}
\eqref{preced} \ll q^{n+2}\sum_{d\in(\ZZ/q\ZZ)^{\ast}}\left(\prod_{\substack{k\neq i_{0},j_{0} \\ k \notin I}}\sum_{x_{k}}\frac{\poids(x_{k})}{|x_{k}|}\card\{a_{k}'\in \ZZ/q\ZZ \; | \; x_{k}a_{k}'\equiv -d^{-1}c_{k} (q)\}\right) \\ \times \left(\prod_{k\in I}\sum_{x_{k}}\poids(x_{k})\card\{a_{k}'\in \ZZ/q\ZZ \; | \; x_{k}a_{k}'\equiv -d^{-1}c_{k} (q)\}\right)
\end{multline}
Par les m\^emes calculs que pour le lemme \ref{lemmegeneral}, on trouve que le deuxi\`eme facteur est major\'e par $ x_{i_{0}}^{r}q^{r\epsilon_{1}} $, avec un $ \epsilon_{1}>0 $ arbitrairement petit. Pour le premier facteur on pose $ q=p_{1}^{e_{1}}...p_{s}^{e_{s}} $, et on \'ecrit : \begin{multline*}
\sum_{x_{k}}\frac{\poids(x_{k})}{|x_{k}|}\card\{a_{k}'\in \ZZ/q\ZZ \; | \; x_{k}a_{k}'\equiv -d^{-1}c_{k} (q)\} \\ \ll \sum_{\substack{f_{l}\leqslant e_{l} \\ l \in \{1,...,s\}}}\sum_{\substack{x_{k} \\ p_{1}^{f_{1}}...p_{s}^{f_{s}}|x_{k}}}\frac{\poids(x_{k})}{|x_{k}|}p_{1}^{f_{1}}...p_{s}^{f_{s}} \\ \ll \sum_{\substack{f_{l}\leqslant e_{l} \\ l \in \{1,...,s\}}}\sum_{\substack{h \leqslant \frac{x_{i_{0}}}{p_{1}^{f_{1}}...p_{s}^{f_{s}}} }}\frac{1}{hp_{1}^{f_{1}}...p_{s}^{f_{s}}}p_{1}^{f_{1}}...p_{s}^{f_{s}} \\ \ll \sum_{\substack{f_{l}\leqslant e_{l} \\ l \in \{1,...,s\}}}\sum_{h \leqslant x_{i_{0}} }\frac{1}{h} \ll q^{\epsilon_{2}}\log(x_{i_{0}})
\end{multline*} 
pout $ \epsilon_{2}>0 $ arbitrairement petit. On a ainsi que le premier facteur de \eqref{deuxpreced} est major\'e par $ q^{(n-1-r)\epsilon_{2}}\log(x_{i_{0}})^{n-1-r} $. Par cons\'equent, on obtient la majoration suivante \`a partir de \eqref{deuxpreced} :
\[ \sum_{\substack{(x_{i})_{\neq i_{0}}\\ \mathcal{I}(\xx)=I}}\sum_{(a_{j}')_{j\neq j_{0}} \in (\ZZ/q\ZZ)^{n}} \\ \frac{\poids(\xx)}{\prod_{\substack{k\neq j_{0} \\ k \notin I}}|x_{k}|}|S_{q}(\xx,\aa')(\cc)| \ll q^{n+3+\epsilon}x_{i_{0}}^{r}\log(x_{i_{0}})^{n-1-r}, \] avec $ \epsilon=r\epsilon_{1}+(n-1-r)\epsilon_{2} $ arbitrairement petit.\\

Si l'on suppose que $ i_{0} \in I $, on obtient plus ou moins la m\^eme formule que pour \eqref{preced} : 
 \begin{multline}\label{preced2}\sum_{d\in(\ZZ/q\ZZ)^{\ast}}\left(\prod_{\substack{k\neq j_{0} \\ k \notin I}}\sum_{x_{k}}\frac{\poids(x_{k})}{|x_{k}|}\sum_{a_{k}'\in\ZZ/q\ZZ}\left|\sum_{b\in \ZZ/q\ZZ}e_{q}((dx_{k}a_{k}'+c_{k})b)\right|\right) \\ \times \left(\prod_{\substack{k\in I \\ k\neq i_{0}}}\sum_{x_{k}}\poids(x_{k})\sum_{a_{k}'\in\ZZ/q\ZZ}\left|\sum_{b\in \ZZ/q\ZZ}e_{q}((dx_{k}a_{k}'+c_{k})b)\right|\right) \\ \times \left(\sum_{a_{i_{0}}'\in\ZZ/q\ZZ}\left|\sum_{b\in \ZZ/q\ZZ}e_{q}((dx_{i_{0}}a_{i_{0}}'+c_{i_{0}})b)\right|\right) \\ \times \left(\sum_{x_{j_{0}}}\poids(x_{j_{0}})\left|\sum_{b\in \ZZ/q\ZZ}e_{q}((dx_{j_{0}}y_{j_{0}}+c_{j_{0}})b)\right|\right). \end{multline} 
Par les m\^emes estimations que pr\'ec\'edement, on voit que l'on peut majorer le premier facteur par $ q^{n-r+(n-r)\epsilon_{2}}\log(x_{i_{0}})^{n-r} $, le deuxi\`eme par $ q^{r-1+(r-1)\epsilon_{1}}x_{i_{0}}^{r-1} $, le troisi\`eme par $ q^{2} $ et le quatri\`eme par $ qx_{i_{0}} $, ce qui nous donne donc la m\^eme majoration que dans le cas $ i_{0} \notin I $, \`a un facteur $ \log(x_{i_{0}}) $ suppl\'ementaire pr\`es : \begin{equation}\label{majofin} \sum_{\substack{(x_{i})_{i\neq i_{0}}\\ \mathcal{I}(\xx)=I}}\sum_{(a_{j}')_{j\neq j_{0}} \in (\ZZ/q\ZZ)^{n}  } \\ \frac{\poids(\xx)}{\prod_{\substack{k\neq j_{0} \\ k \notin I}}|x_{k}|}|S_{q}(\xx,\aa')(\cc)| \ll q^{n+3+\epsilon}x_{i_{0}}^{r}\log(x_{i_{0}})^{n-r}. \end{equation}
\begin{rem}
Lorsque l'on est dans le cas $ i_{0}=j_{0} $, on peut voir par le m\^eme type de calculs que l'on trouve l'estimation suivante : \[  \sum_{\substack{(x_{i})_{i\neq i_{0}}\\ \mathcal{I}(\xx)=I}}\sum_{(a_{j}')_{j\neq j_{0}} \in (\ZZ/q\ZZ)^{n} } \\ \frac{\poids(\xx)}{\prod_{\substack{k\neq j_{0} \\ k \notin I}}|x_{k}|}|S_{q}(\xx,\aa')(\cc)|\ll q^{n+2+\epsilon}x_{i_{0}}^{r}\log(x_{i_{0}})^{n-r}. \] La majoration \eqref{majofin} reste donc valable dans le cas $ i_{0}=j_{0} $.
\end{rem}
En utilisant cette derni\`ere majoration \eqref{majofin} dans la formule \eqref{nouvexpr2} on obtient un terme d'erreur : 

\begin{multline*} \eqref{nouvexpr2} \ll B^{n+\frac{1}{2}}\sum_{x_{i_{0}}, y_{j_{0}}}\frac{\poids(x_{i_{0}},y_{j_{0}})}{x_{i_{0}}^{n+1}y_{j_{0}}^{n+1}}\sum_{\substack{\cc\neq \0 \\ |\cc|\ll B^{\frac{1}{6}+\delta}}}\sum_{q\ll Q}q^{1+\epsilon}x_{i_{0}}^{r}\log(x_{i_{0}})^{n-r} \\ \left(\frac{R^{n-r}x_{i_{0}}^{\frac{3(n-r)}{2}+1}y_{j_{0}}^{\frac{3(n-r)}{2}+1}}{q^{\frac{n-r}{2}+1}|\cc|^{\frac{n-r}{2}}B^{\frac{n-r}{2}}}\right)\left(\frac{x_{i_{0}}^{\frac{2r}{3}}q^{\frac{r}{3}}}{R^{\frac{2r}{3}}B^{\frac{r}{6}}|\cc|^{\frac{2r}{3}}}\right)
\end{multline*}
que l'on r\'e\'ecrit, en rappelant que, d'apr\`es la remarque \ref{restriction}, on ne consid\`ere que les $ x_{i_{0}},y_{j_{0}} $ tels que $ x_{i_{0}}\leqslant y_{j_{0}} $ et $ y_{j_{0}}\geqslant \sqrt{q} $ :
\begin{multline*}  B^{\frac{n}{2}+\frac{r}{3}+\frac{1}{2}}R^{n-\frac{5r}{3}}\sum_{x_{i_{0}}, y_{j_{0}}}\poids(x_{i_{0}},y_{j_{0}})\frac{x_{i_{0}}^{r}}{x_{i_{0}}^{r}}\underbrace{\frac{x_{i_{0}}^{\frac{2r}{3}}}{y_{j_{0}}^{\frac{2r}{3}}}}_{\leqslant 1}\frac{x_{i_{0}}^{\frac{3(n-r)}{2}+1}}{x_{i_{0}}^{n-r+1}}\frac{y_{j_{0}}^{\frac{3(n-r)}{2}+1}}{y_{j_{0}}^{n-r+1}}\log(x_{i_{0}})^{n-r} \\ \sum_{q\ll Q}\underbrace{\frac{q^{\epsilon-\frac{n}{2}+\frac{5r}{6}}}{y_{j_{0}}^{\frac{r}{3}}}}_{\ll q^{\epsilon-\frac{n}{2}+\frac{2r}{3}} }\underbrace{\sum_{\substack{\cc\neq \0 \\ |\cc|\ll B^{\frac{1}{6}+\delta}}}\frac{1}{|\cc|^{\frac{n}{2}+\frac{r}{6}}}}_{\substack{\ll (B^{\frac{1}{6}+\delta})^{\frac{n}{2}-\frac{r}{6}+1}\\=B^{\frac{n}{12}-\frac{r}{36}+\frac{1}{6}+\delta'} }}
\end{multline*} (avec $ \delta'=\delta(\frac{n}{2}-\frac{r}{6}+1) $ arbitrairement petit), et on obtient ainsi une majoration finale (en rappelant que $ R=B^{\lambda} $, avec $ \lambda>0 $ qui sera fix\'e ult\'erieurement, et $ x_{i_{0}}y_{j_{0}}\leqslant B^{\frac{2}{3}} $) : 

\begin{multline}\label{termeB}  B^{\frac{7n}{12}+\frac{11r}{36}+\frac{1}{2}+\frac{1}{6}+(n-\frac{5r}{3})\lambda+\delta'}\log(B)^{n-r}\sum_{x_{i_{0}}, y_{j_{0}}}\poids(x_{i_{0}},y_{j_{0}})x_{i_{0}}^{\frac{n-r}{2}} y_{j_{0}}^{\frac{n-r}{2}}\sum_{q\ll Q}q^{\epsilon-\frac{n}{2}+\frac{2r}{3}} \\  \ll B^{\frac{11n}{12}-\frac{r}{36}+\frac{4}{3}+(n-\frac{5r}{3})\lambda+\delta''}\sum_{q\ll Q}q^{\epsilon-\frac{n}{2}+\frac{2r}{3}}. \end{multline}
(pour un certain $ \delta''>0 $ arbitrairement petit. On note \`a pr\'esent l'exposant $ \lambda $ sous la forme $ \lambda=\frac{\beta}{n} $, avec $ \beta>0 $. Dans le cas o\`u $ \epsilon-\frac{n}{2}+\frac{2r}{3}\leqslant -1 $, on peut majorer la somme sur $ q $ par $ \log(B) $, et on obtient ainsi un terme d'erreur du type \[ B^{\frac{11n}{12}-\frac{r}{36}+\frac{4}{3}+(n-\frac{5r}{3})\lambda+\delta''}\ll B^{\frac{11n}{12}+\frac{4}{3}+n\lambda+\delta''}, \] qui est alors du type $ O(B^{n}) $ pour tout \begin{equation}\label{n1}
n> 16+12n\lambda=16+12\beta,
\end{equation} (et pour $ \delta'' $ et $ \epsilon $ choisis suffisamment petits). Si l'on a $ \epsilon-\frac{n}{2}+\frac{2r}{3}> -1 $, alors on a la majoration : \[ \sum_{q\ll Q}q^{\epsilon-\frac{n}{2}+\frac{2r}{3}}\ll Q^{1+\epsilon-\frac{n}{2}+\frac{2r}{3}} =B^{\frac{1}{2}+\frac{r}{3}-\frac{n}{4}+\frac{\epsilon}{2}}, \] et donc, avec \eqref{termeB} on trouve un terme d'erreur du type $ O(B^{\alpha}\log(B)) $, avec \begin{align*} \alpha & =\frac{2n}{3}+\frac{11r}{36}+2+\frac{\epsilon}{2}+ (n-\frac{5r}{3})\lambda+\delta' \\  & =\frac{2n}{3}+(\frac{11}{36}-\frac{5}{3}\lambda)r+2+\frac{\epsilon}{2}+ n\lambda+\delta'\\ & \leqslant  \frac{2n}{3}+(\frac{11}{36}-\frac{5}{3}\lambda)n+2+\frac{\epsilon}{2}+n\lambda+\delta' \\ & = \frac{35n}{36}+2+\frac{\epsilon}{2}-\frac{2n}{3}\lambda +\delta' \end{align*} 
(en effet, on peut supposer que $ \frac{11}{36}-\frac{5}{3}\lambda>0 $, car dans le cas contraire, on a alors $ \lambda\geqslant \frac{11}{60} $, ce qui vient contredire l'in\'egalit\'e \eqref{n1}).
Le terme d'erreur est donc du type $ O(B^{n}) $ pour \begin{equation}\label{n2}
n > 72-24n\lambda=66-24\beta,
\end{equation}(et pour $ \delta' $ et $ \epsilon $ choisis assez petits). On en d\'eduit, en regroupant les in\'egalit\'es \eqref{n1} et \eqref{n2}, que la partie $ I'(\cc) $ de l'int\'egrale $ I(\cc) $ fournit un terme d'erreur du type $ O(B^{n}) $ pour tout entier $ n> \min\{16+12\beta,72-24\beta\} $, et on remarque que cette borne inf\'erieure pour $ n $ est minimale pour $ \beta=\frac{14}{9} $ et vaut $ \frac{104}{3}\in [34,35] $\label{page2}. \\

On peut donc finalement conclure que d'apr\`es les r\'esultats obtenus dans la section \ref{estimationsimples}, ceux de la page \pageref{page1} et ceux de la page \pageref{page2} on a la proposition suivante : 
\begin{prop}\label{propTE}
Pour tout $ n\geqslant 35 $, on a que :  
\[ \sum_{x_{i_{0}},y_{j_{0}}}\poids(x_{i_{0}},y_{j_{0}})\sum_{\substack{(x_{i})_{i\neq i_{0}} \\ (y_{j})_{j\neq j_{0}} }}\poids(\xx,\yy)\sum_{\cc\neq \0}\sum_{q\ll Q}q^{-n-1}S_{q}(\xx,\yy)(\cc)I_{q}(\xx,\yy)(\cc)=O_{\varepsilon}(B^{n}). \]
 \end{prop}

\section{ Le terme principal}

\subsection{ G\'en\'eralit\'es}\label{generalite}

On suppose d\'esormais que $ n\geqslant 37 $. Rappelons que l'on cherche \`a montrer que la somme \eqref{nouvformule} est du type $ C_{i_{0},j_{0}}B^{n}\log(B)^{2} $ (o\`u $ C_{i_{0},j_{0}} $ est une constante que nous expliciterons) et d'apr\`es la proposition \ref{propTE}, on constate qu'il nous reste \`a traiter la partie $ \cc= \0 $ de cette somme. Nous allons donc chercher \`a montrer que la somme : \begin{equation}\label{formulegeneral}
B^{n}\sum_{x_{i_{0}},y_{j_{0}}}\frac{\poids(x_{i_{0}},y_{j_{0}})}{x_{i_{0}}^{n+1}y_{j_{0}}^{n+1}}\sum_{\substack{(x_{i})_{i\neq i_{0}} \\ (y_{j})_{j\neq j_{0}}, }}\poids(\xx,\yy)\sum_{q\ll Q}q^{-n-1}S_{q}(\xx,\yy)(\0)I(\0)
\end{equation} avec \begin{equation}\label{S(00)} S_{q}(\xx,\yy)(\0)=\sum_{d\in (\ZZ/q\ZZ)^{\ast}}\sum_{\bb \in (\ZZ/q\ZZ)^{n+1}}e_{q}(d\xx.\yy.\bb), \end{equation} et (voir \eqref{oubli})
 \begin{equation}\label{I(00)} I(\0)=\int_{\RR^{n+1}}\poids(\uu)h\left(\frac{q}{Q},\frac{\xx.\yy.\uu}{x_{i_{0}}y_{j_{0}}}\right)d\uu. \end{equation}
fournit un terme principal du type $ C_{i_{0},j_{0}}B^{n}\log(B)^{2}(1+o_{\epsilon \ra 0}(1)) $.\\

Nous allons voir avant tout, par des m\'ethodes analogues \`a celles qui ont \'et\'e d\'evelopp\'ees pour le terme d'erreur, que l'on peut restreindre la somme sur $ q $ aux entiers $ q\leqslant x_{i_{0}}y_{j_{0}} $. Remarquons avant tout que la somme \[ S_{q}(\xx,\yy)(\0)=\sum_{d \in (\ZZ/q\ZZ)^{\ast}}\sum_{\bb\in (\ZZ/q\ZZ)^{n+1}} e_{q}(d\xx.\yy.\bb)\] vaut $ q^{n+1}\varphi(q) $ si $ x_{k}y_{k}\equiv 0 (q) $ pour tout $ k $, et $ 0 $ sinon. Par ailleurs, si $ q>x_{i_{0}}y_{j_{0}} $, $ x_{k}y_{k}\equiv 0 (q) $ \'equivaut \`a $ x_{k}=0 $ ou $ y_{k}=0 $. Or, sur $ \supp\poids(\xx,\yy) $ on a $ x_{k}\neq 0 $ et $ y_{k}\neq 0 $ pour tout $ k $. On peut donc exclure le cas o\`u  $ q>x_{i_{0}}y_{j_{0}} $ (en fait, nous verrons ult\'erieurement que l'on peut se restreindre \`a la somme sur $ q\leqslant \min\{x_{i_{0}},y_{j_{0}}\} $).

On souhaite donc \'evaluer le terme principal \eqref{formulegeneral}. On peut r\'e\'ecrire la formule \eqref{formulegeneral} sous la forme :
 
 \begin{equation}\label{TP4} B^{n}\sum_{x_{i_{0}},y_{j_{0}}}\frac{\poids(x_{i_{0}},y_{j_{0}})}{x_{i_{0}}^{n+1}y_{j_{0}}^{n+1}}\sum_{q\ll Q}\sum_{\substack{(a_{i})_{i\neq i_{0}}\in (\ZZ /q\ZZ)^{n} \\  (a_{j}')_{j\neq j_{0}}\in (\ZZ /q\ZZ)^{n}} }q^{-n-1}S_{q}(\aa,\aa')(\0)I(\aa,\aa')\end{equation}
 
 o\`u l'on a pos\'e \begin{equation}
  I(\aa,\aa')=\sum_{\substack{x_{i}\equiv a_{i}(q), i\neq i_{0} \\ y_{j}\equiv a_{j}'(q), j\neq j_{0}}}\poids(\xx,\yy)\int_{\RR^{n+1}}\poids(\uu)h\left(\frac{q}{Q},\frac{\xx.\yy.\uu}{x_{i_{0}}y_{j_{0}}}\right)d\uu, \end{equation} o\`u $ \aa $ (resp. $ \aa' $) d\'esigne le vecteur de $ (\ZZ/q\ZZ)^{n+1} $ dont les coordonn\'ees sont les $ a_{i} $ pour $ i\neq i_{0} $ (resp. $ a_{j}' $ pour $ j\neq j_{0} $) et $ x_{i_{0}}\mod q $ pour $ i=i_{0} $ (resp. $ y_{j_{0}}\mod q $ pour $ j=j_{0} $). Pour traiter ce terme principal, il est n\'ecessaire de traiter s\'epar\'ement le cas $ i_{0}=j_{0} $ et le cas $ i_{0}\neq j_{0} $.

\subsection{ Le cas o\`u $ i_{0}=j_{0} $}

On suppose dans cette partie que $ i_{0}=j_{0} $. Commen\c{c}ons par remplacer la variable $ u_{i_{0}} $ par $ t=\frac{\xx.\yy.\uu}{x_{i_{0}}y_{i_{0}}}=u_{i_{0}}+\sum_{k\neq i_{0}}\frac{x_{k}y_{k}u_{k}}{x_{i_{0}}y_{i_{0}}} $. On a alors \[ I(\aa,\aa')=\sum_{\substack{x_{i}\equiv a_{i}(q), i\neq i_{0} \\ y_{j}\equiv a_{j}'(q), j\neq i_{0}}}\int_{\RR} \poids(\xx,\yy)J(t)h\left(\frac{q}{Q},t\right)dt \] avec \[ J(t)=\int_{\RR^{n}}\poids(\hat{\uu},t)d\hat{\uu} \] ($ \hat{\uu} $ d\'esignant $ \uu $ priv\'ee de la variable $ u_{i_{0}} $) o\`u \nomenclature{$ \poids(\hat{\uu},t) $}{ nouvelle fonction poids}$ \poids(\hat{\uu},t) $ est la nouvelle fonction poids (cf. \eqref{(8)} pour l'ancienne fonction poids) \begin{multline*} \poids(\hat{\uu},t)=\left(\prod_{k\neq i_{0}}\omeg\left(1-|u_{k}|\right)\omeg\left(|u_{k}|-\frac{1}{P}\right)\right) \\ \omeg\left(1-\left|t-\sum_{k\neq i_{0}}\frac{x_{k}y_{k}u_{k}}{x_{i_{0}}y_{i_{0}}}\right|\right) \omeg\left(\left|t-\sum_{k\neq i_{0}}\frac{x_{k}y_{k}u_{k}}{x_{i_{0}}y_{i_{0}}}\right|-\frac{1}{P}\right) . \end{multline*} On applique ensuite le lemme suivant  (cf. \cite[Lemme 9]{HB1}) 
\begin{lemma}\label{TE0}
On a pour tout entier $ N>0 $ : \[ \int_{\RR}J(t)h\left(\frac{q}{Q},t\right)dt=J(0)+O\left(\left(\frac{q}{Q}\right)^{N}\right). \]
\end{lemma}

Nous allons \`a pr\'esent montrer que l'on peut restreindre la somme sur $ q $ aux entiers $ q\leqslant\min\{ x_{i_{0}},y_{j_{0}}\} $. Rappelons dans un premier temps que l'on a suppos\'e $ x_{i_{0}}y_{j_{0}}\geqslant q $. Quitte \`a intervertir $ x_{i_{0}} $ et $ y_{j_{0}} $, on peut supposer que $ x_{i_{0}}\leqslant y_{j_{0}} $. On a donc $ y_{j_{0}} \geqslant \sqrt{q} $. Supposons donc que l'on a $ x_{i_{0}}<q $. On remarque que d'apr\`es le lemme \ref{TE0}, on a que \[ \int_{\RR}J(t)h\left(\frac{q}{Q},t\right)dt=O(1), \]
on peut donc donner la majoration suivante du terme principal :
\begin{multline}\label{TP3} \eqref{TP4} \ll B^{n}\sum_{x_{i_{0}},y_{j_{0}}}\frac{\poids(x_{i_{0}},y_{j_{0}})}{x_{i_{0}}^{n+1}y_{j_{0}}^{n+1}}\sum_{\substack{q\leqslant Q \\ x_{i_{0}}<q}}q^{-n-1}\sum_{\substack{(x_{i})_{i\neq i_{0}} \\  (a_{j}')_{j\neq j_{0}}\in (\ZZ /q\ZZ)^{n}} }\poids(\xx)S_{q}(\xx,\aa')(\0) \\ \sum_{\substack{ y_{j}\equiv a_{j}'(q), j\neq j_{0}}}\underline{w}_{\varepsilon}(\yy). \end{multline}

Par ailleurs, on peut majorer la partie $ \sum_{\substack{ y_{j}\equiv a_{j}'(q), j\neq j_{0}}}\underline{w}_{\varepsilon}(\yy) $ par $ \frac{y_{j_{0}}^{n}}{q^{\frac{n}{2}}} $, et en utilisant la majoration \eqref{inegmoi} et la remarque \ref{remlemme}, on trouve : \[ \sum_{\substack{(x_{i})_{i\neq i_{0}} \\  (a_{j}')_{j\neq j_{0}}\in (\ZZ /q\ZZ)^{n}} }\poids(\xx)S_{q}(\xx,\aa')(\0)\ll q^{n+3+\epsilon}x_{i_{0}}^{n}. \]

On obtient donc, en utilisant le fait que $ x_{i_{0}}<q $ : 

\begin{multline}
\eqref{TP4} \ll  B^{n}\sum_{x_{i_{0}},y_{j_{0}}}\frac{\poids(x_{i_{0}},y_{j_{0}})}{x_{i_{0}}y_{j_{0}}}\sum_{\substack{q\leqslant Q \\ x_{i_{0}}<q}}q^{-\frac{n}{2}+2+\epsilon} \\ \ll  B^{n}\sum_{x_{i_{0}},y_{j_{0}}}\frac{\poids(x_{i_{0}},y_{j_{0}})}{x_{i_{0}}y_{j_{0}}}\sum_{\substack{q\leqslant Q \\ x_{i_{0}}<q}}q^{-\frac{n}{2}+2+\epsilon}\frac{q}{x_{i_{0}}} \\ = B^{n}\sum_{x_{i_{0}},y_{j_{0}}}\frac{\poids(x_{i_{0}},y_{j_{0}})}{x_{i_{0}}^{2}y_{j_{0}}}\sum_{\substack{q\leqslant Q \\ x_{i_{0}}<q}}q^{-\frac{n}{2}+3+\epsilon}.
\end{multline}
Or, on a que $ \sum_{\substack{q\leqslant Q \\ x_{i_{0}}<q}}q^{-\frac{n}{2}+3+\epsilon} $ est convergente pour tout $ n\geqslant 9 $ et \linebreak $ B^{n}\sum_{x_{i_{0}}\leqslant y_{j_{0}}}\frac{\poids(x_{i_{0}},y_{j^{0}})}{x_{i_{0}}^{2}y_{j_{0}}}=O(B^{n}\log(B)) $, qui sera donc n\'egligeable par rapport au reste de la somme \eqref{TP4} qui sera du type $ O(B^{n}\log(B)^{2}) $. On peut donc restreindre la somme aux entiers $ q $ tels que $ q\leqslant x_{i_{0}} $ et $ q\leqslant y_{j_{0}} $.

On ne consid\`ere dor\'enavant que des entiers $ q $ tels que $ q\leqslant x_{i_{0}} $ et $ q\leqslant y_{j_{0}} $ et nous allons \`a pr\'esent nous int\'eresser \`a la partie $ O((\frac{q}{Q})^{N}) $ de la d\'ecomposition du lemme \ref{TE0}, et montrer que cette partie fournit un terme n\'egligeable : on choisit $ N=n/2 $, et on obtient alors une contribution : \[ B^{\frac{3n}{4}}\sum_{x_{i_{0}},y_{j_{0}}}\frac{\poids(x_{i_{0}},y_{j_{0}})}{x_{i_{0}}^{n+1}y_{j_{0}}^{n+1}}\sum_{q\leqslant Q}\sum_{\substack{(a_{i})_{i\neq i_{0}}\in (\ZZ /q\ZZ)^{n}\\  (a_{j}')_{j\neq i_{0}}\in (\ZZ /q\ZZ)^{n}} }q^{-\frac{n}{2}-1}S_{q}(\aa,\aa')(\0)\sum_{\substack{x_{i}\equiv a_{i}(q), i\neq i_{0} \\ y_{j}\equiv a_{j}'(q), j\neq j_{0}}}\poids(\xx,\yy). \]
En utilisant les majorations  \[ \sum_{\substack{x_{i}\equiv a_{i}(q), i\neq i_{0} \\ y_{j}\equiv a_{j}'(q), j\neq j_{0}}}\poids(\xx,\yy)\ll \frac{x_{i_{0}}^{n}y_{j_{0}}^{n}}{q^{2n}}, \] (\'etant donn\'e que l'on a suppos\'e $ q\leqslant \min\{x_{i_{0}}y_{j_{0}}\} $) et \[ \sum_{\substack{(a_{i})_{i\neq i_{0}}\in (\ZZ /q\ZZ)^{n}\\  (a_{j}')_{j\neq j_{0}}\in (\ZZ /q\ZZ)^{n}} }S_{q}(\aa,\aa')(\0) \ll q^{3+2n+\epsilon} \] (majoration issue du lemme \ref{lemmemoi}), on obtient un terme d'erreur : \begin{align*} B^{\frac{3n}{4}}\sum_{x_{i_{0}},y_{j_{0}}}\frac{\poids(x_{i_{0}},y_{j_{0}})}{x_{i_{0}}y_{j_{0}}}\sum_{q\leqslant Q}q^{2-\frac{n}{2}+\epsilon} & \ll  B^{\frac{3n}{4}}\sum_{x_{i_{0}},y_{j_{0}}}\frac{\poids(x_{i_{0}},y_{j_{0}})}{x_{i_{0}}y_{j_{0}}} \\ & \ll B^{\frac{3n}{4}}\log(B)^{2}\end{align*} qui est du type $ O(B^{n})$ pour tout $ n\geqslant 7 $.\\

On s'int\'eresse donc \`a pr\'esent au terme principal 

\begin{multline}\label{TP2} B^{n}\sum_{x_{i_{0}},y_{j_{0}}}\frac{\poids(x_{i_{0}},y_{j_{0}})}{x_{i_{0}}^{n+1}y_{j_{0}}^{n+1}}\sum_{\substack{q\leqslant Q \\ x_{i_{0}}<q}}\sum_{\substack{(a_{i})_{i\neq i_{0}}\in (\ZZ /q\ZZ)^{n} \\  (a_{j}')_{j\neq j_{0}}\in (\ZZ /q\ZZ)^{n}} }q^{-n-1}S_{q}(\aa,\aa')(\0) \\ \sum_{\substack{x_{i}\equiv a_{i}(q), i\neq i_{0} \\ y_{j}\equiv a_{j}'(q), j\neq j_{0}}}\underline{w}_{\varepsilon}(\xx,\yy)\int_{\RR^{n}}\poids(0,\hat{\uu})d\hat{\uu}. \end{multline}

On remarque que l'on a :

\begin{align*} \sum_{\substack{x_{i}\equiv a_{i}(q), i\neq i_{0} \\ y_{j}\equiv a_{j}'(q), j\neq j_{0}}}\underline{w}_{\varepsilon}(\xx,\yy) & = \left(\prod_{i\neq i_{0}}\sum_{x_{i}\equiv a_{i}(q)}\poids(x_{i})\right)\left(\prod_{j\neq j_{0}}\sum_{y_{j}\equiv a_{j}'(q)}\poids(y_{j})\right) \\ & = \left(\prod_{i\neq i_{0}}\sum_{z_{i}}\poids(a_{i}+qz_{i})\right)\left(\prod_{j\neq j_{0}}\sum_{z_{j}'}\poids(a_{j}'+qz_{j})\right) \end{align*}
On remarque \'egalement que, par la formule sommatoire d'Euler, la somme $ \sum_{z_{i}}\poids(a_{i}+qz_{i}) $ vaut : \begin{equation}\label{euler}
\int_{\RR}\poids(a_{i}+qz_{i})dz_{i}+\int_{\RR}(z_{i}-\lfloor z_{i} \rfloor) \frac{\partial}{\partial z_{i}}(\poids(a_{i}+qz_{i}))dz_{i}.
\end{equation}   
Par ailleurs, \begin{multline*} \frac{\partial}{\partial z_{i}}(\poids(x_{i_{0}},a_{i}+qz_{i}))=\frac{q}{x_{i_{0}}}\omeg'\left(1-\frac{|a_{i}+qz_{i}|}{x_{i_{0}}}\right)\omeg\left(\frac{|a_{i}+qz_{i}|}{x_{i_{0}}}-\frac{1}{x_{i_{0}}}\right) \\ +\frac{q}{x_{i_{0}}}\omeg\left(1-\frac{|a_{i}+qz_{i}|}{x_{i_{0}}}\right)\omeg'\left(\frac{|a_{i}+qz_{i}|}{x_{i_{0}}}-\frac{1}{x_{i_{0}}}\right)\end{multline*} est non nul et vaut $ O(qx_{i_{0}}^{-1}) $ sur un ensemble de mesure $ O_{\epsilon}(x_{i_{0}}q^{-1}) $. Ainsi, on a \[  \sum_{z_{i}}\poids(a_{i}+qz_{i})=\underbrace{\frac{1}{q}\int_{\RR}\poids(x_{i})dx_{i}}_{=O(q^{-1}x_{i_{0}})}+O(1), \] pour tout $ i\neq i_{0} $.
On a de m\^eme avec les $ y_{j} $ : 

\[  \sum_{z_{j}'}\poids(a_{j}'+qz_{j}')=\underbrace{\frac{1}{q}\int_{\RR}\poids(y_{j})dy_{j}}_{=O(q^{-1}y_{j_{0}})}+O(1), \] pour tout $ j\neq j_{0} $.
Par cons\'equent, on trouve \begin{multline*} \sum_{\substack{x_{i}\equiv a_{i}(q), i\neq i_{0} \\ y_{j}\equiv a_{j}'(q), j\neq j_{0}}}\underline{w}_{\varepsilon}(\xx,\yy) =\frac{1}{q^{2n}}\int_{\RR^{2n}}\poids(\xx,\yy)d\xx'd\yy' \\ +O(q^{-2n+1}x_{i_{0}}^{n-1}y_{j_{0}}^{n-1}\max\{x_{i_{0}},y_{j_{0}}\}). \end{multline*}

On remarque alors que le terme d'erreur issu de cette \'egalit\'e est (en utilisant le lemme \ref{lemmemoi}): \begin{multline*} B^{n}\sum_{x_{i_{0}},y_{j_{0}}}\frac{\poids(x_{i_{0}},y_{j_{0}})\max\{x_{i_{0}},y_{j_{0}}\}}{x_{i_{0}}^{2}y_{j_{0}}^{2}}\sum_{q\leqslant Q}\sum_{\substack{(a_{i})_{i\neq i_{0}}\in (\ZZ /q\ZZ)^{n} \\  (a_{j}')_{j\neq j_{0}}\in \ZZ /q\ZZ} }q^{-3n}S_{q}(\aa,\aa')(\0)\\  \ll B^{n}\sum_{x_{i_{0}},y_{j_{0}}}\frac{\poids(x_{i_{0}},y_{j_{0}})}{x_{i_{0}}^{2}y_{j_{0}}^{2}}\max\{x_{i_{0}},y_{j_{0}}\}\sum_{q\leqslant Q}q^{3-n+\varepsilon} \\ \ll B^{n}\sum_{x_{i_{0}},y_{j_{0}}}\frac{\poids(x_{i_{0}},y_{j_{0}})}{x_{i_{0}}^{2}y_{j_{0}}^{2}}\max\{x_{i_{0}},y_{j_{0}}\}\ll B^{n}\log(B). \end{multline*}

Le terme principal est alors :

 \begin{multline*}B^{n}\sum_{x_{i_{0}},y_{j_{0}}}\frac{\poids(x_{i_{0}},y_{j_{0}})}{x_{i_{0}}^{n+1}y_{j_{0}}^{n+1}}\sum_{q\leqslant Q}\sum_{\substack{(a_{i})_{i \neq i_{0}}\in (\ZZ /q\ZZ)^{n}\\  (a_{j}')_{j \neq j_{0}}\in (\ZZ /q\ZZ)^{n} } }q^{-3n-1}S_{q}(\aa,\aa')(\0) \\ \int_{\RR^{3n}}\underline{w}_{\varepsilon}(\xx,\yy)\poids(0,\hat{\uu})d\xx'd\yy'd\hat{\uu}. \end{multline*}
Que l'on r\'e\'ecrit, apr\`es changement de variables $ \xx=x_{i_{0}}\ss $ et $ \yy=y_{j_{0}}\tt $ :

\begin{multline*}B^{n}\sum_{x_{i_{0}},y_{j_{0}}}\frac{\poids(x_{i_{0}},y_{j_{0}})}{x_{i_{0}}y_{j_{0}}}\sum_{q\leqslant Q}\sum_{\substack{(a_{i})_{i\neq i_{0}}\in (\ZZ /q\ZZ)^{n}\\  (a_{j}')_{j \neq j_{0}}\in (\ZZ /q\ZZ)^{n}, } }q^{-3n-1}S_{q}(\aa,\aa')(\0)\\ \int_{\RR^{3n}}\underline{w}_{\varepsilon}(\ss,\tt)\poids(0,\hat{\uu})d\ss'd\tt'd\hat{\uu}. \end{multline*}

Puis, en notant \[ \tilde{\chi}_{i_{0},i_{0}}(\ss',\tt',\uu)=\left\{\begin{array}{lllll}
1 & \mbox{si} & |\ss'|\in [\frac{1}{x_{i_{0}}},1],|\tt'|\in [\frac{1}{y_{i_{0}}},1],|\hat{\uu}|\in[\frac{1}{P},1] && \\  & \mbox{et} & \left| \sum_{k\neq i_{0}}\frac{x_{k}y_{k}u_{k}}{x_{i_{0}}y_{j_{0}}}\right| \in[\frac{1}{P},1]&& \\ 0 & \mbox{sinon}, & & &
\end{array}\right. \]
on remarque que $ \underline{w}_{\varepsilon}(\ss,\tt)\poids(\uu)\ra  \tilde{\chi}_{i_{0},i_{0}}(\ss',\tt',\hat{\uu}) $ lorsque $\varepsilon\ra 0 $, et qu'ainsi on a   \[
 \int_{\RR^{3n}}\underline{w}_{\varepsilon}(\ss,\tt)\poids(0,\hat{\uu})d\ss'd\tt'd\hat{\uu}\\=\tilde{\sigma}_{i_{0},i_{0}}+O(\varepsilon) \] o\`u 
\[ \tilde{\sigma}_{i_{0},i_{0}}=\int_{\RR^{3n}}\tilde{\chi}_{i_{0},i_{0}}(\ss',\tt',\hat{\uu})d\ss' d\tt' d\hat{\uu}. \]

Nous voudrions \`a pr\'esent remplacer $ \tilde{\chi}_{i_{0},i_{0}} $ par  $ \chi_{i_{0},i_{0}} $, avec : 

\[ \chi_{i_{0},i_{0}}(\ss',\tt',\uu)=\left\{\begin{array}{lllll}
1 & \mbox{si} & |\ss'|,|\tt'|,|\hat{\uu}|\leqslant 1 & \mbox{et} & \left| \sum_{k\neq i_{0}}s_{k}t_{k}u_{k}\right| \leqslant 1\\ 0 &\mbox{sinon}. & & &
\end{array}\right. \]

Or il est clair que l'on a : 

\[ (\chi_{i_{0},i_{0}}-\tilde{\chi}_{i_{0},i_{0}})(\ss',\tt',\uu)=\left\{\begin{array}{lll}
1 & \mbox{si} & \min|s_{i}|< \frac{1}{x_{i_{0}}} \\ & \mbox{ou} & \min|t_{j}|< \frac{1}{y_{i_{0}}}\\ & \mbox{ou} & \min|u_{k}|< \frac{1}{P} \\ & \mbox{ou} & \left| \sum_{k\neq i_{0}}s_{k}t_{k}u_{k}\right|<\frac{1}{P}\\ 0 & \mbox{sinon}. & 
\end{array}\right. \]

Par cons\'equent, si l'on note \[ \sigma_{i_{0},i_{0}}=\int_{\RR^{3n}}\chi_{i_{0},i_{0}}(\ss',\tt',\hat{\uu})d\ss' d\tt' d\hat{\uu}, \] on trouve : \[  \sigma_{i_{0},i_{0}}-\tilde{\sigma}_{i_{0},i_{0}}=\int_{\RR^{3n}}(\chi_{i_{0},i_{0}}-\tilde{\chi}_{i_{0},i_{0}})(\ss',\tt',\uu)d\ss' d\tt' d\hat{\uu}\ll \min\{ \frac{1}{x_{i_{0}}},\frac{1}{y_{i_{0}}},\frac{1}{P}\} \]
et on a ainsi que \begin{multline*} B^{n}\sum_{x_{i_{0}},y_{j_{0}}}\frac{\poids(x_{i_{0}},y_{j_{0}})}{x_{i_{0}}y_{j_{0}}}\sum_{q\leqslant Q}\sum_{\substack{(a_{i})_{i\neq i_{0}}\in (\ZZ /q\ZZ)^{n}\\  (a_{j}')_{ j \neq j_{0}}\in (\ZZ /q\ZZ)^{n}} }q^{-3n-1}S_{q}(\aa,\aa')(\0)(\sigma_{i_{0},i_{0}}-\tilde{\sigma}_{i_{0},i_{0}})\\ =O(B^{n}\log(B)). \end{multline*}
Ainsi, dans le cas o\`u $ i_{0}=j_{0} $, le terme principal sera donn\'e par :
\[ \sigma_{i_{0},i_{0}}B^{n}\sum_{x_{i_{0}},y_{j_{0}}}\frac{\poids(x_{i_{0}},y_{j_{0}})}{x_{i_{0}}y_{j_{0}}}\sum_{q\leqslant Q}\sum_{\substack{(a_{i})_{i\neq i_{0}}\in (\ZZ /q\ZZ)^{n}\\  (a_{j}')_{ j \neq j_{0}}\in (\ZZ /q\ZZ)^{n}} }q^{-3n-1}S_{q}(\aa,\aa')(\0). \]

\subsection{ Le cas o\`u $ i_{0}\neq j_{0} $}

Supposons \`a pr\'esent que $ i_{0}\neq j_{0} $. Nous allons s\'eparer la formule en deux parties, en distinguant deux cas :

On introduit \`a cette fin les fonctions poids :\[\underline{w}_{\varepsilon,1}(\xx,\yy)=\left(\prod_{k\neq i_{0},j_{0}}\poids(x_{k})\poids(y_{k})\right)\poids(y_{i_{0}})\poidsun(y_{i_{0}},x_{j_{0}}), \] o\`u \[ \poidsun(y_{i_{0}},x_{j_{0}})=\poids(x_{j_{0}})\omeg\left( \frac{|x_{j_{0}}|}{x_{i_{0}}}-\frac{|y_{i_{0}}|}{y_{j_{0}}}\right), \] (cette fonction poids est la fonction poids traduit la condition suppl\'ementaire $ |x_{j_{0}}y_{j_{0}}|\geqslant |x_{i_{0}}y_{i_{0}}| $) et 

\[\underline{w}_{\varepsilon,2}(\xx,\yy)=\left(\prod_{k\neq i_{0},j_{0}}\poids(x_{k})\poids(y_{k})\right)\poids(x_{j_{0}})\poidsdeux(x_{j_{0}},y_{i_{0}}), \] o\`u \[ \poidsdeux(y_{i_{0}},x_{j_{0}})=\poids(y_{i_{0}})\left(1-\omeg\left( \frac{|y_{i_{0}}|}{y_{j_{0}}}-\frac{|x_{j_{0}}|}{x_{i_{0}}}\right)\right) \]
(cette fonction permet de prendre en compte  la condition suppl\'ementaire $ |x_{j_{0}}y_{j_{0}}|\leqslant |x_{i_{0}}y_{i_{0}}| $).
On a alors \[ \poids(\xx,\yy)=\underline{w}_{\varepsilon,1}(\xx,\yy)+\underline{w}_{\varepsilon,2}(\xx,\yy),\]
et on peut donc d\'ecomposer l'int\'egrale : \[ I(\aa,\aa')=I_{1}(\aa,\aa')+I_{2}(\aa,\aa'), \]
o\`u \[I_{l}(\aa,\aa')=\sum_{\substack{x_{i}\equiv a_{i}(q), i\neq i_{0} \\ y_{j}\equiv a_{j}'(q), j\neq j_{0}}}\underline{w}_{\varepsilon,l}(\xx,\yy)\int_{\RR^{n+1}}\poids(\uu)h\left(\frac{q}{Q},\frac{\xx.\yy.\uu}{x_{i_{0}}y_{j_{0}}}\right)d\uu .\]

Nous allons traiter la partie de la somme \eqref{TP4} correspondant \`a $ I_{1}(\aa,\aa') $ (la partie correspondant \`a $ I_{2}(\aa,\aa') $ se traitant de mani\`ere analogue, par sym\'etrie). On remplace alors la variable $ u_{j_{0}} $ de l'int\'egrale $ I_{1}(\aa,\aa') $ par $ t=\frac{\xx.\yy.\uu}{x_{i_{0}}y_{j_{0}}}=\frac{x_{j_{0}}}{x_{i_{0}}}u_{j_{0}}+\sum_{k\neq j_{0}}\frac{x_{k}y_{k}u_{k}}{x_{i_{0}}y_{j_{0}}} $. On obtient alors :

\[ I_{1}(\aa,\aa')=x_{i_{0}}\sum_{\substack{x_{i}\equiv a_{i}(q), i\neq i_{0} \\ y_{j}\equiv a_{j}'(q), j\neq j_{0}}}\int_{\RR} \frac{\underline{w}_{\varepsilon,1}(\xx,\yy)}{|x_{j_{0}}|}J(t)h\left(\frac{q}{Q},t\right)dt \] avec \[ J(t)=\int_{\RR^{n}}\poids(\hat{\uu},t)d\hat{\uu} \] ($ \hat{\uu} $ d\'esignant $ \uu $ priv\'ee de la variable $ u_{j_{0}} $) o\`u $ \poids(\hat{\uu},t) $ est la nouvelle fonction poids \begin{multline*} \poids(\hat{\uu},t)=\left(\prod_{k\neq j_{0}}\omeg\left(1-|u_{k}|\right)\omeg\left(|u_{k}|-\frac{1}{P}\right)\right) \\ \omeg\left(1-\frac{x_{i_{0}}}{|x_{j_{0}}|}\left|t-\sum_{k\neq j_{0}}\frac{x_{k}y_{k}u_{k}}{x_{i_{0}}y_{j_{0}}}\right|\right) \omeg\left(\frac{x_{i_{0}}}{|x_{j_{0}}|}\left|t-\sum_{k\neq j_{0}}\frac{x_{k}y_{k}u_{k}}{x_{i_{0}}y_{j_{0}}}\right|-\frac{1}{P}\right) . \end{multline*} 

Nous allons tout d'abord remarquer que l'int\'egrale \[ \int_{\RR} J(t)h\left(\frac{q}{Q},t\right)dt \] est du type $ O(1) $. En effet, on remarque que $ |J(t)|\ll 1 $ et a un support de type $ O(1) $, on en d\'eduit que \[  \int_{\RR} J(t)h\left(\frac{q}{Q},t\right)dt \ll \int_{|t|\ll 1}\left|h\left(\frac{q}{Q},t\right)\right|dt.\] On a par ailleurs pour tout $ |t|\geqslant \frac{q}{Q} $, d'apr\`es \cite[Lemme 5]{HB1}: \[ h\left(\frac{q}{Q},t\right) \ll \left(\frac{q}{Q}\right)^{-1}\left(\left(\frac{q}{Q}\right)^{2}\left(1+\frac{1}{|t|^{2}}\right)\right)   \] et on en d\'eduit que \[ \int_{\frac{q}{Q}\leqslant |t|\ll 1}\left|h\left(\frac{q}{Q},t\right)\right|dt \ll 1. \] Par ailleurs, si $ |t|\leqslant \frac{q}{Q} $, on majore trivialement $ h\left(\frac{q}{Q},t\right) $ par $ \frac{Q}{q} $ (toujours \`a l'aide de \cite[Lemme 5]{HB1}), et on a alors \[ \int_{ |t|\leqslant \frac{q}{Q}}\left|h\left(\frac{q}{Q},t\right)\right|dt \ll \frac{Q}{q}\int_{ |t|\leqslant \frac{q}{Q}}dt \ll 1. \] On a donc bien \'etablit que \begin{equation}\label{ouf}
\int_{\RR} J(t)h\left(\frac{q}{Q},t\right)dt \ll 1.
\end{equation} Ceci va nous permettre entre autres de montrer que l'on peut restreindre la somme sur $ q $ aux entiers $ q\leqslant \min\{x_{i_{0}},y_{j_{0}}\}. $ On suppose dans un premier temps que $ q> \min\{x_{i_{0}},y_{j_{0}}\} $, et puisque l'on a suppos\'e que $ x_{i_{0}}y_{j_{0}}\geqslant q $, on a que $ x_{i_{0}}<q, \; y_{j_{0}}\geqslant \sqrt{q}  $ (cas (1)) ou $ y_{j_{0}}<q, \; x_{i_{0}}\geqslant \sqrt{q}  $ (cas (2)). Consid\'erons d'abord le cas (1) : d'apr\`es la majoration \eqref{ouf}, on peut majorer le terme principal \eqref{TP4} par : \begin{multline*}
B^{n}\sum_{x_{i_{0}},y_{j_{0}}}\frac{\poids(x_{i_{0}},y_{j_{0}})}{x_{i_{0}}^{n}y_{j_{0}}^{n+1}}\sum_{\substack{q\leqslant Q \\ x_{i_{0}}<q}}q^{-n-1}\sum_{\substack{(x_{i})_{i\neq i_{0}} \\  (a_{j}')_{j\neq j_{0}}\in (\ZZ /q\ZZ)^{n}} }\frac{\poids(\xx)}{|x_{j_{0}}|}S_{q}(\xx,\aa')(\0) \\ \sum_{\substack{ y_{j}\equiv a_{j}'(q), j\neq j_{0}}}\underline{w}_{\varepsilon}(\yy)
\end{multline*}
En utilisant le fait que $ y_{j_{0}}\geqslant \sqrt{q} $ : \[ \sum_{\substack{ y_{j}\equiv a_{j}'(q), j\neq j_{0}}}\underline{w}_{\varepsilon}(\yy) \ll \frac{y_{j_{0}}^{n}}{q^{\frac{n}{2}}}. \]
En majorant par ailleurs la partie \[ \sum_{\substack{(x_{i})_{i\neq i_{0}} \\  (a_{j}')_{j\neq j_{0}}\in (\ZZ /q\ZZ)^{n}} }\frac{\poids(\xx)}{|x_{j_{0}}|}S_{q}(\xx,\aa')(\0)\ll \sum_{\substack{(x_{i})_{i\neq i_{0}} \\  (a_{j}')_{j\neq j_{0}}\in (\ZZ /q\ZZ)^{n}} }\poids(\xx)S_{q}(\xx,\aa')(\0) \] par $ x_{i_{0}}^{n}q^{3+n+\epsilon} $ \`a l'aide du lemme \ref{lemmegeneral}, on obtient finalement un terme d'erreur :  \begin{multline*}
B^{n}\sum_{x_{i_{0}},y_{j_{0}}}\frac{\poids(x_{i_{0}},y_{j_{0}})}{y_{j_{0}}}\sum_{\substack{q\leqslant Q \\ x_{i_{0}}<q}}q^{2-\frac{n}{2}+\epsilon} \\ \ll B^{n}\sum_{x_{i_{0}},y_{j_{0}}}\frac{\poids(x_{i_{0}},y_{j_{0}})}{y_{j_{0}}}x_{i_{0}}^{3-\frac{n}{2}+\epsilon} \ll B^{n}\log(B),
\end{multline*}
pour tout $ n\geqslant 9 $. 
Pour le cas (2) le terme principal devient \begin{multline*}
B^{n}\sum_{x_{i_{0}},y_{j_{0}}}\frac{\poids(x_{i_{0}},y_{j_{0}})}{x_{i_{0}}^{n}y_{j_{0}}^{n+1}}\sum_{\substack{q\leqslant Q \\ y_{j_{0}}<q}}q^{-n-1}\sum_{\substack{(y_{j})_{j\neq j_{0}} \\  (a_{i})_{i\neq i_{0}}\in (\ZZ /q\ZZ)^{n}} }\poids(\yy)S_{q}(\xx,\aa')(\0) \\ \sum_{\substack{ x_{i}\equiv a_{i}(q), i\neq i_{0}}}\frac{\underline{w}_{\varepsilon}(\xx)}{|x_{j_{0}}|}.
\end{multline*}
Puisque $ |x_{i_{0}}y_{i_{0}}|\leqslant |x_{j_{0}}y_{j_{0}}| $, on a $ \frac{1}{|x_{j_{0}}|}\leqslant \frac{y_{j_{0}}}{x_{i_{0}}}\frac{1}{|y_{i_{0}}|} $ et la formule ci-dessus peut donc \^etre major\'ee par : \begin{multline*}
B^{n}\sum_{x_{i_{0}},y_{j_{0}}}\frac{\poids(x_{i_{0}},y_{j_{0}})}{x_{i_{0}}^{n}y_{j_{0}}^{n+1}}\sum_{\substack{q\leqslant Q \\ y_{j_{0}}<q}}q^{-n-1}\sum_{\substack{(y_{j})_{j\neq j_{0}} \\  (a_{i})_{i\neq i_{0}}\in (\ZZ /q\ZZ)^{n}} }\frac{\poids(\yy)}{|y_{i_{0}}|}S_{q}(\xx,\aa')(\0) \\ \sum_{\substack{ x_{i}\equiv a_{i}(q), i\neq i_{0}}}\underline{w}_{\varepsilon}(\xx).
\end{multline*}
On effectue alors les majorations : \[  \sum_{\substack{ x_{i}\equiv a_{i}(q), i\neq i_{0}}}\underline{w}_{\varepsilon}(\xx) \ll  \frac{x_{i_{0}}^{n}}{q^{\frac{n}{2}}}, \] et (par le lemme \ref{lemmegeneral}) \[                           \sum_{\substack{(y_{j})_{j\neq j_{0}} \\  (a_{i})_{i\neq i_{0}}\in (\ZZ /q\ZZ)^{n}} }\frac{\poids(\yy)}{|y_{i_{0}}|}S_{q}(\xx,\aa')(\0)\ll y_{j_{0}}^{n}q^{3+n+\epsilon}. \]
On trouve alors comme pour le cas (1) la majoration du terme d'erreur : \begin{equation*}B^{n}\sum_{x_{i_{0}},y_{j_{0}}}\frac{\poids(x_{i_{0}},y_{j_{0}})}{x_{i_{0}}}\sum_{\substack{q\leqslant Q \\ y_{j_{0}}<q}}q^{2-\frac{n}{2}+\epsilon}\ll B^{n}\log(B) \end{equation*} pour tout $ n\geqslant 9$.\\

Par cons\'equent la somme sur les entier $ q $ tels que $ q>\min\{ x_{i_{0}},y_{j_{0}}\} $ donne un terme n\'egligeable par rapport \`a $ B^{n}\log(B)^{2} $. On ne consid\'erera donc dor\'enavant que des entiers $ q\leqslant \min\{ x_{i_{0}},y_{j_{0}}\}  $.\\

Nous allons \`a pr\'esent donner une expression plus pr\'ecise de l'int\'egrale $ \int_{\RR}J(t)h\left(\frac{q}{Q},t\right)dt $. On commence par remarquer que l'on peut supposer que  $ \frac{q}{Q} \leqslant \frac{|x_{j_{0}}|}{x_{i_{0}}} $. En effet, dans le cas contraire, on remarque que la somme \eqref{formulegeneral} fournit un terme n\'egligeable par rapport \`a $ B^{n}\log(B)^{2} $ : si l'on restreint la somme \eqref{TP4} aux entiers $ x_{j_{0}} $ tels que $ \frac{q}{Q} \geqslant \frac{|x_{j_{0}}|}{x_{i_{0}}} $, alors en majorant l'int\'egrale $ \int_{\RR}J(t)h\left(\frac{q}{Q},t\right)dt $ par $ 1 $ (cf. \eqref{ouf}) on trouve \[ I(\aa,\aa')\ll \sum_{\substack{x_{i}\equiv a_{i}(q), i\neq i_{0} \\ y_{j}\equiv a_{j}'(q), j\neq j_{0}}}\frac{\underline{w}_{\varepsilon,1}(\xx,\yy)}{|x_{j_{0}}|}, \]et on obtient une partie du terme principal du type : 

\begin{multline*} B^{n}\sum_{x_{i_{0}},y_{j_{0}}}\frac{\poids(x_{i_{0}},y_{j_{0}})}{x_{i_{0}}^{n}y_{j_{0}}^{n+1}}\sum_{q\leqslant Q}\sum_{\substack{(a_{i})_{i\neq i_{0}}\in (\ZZ /q\ZZ)^{n} \\  (a_{j}')_{j \neq j_{0}}\in (\ZZ /q\ZZ)^{n} } }q^{-n-1}S_{q}(\aa,\aa')(\0) \\ \sum_{\substack{x_{i}\equiv a_{i}(q), i\neq i_{0} \\ y_{j}\equiv a_{j}'(q), j\neq j_{0}}}\frac{\underline{w}_{\varepsilon,1}(\xx,\yy)}{|x_{j_{0}}|}. \end{multline*}

En supposant que $ \frac{|x_{j_{0}}|}{x_{i_{0}}}\leqslant \frac{q}{Q} $, (et de m\^eme $ \frac{|y_{i_{0}}|}{y_{j_{0}}}\leqslant \frac{q}{Q} $ puisque $ |x_{i_{0}}y_{i_{0}}|\leqslant |x_{j_{0}}y_{j_{0}}|  $) on a donc $ |x_{j_{0}}|\leqslant \frac{q}{Q}x_{i_{0}} $ (ainsi que $ |y_{i_{0}}|\leqslant \frac{q}{Q}y_{j_{0}} $), ce qui nous permet de donner la majoration : \begin{multline*} \sum_{\substack{x_{i}\equiv a_{i}(q), i\neq i_{0} \\ y_{j}\equiv a_{j}'(q), j\neq j_{0}}}\frac{\underline{w}_{\varepsilon,1}(\xx,\yy)}{|x_{j_{0}}|}\ll \sum_{\substack{x_{i}\equiv a_{i}(q), i\neq i_{0} \\ y_{j}\equiv a_{j}'(q), j\neq j_{0}}}\underline{w}_{\varepsilon,1}(\xx,\yy) \\ \ll \frac{x_{i_{0}}^{n-1}y_{j_{0}}^{n-1}}{q^{2n-2}}\frac{q^{2}}{Q^{2}}x_{i_{0}}y_{j_{0}}=B^{-1}\frac{x_{i_{0}}^{n}y_{j_{0}}^{n}}{q^{2n-4}}, \end{multline*}
et par le lemme \ref{lemmemoi}, \[ \sum_{\substack{(a_{i})_{i\neq i_{0}}\in (\ZZ /q\ZZ)^{n} \\  (a_{j}')_{j \neq j_{0}}\in (\ZZ /q\ZZ)^{n} } }S_{q}(\aa,\aa')(\0) \ll q^{3+2n+\epsilon}, \] et on obtient finalement un terme d'erreur : \[ B^{n-1}\sum_{x_{i_{0}},y_{j_{0}}}\frac{\poids(x_{i_{0}},y_{j_{0}})}{y_{j_{0}}}\sum_{q\leqslant Q}q^{6-n+\epsilon} \ll B^{n-\frac{1}{3}}, \]
(en rappelant que $ x_{i_{0}}y_{j_{0}} \leqslant B^{\frac{2}{3}} $ sur le support de $ \poids(x_{i_{0}},y_{j_{0}}) $) qui est n\'egligeable par rapport \`a $ B^{n}\log(B)^{2} $. On peut donc supposer que l'on a $ \frac{q}{Q} \leqslant \frac{x_{j_{0}}}{x_{i_{0}}} $. \\

Nous allons \`a pr\'esent majorer $ \int_{\RR}J(t)h\left(\frac{q}{Q},t\right)dt $ l'aide du lemme suivant inspir\'e de \cite[Lemme 9]{HB1} : 

\begin{lemma}\label{TE}
On a pour tout entier $ M>0 $, pour $ \frac{q}{Q}<\frac{|x_{j_{0}}|}{x_{i_{0}}} $ : \[ \int_{\RR}J(t)h\left(\frac{q}{Q},t\right)dt=J(0)+O\left(\left(\frac{x_{i_{0}}}{|x_{j_{0}}|}\frac{q}{Q}\right)^{M}\right). \]
\end{lemma}

\begin{proof}
On pose $ X=\left(\frac{q}{Q}\frac{x_{j_{0}}}{x_{i_{0}}}\right)^{\frac{1}{2}}\}\geqslant \frac{q}{Q} $. D'apr\`es \cite[Lemme 5]{HB1} on a \[ h(\frac{q}{Q},t)\ll \left(\frac{q}{Q}\right)^{N-1}+\left(\frac{q}{Q}\right)^{N-1}\frac{1}{|t|^{N}} \ll \left(\frac{q}{Q}\right)^{N-1}\frac{1}{|t|^{N}} , \]
 pour tout $ t $ tel que $ \frac{q}{Q}\leqslant |X|<|t|\ll 1  $. Etant donn\'e que $ |J(t)|\ll 1 $ ($ \hat{\uu} \mt \poids(\hat{\uu},t) $ \'etant born\'e sur un support de mesure $ O(1) $) et de support $ O(1) $, on a que :
 \begin{multline*} \int_{|t|\geqslant X}J(t)h\left(\frac{q}{Q},t\right)dt \ll  \int_{ X\leqslant |t| \ll 1}\left|h\left(\frac{q}{Q},t\right)\right|dt  \\ \ll  \int_{ X\leqslant |t| \ll 1}\left(\frac{q}{Q}\right)^{N-1}\frac{1}{|t|^{N}}dt \\ \ll \left(\left(\frac{q}{Q}\right) X^{-1}\right)^{N-1}  = \left(\frac{x_{i_{0}}}{|x_{j_{0}}|}\frac{q}{Q}\right)^{\frac{N-1}{2}},\end{multline*}
 
et on obtient un $ O\left(\left(\frac{x_{i_{0}}}{|x_{j_{0}}|}\frac{q}{Q}\right)^{M}\right) $ pour cette partie de l'int\'egrale, en prenant $ N=2M+1 $.\\

Pour $ |t|\leqslant X $, on d\'eveloppe $ J(t) $ en s\'erie de Taylor au voisinage de $ 0 $. En remarquant que pour tout $ k>0 $, $ \frac{d^{k}}{dt^{k}}\poids(\hat{\uu},t) \ll \left(\frac{x_{i_{0}}}{|x_{j_{0}}|}\right)^{k} $ , on voit que la d\'ecomposition en s\'erie de Taylor de $ J(t) $ \`a l'ordre $ 2M $ donne un polyn\^ome en $ t $ de degr\'e $ 2M $ et un terme d'erreur $ O\left( \left(\frac{x_{i_{0}}}{|x_{j_{0}}|}X\right)^{2M+1}\right) $. Puisque $ h(\frac{q}{Q},t)\ll \left(\frac{Q}{q}\right) $ ce terme d'erreur donne une contribution \[ O\left(\frac{Q}{q}\left(\frac{x_{i_{0}}}{|x_{j_{0}}|}\right)^{2M+1}X^{2M+2}\right)=O\left(\left(\frac{q}{Q}\frac{x_{i_{0}}}{|x_{j_{0}}|}\right)^{M}\right) \] pour l'int\'egrale consid\'er\'ee. Le terme principal de la d\'ecomposition de Taylor est de la forme : \[ J(0)+\sum_{k=1}^{2M}\alpha_{k}(\hat{\uu})\left(\frac{x_{i_{0}}}{|x_{j_{0}}|}\right)^{k}t^{k} \] o\`u $ \alpha_{k}(\hat{\uu}) $ est une fonction poids \`a valeurs $ O(1) $ sur un support de type $ O(1) $. On remarque, en utilisant \cite[Lemme 8]{HB1}, que pour tout $ k \geqslant 1 $ et pour tout entier $ N>0 $, on a (en rappelant que $ \frac{q}{Q}\leqslant |X|\ll 1 $) : \begin{align*}\left(\frac{x_{i_{0}}}{|x_{j_{0}}|}\right)^{k}\int_{-X}^{X}t^{k}h(\frac{q}{Q},t)dt &  \ll \left(\frac{x_{i_{0}}}{|x_{j_{0}}|}\right)^{k}X^{k}\left(X\left(\frac{q}{Q}\right)^{N-1}+\left(\frac{q}{Q}X^{-1}\right)^{N}\right) \\ & \ll \underbrace{\left(\frac{q}{Q}\frac{x_{i_{0}}}{|x_{j_{0}}|}\right)^{\frac{k}{2}}}_{\leqslant 1}\left(\frac{q}{Q}\frac{x_{i_{0}}}{|x_{j_{0}}|}\right)^{\frac{N}{2}} \\ & \ll \left(\frac{q}{Q}\frac{x_{i_{0}}}{|x_{j_{0}}|}\right)^{M}, \end{align*}
en choisissant $ N=2M $.\\

Enfin, en utilisant \cite[Lemme 6]{HB1}, on remarque que le terme principal de la d\'ecomposition de Taylor nous donne (pour tout $ N>0 $) : \begin{align*}
J(0)\int_{-X}^{X}h(\frac{q}{Q},t)dt & = J(0)\left(1+O\left(X\left(\frac{q}{Q}\right)^{N-1}+\left(\frac{q}{Q}X^{-1}\right)^{N}\right)\right) \\ & =J(0)+O\left(\left(\frac{q}{Q}\frac{x_{i_{0}}}{|x_{j_{0}}|}\right)^{\frac{N}{2}}\right).
\end{align*}
On obtient alors le r\'esultat en prenant $ N=2M $.

\end{proof}
Nous allons \`a pr\'esent montrer que le terme d'erreur issu de ce lemme donne dans \eqref{TP4} un terme n\'egligeable par rapport \`a $ B^{n}\log(B)^{2} $. Le terme d'erreur du lemme \ref{TE} appliqu\'e avec $ M=1 $ vaut alors \[ O\left(\frac{x_{i_{0}}}{|x_{j_{0}}|}\frac{q}{Q}\right) =O\left(\frac{y_{j_{0}}}{|y_{i_{0}}|}\frac{q}{Q}\right) \] en utilisant le fait que $ |x_{i_{0}}y_{i_{0}}|\leqslant |x_{j_{0}}y_{j_{0}}| $. 

On obtient une contribution :   \begin{multline}\label{partie1}
 B^{n-\frac{1}{2}}\sum_{x_{i_{0}},y_{j_{0}}}\frac{\poids(x_{i_{0}},y_{j_{0}})}{x_{i_{0}}^{n}y_{j_{0}}^{n}}\sum_{q\leqslant Q}\sum_{\substack{(a_{i})_{i\neq i_{0}}\in (\ZZ /q\ZZ)^{n} \\  (a_{j}')_{j \neq j_{0}}\in (\ZZ /q\ZZ)^{n} } }q^{-n}S_{q}(\aa,\aa')(\0) \\ \sum_{\substack{x_{i}\equiv a_{i}(q) \\ i\neq i_{0} }}\sum_{ \substack{y_{j}\equiv a_{j}'(q) \\ j\neq j_{0}}}\frac{\poidsun(\xx,\yy)}{|x_{j_{0}}||y_{i_{0}}|}. \end{multline}
On remarque que \begin{align*}
\sum_{\substack{x_{i}\equiv a_{i}(q), i\neq i_{0} \\ y_{j}\equiv a_{j}'(q), j\neq j_{0} }}\frac{\poidsun(\xx,\yy)}{|y_{i_{0}}x_{j_{0}}|} \\  & =\underbrace{\sum_{\substack{x_{j_{0}}\equiv a_{j_{0}}(q)\\y_{i_{0}}\equiv a_{i_{0}}'(q) }}\frac{\poids(y_{i_{0}})\poidsun(x_{j_{0}})}{|x_{j_{0}}y_{i_{0}}|}}_{=O(q^{-2}\log(x_{i_{0}})\log(y_{j_{0}}))} \\ & \underbrace{\prod_{i\neq i_{0},j_{0}}\sum_{x_{i}\equiv a_{i}(q)}\poids(x_{i_{0}},x_{i})}_{=O(x_{i_{0}}^{n-1}q^{-n+1})}\underbrace{\prod_{j\neq i_{0},j_{0}}\sum_{y_{j}\equiv a_{j}'(q)}\poids(y_{j_{0}},y_{j}) }_{=O(y_{j_{0}}^{n-1}q^{-n+1})}.
\end{align*}

Par cons\'equent, \[ \sum_{\substack{x_{i}\equiv a_{i}(q), i\neq i_{0} \\ y_{j}\equiv a_{j}'(q), j\neq j_{0} }}\frac{\poidsun(\xx,\yy)}{|x_{j_{0}}||y_{i_{0}}|}=O(x_{i_{0}}^{n-1}y_{j_{0}}^{n-1}\log(x_{i_{0}})\log(y_{j_{0}})q^{-2n}), \] et en utilisant \`a nouveau le lemme \ref{lemmemoi}, on trouve : \begin{align*} \eqref{partie1} & \ll B^{n-\frac{1}{2}}\sum_{x_{i_{0}},y_{j_{0}}}\poids(x_{i_{0}},y_{j_{0}})\frac{\log(x_{i_{0}})\log(y_{j_{0}})}{x_{i_{0}}y_{j_{0}}}\sum_{q\leqslant Q}q^{3-n+\epsilon} \\ & \ll B^{n-\frac{1}{2}}\sum_{x_{i_{0}},y_{j_{0}}}\poids(x_{i_{0}},y_{j_{0}})\frac{\log(x_{i_{0}})\log(y_{j_{0}})}{x_{i_{0}}y_{j_{0}}} \ll B^{n-\frac{1}{2}}\log(B)^{4}. 
\end{align*}\\
On \'etudie \`a pr\'esent le terme principal :

\begin{multline}\label{numero} B^{n}\sum_{x_{i_{0}},y_{j_{0}}}\frac{\poids(x_{i_{0}},y_{j_{0}})}{x_{i_{0}}^{n}y_{j_{0}}^{n+1}}\sum_{q\leqslant Q}\sum_{\substack{(a_{i})_{i\neq i_{0}}\in (\ZZ /q\ZZ)^{n} \\  (a_{j}')_{j \neq j_{0}}\in (\ZZ /q\ZZ)^{n} } }q^{-n-1}S_{q}(\aa,\aa')(\0) \\ \sum_{\substack{x_{i}\equiv a_{i}(q), i\neq i_{0} \\ y_{j}\equiv a_{j}'(q), j\neq j_{0}}}\frac{\underline{w}_{\varepsilon,1}(\xx,\yy)}{|x_{j_{0}}|}\int_{\RR^{n}}\poids(0,\hat{\uu})d\hat{\uu}.\end{multline}

Dans tout ce qui va suivre nous allons supposer, pour simplifier, que les entiers $ x_{j_{0}} $ et $ y_{i_{0}} $ sont tous deux strictement positifs. Nous voudrions pouvoir r\'e\'ecrire la somme \[ \sum_{\substack{x_{i}\equiv a_{i}(q), i\neq i_{0} \\ y_{j}\equiv a_{j}'(q), j\neq j_{0}}}\frac{\underline{w}_{\varepsilon,1}(\xx,\yy)}{x_{j_{0}}} \] sous forme d'une int\'egrale. Remarquons avant tout que : \begin{multline*} \sum_{\substack{x_{i}\equiv a_{i}(q), i\neq i_{0} \\ y_{j}\equiv a_{j}'(q), j\neq j_{0}}}\frac{\underline{w}_{\varepsilon,1}(\xx,\yy)}{x_{j_{0}}}  =\left(\sum_{y_{i_{0}}\equiv a_{i_{0}}'(q)}\poids(y_{i_{0}})\sum_{x_{j_{0}}\equiv a_{j_{0}}(q)}\frac{\poidsun(y_{i_{0}},x_{j_{0}})}{x_{j_{0}}}\right)\\ \left(\prod_{i\neq i_{0},j_{0}}\sum_{x_{i}\equiv a_{i} (q)}\poids(x_{i})\right) \left(\prod_{j\neq i_{0},j_{0}}\sum_{y_{j}\equiv a_{j}' (q)}\poids(y_{j})\right).\end{multline*}
De la m\^eme mani\`ere que pour le cas $ i_{0}=j_{0} $, on \'etablit que pour tout $ i,j\notin \{i_{0},j_{0}\} $ :
\[  \sum_{x_{i}\equiv a_{i} (q)}\poids(x_{i})=\underbrace{\frac{1}{q}\int_{\RR}\poids(x_{i})dx_{i}}_{=O(q^{-1}x_{i_{0}})}+O(1), \] 
\[  \sum_{y_{j}\equiv a_{j}' (q)}\poids(y_{j})=\underbrace{\frac{1}{q}\int_{\RR}\poids(y_{j})dy_{j}}_{=O(q^{-1}y_{j_{0}})}+O(1). \] 
On a donc : \begin{multline*} \left(\prod_{i\neq i_{0},j_{0}}\sum_{x_{i}\equiv a_{i} (q)}\poids(x_{i})\right) \left(\prod_{j\neq i_{0},j_{0}}\sum_{y_{j}\equiv a_{j}' (q)}\poids(y_{j})\right) \\ =\underbrace{\frac{1}{q^{2n-2}}\prod_{k\neq i_{0},j_{0}}\int_{\RR^{2}}\poids(x_{k})\poids(y_{k})dx_{k}dy_{k}}_{=O(q^{-2n+2}x_{i_{0}}^{n-1}y_{j_{0}}^{n-1})}
+O(q^{-2n+3}x_{i_{0}}^{n-2}y_{j_{0}}^{n-2}\max\{x_{i_{0}}, y_{j_{0}} \}). \end{multline*}

La partie $ O(q^{-2n+3}x_{i_{0}}^{n-2}y_{j_{0}}^{n-2}\max\{x_{i_{0}}, y_{j_{0}} \}) $ fournit une contribution n\'egligeable par rapport \`a $ B^{n}\log(B)^{2} $.

En effet, lorsque $ x_{i_{0}}\leqslant y_{j_{0}}  $, on obtient : \begin{multline*}B^{n}\sum_{x_{i_{0}},y_{j_{0}}}\frac{\poids(x_{i_{0}},y_{j_{0}})}{x_{i_{0}}^{2}y_{j_{0}}^{2}}\sum_{q\leqslant Q}\sum_{\substack{(a_{i})_{i\neq i_{0}}\in (\ZZ /q\ZZ)^{n} \\  (a_{j}')_{j \neq j_{0}}\in (\ZZ /q\ZZ)^{n} } }q^{-3n+2}S_{q}(\aa,\aa')(\0) \\ \underbrace{\sum_{y_{i_{0}}\equiv a_{i_{0}}'(q)}\poids(y_{i_{0}})\sum_{x_{j_{0}}\equiv a_{j_{0}}(q)}\frac{\poidsun(y_{i_{0}},x_{j_{0}})}{x_{j_{0}}} }_{=O(q^{-2}y_{j_{0}}\log(x_{i_{0}}))} \\ \ll B^{n}\sum_{x_{i_{0}},y_{j_{0}}}\poids(x_{i_{0}},y_{j_{0}})\frac{\log(x_{i_{0}})}{x_{i_{0}}^{2}y_{j_{0}}}\underbrace{\sum_{q\leqslant Q}\sum_{\substack{(a_{i})_{i\neq i_{0}}\in (\ZZ /q\ZZ)^{n} \\  (a_{j}')_{j \neq j_{0}}\in (\ZZ /q\ZZ)^{n} } }q^{-3n}S_{q}(\aa,\aa')(\0)}_{=O(1) \; \dapres \; \lle \; \lemme \; \ref{lemmemoi}} \\ \ll B^{n}\sum_{x_{i_{0}},y_{j_{0}}}\poids(x_{i_{0}},y_{j_{0}})\frac{\log(x_{i_{0}})}{x_{i_{0}}^{2}y_{j_{0}}} \ll B^{n}\log(B). \end{multline*} 

Lorsque $ x_{i_{0}}\geqslant y_{j_{0}}  $, en utilisant la condition $ x_{i_{0}}y_{i_{0}} \leqslant x_{j_{0}}y_{j_{0}} $, on trouve : 

\begin{multline*}B^{n}\sum_{x_{i_{0}},y_{j_{0}}}\frac{\poids(x_{i_{0}},y_{j_{0}})}{x_{i_{0}}^{2}y_{j_{0}}^{2}}\sum_{q\leqslant Q}\sum_{\substack{(a_{i})_{i\neq i_{0}}\in (\ZZ /q\ZZ)^{n} \\  (a_{j}')_{j \neq j_{0}}\in (\ZZ /q\ZZ)^{n} } }q^{-3n+2}S_{q}(\aa,\aa')(\0) \\ \underbrace{\sum_{y_{i_{0}}\equiv a_{i_{0}}'(q)}\poids(y_{i_{0}})\sum_{x_{j_{0}}\equiv a_{j_{0}}(q)}\frac{\poidsun(y_{i_{0}},x_{j_{0}})}{y_{i_{0}}} }_{=O(q^{-2}x_{i_{0}}\log(y_{j_{0}}))} \\ \ll B^{n}\sum_{x_{i_{0}},y_{j_{0}}}\poids(x_{i_{0}},y_{j_{0}})\frac{\log(y_{j_{0}})}{x_{i_{0}}y_{j_{0}}^{2}} \ll B^{n}\log(B). \end{multline*}

Le terme principal est donc donn\'e par : 

\begin{multline*}B^{n}\sum_{x_{i_{0}},y_{j_{0}}}\frac{\poids(x_{i_{0}},y_{j_{0}})}{x_{i_{0}}^{n}y_{j_{0}}^{n+1}}\sum_{q\leqslant Q}\sum_{\substack{(a_{i})_{i\neq i_{0}}\in (\ZZ /q\ZZ)^{n} \\  (a_{j}')_{j \neq j_{0}}\in (\ZZ /q\ZZ)^{n}  } }q^{-3n+1}S_{q}(\aa,\aa')(\0) \\\sum_{y_{i_{0}}\equiv a_{i_{0}}'(q)}\poids(y_{i_{0}})\sum_{x_{j_{0}}\equiv a_{j_{0}}(q)}\frac{\poidsun(y_{i_{0}},x_{j_{0}})}{x_{j_{0}}} \prod_{k\neq i_{0},j_{0}}\int_{\RR^{2}}\poids(x_{k})\poids(y_{k})dx_{k}dy_{k}\int_{\RR^{n}}\poids(0,\hat{\uu})d\hat{\uu},\end{multline*} avec \begin{equation}\label{grossier}
\prod_{k\neq i_{0},j_{0}}\int_{\RR^{2}}\poids(x_{k})\poids(y_{k})dx_{k}dy_{k}=O(x_{i_{0}}^{n-1}y_{j_{0}}^{n-1}).
\end{equation}
Remarquons \`a pr\'esent qu'\`a $ y_{i_{0}} $ fix\'e, la formule sommatoire d'Euler \eqref{euler} donne : 
 \begin{multline*} \sum_{x_{j_{0}}\equiv a_{j_{0}}}\frac{\poidsun(y_{i_{0}},x_{j_{0}})}{x_{j_{0}}}=\frac{1}{q}\int_{\RR}\frac{\poidsun(y_{i_{0}},x_{j_{0}})}{x_{j_{0}}}dx_{j_{0}} \\ +\int_{\RR}(z_{j_{0}}-\lfloor z_{j_{0}} \rfloor) \frac{\partial}{\partial z_{j_{0}}}\left(\frac{\poidsun(y_{i_{0}},a_{j_{0}}+qz_{j_{0}})}{a_{j_{0}}+qz_{j_{0}}}\right)dz_{j_{0}}.  \end{multline*}
 Or, \begin{multline*} \left|\frac{\partial}{\partial z_{j_{0}}}\left(\frac{\poidsun(y_{i_{0}},a_{j_{0}}+qz_{j_{0}})}{a_{j_{0}}+qz_{j_{0}}}\right)\right|  =
|- \frac{q}{x_{j_{0}}x_{i_{0}}}\omeg'(1-\frac{x_{j_{0}}}{x_{i_{0}}})\omeg(\frac{x_{j_{0}}}{x_{i_{0}}}-\frac{y_{i_{0}}}{y_{j_{0}}})\omeg(\frac{x_{j_{0}}}{x_{i_{0}}}-\frac{1}{x_{i_{0}}}) \\+\frac{q}{x_{j_{0}}x_{i_{0}}}\omeg(1-\frac{x_{j_{0}}}{x_{i_{0}}})\omeg'(\frac{x_{j_{0}}}{x_{i_{0}}}-\frac{y_{i_{0}}}{y_{j_{0}}})\omeg(\frac{x_{j_{0}}}{x_{i_{0}}}-\frac{1}{x_{i_{0}}})\\ + \frac{q}{x_{j_{0}}x_{i_{0}}}\omeg(1-\frac{x_{j_{0}}}{x_{i_{0}}})\omeg(\frac{x_{j_{0}}}{x_{i_{0}}}-\frac{y_{i_{0}}}{y_{j_{0}}})\omeg'(\frac{x_{j_{0}}}{x_{i_{0}}}-\frac{1}{x_{i_{0}}}) +\frac{q}{x_{j_{0}}^{2}}\poidsun(y_{i_{0}},x_{j_{0}})|, 
 \end{multline*} avec $ x_{j_{0}}=a_{j_{0}}+qz_{j_{0}} $. Ceci est donc du type $ O\left(\frac{q}{x_{j_{0}}x_{i_{0}}} +\frac{q}{x_{j_{0}}^{2}}\right) $ et \`a un support inclus dans $ [\frac{1}{q}(\frac{x_{i_{0}}y_{i_{0}}}{y_{j_{0}}}-a_{j_{0}}),\frac{1}{q}(x_{i_{0}}-a_{j_{0}})] $, on a alors, par changement de variables $ a_{j_{0}}+qz_{j_{0}}=x_{j_{0}} $ : 
\[ \left|\int_{\RR}(z_{j_{0}}-\lfloor z_{j_{0}} \rfloor) \frac{\partial}{\partial z_{j_{0}}}\left(\frac{\poidsun(y_{i_{0}},a_{j_{0}}+qz_{j_{0}})}{a_{j_{0}}+qz_{j_{0}}}\right)dz_{j_{0}} \right| \ll \underbrace{\int_{1}^{x_{i_{0}}}\frac{1}{x_{i_{0}}x_{j_{0}}}dx_{j_{0}}}_{(\alpha)} + \underbrace{\int_{1}^{x_{i_{0}}}\frac{1}{x_{j_{0}}^{2}}dx_{j_{0}}}_{(\beta)}. \]

La partie $ (\alpha) $ est du type $ O\left(\frac{\log(x_{i_{0}})}{x_{i_{0}}}\right) $ et donne donc une contribution (en tenant compte de l'estimation triviale \eqref{grossier}) : \begin{multline*} 
B^{n}\sum_{x_{i_{0}},y_{j_{0}}}\poids(x_{i_{0}},y_{j_{0}})\frac{\log(x_{i_{0}})}{x_{i_{0}}^{2}y_{j_{0}}^{2}}\sum_{q\leqslant Q}\sum_{\substack{(a_{i})_{i\neq i_{0}}\in (\ZZ /q\ZZ)^{n} \\  (a_{j}')_{j \neq j_{0}}\in (\ZZ /q\ZZ)^{n}} }q^{-3n+1}S_{q}(\aa,\aa')(\0) \underbrace{\sum_{y_{i_{0}}\equiv a_{i_{0}}'(q)}\poids(y_{i_{0}})}_{=O(q^{-1}y_{j_{0}})} \\ \ll B^{n}\sum_{x_{i_{0}},y_{j_{0}}}\poids(x_{i_{0}},y_{j_{0}})\frac{\log(x_{i_{0}})}{x_{i_{0}}^{2}y_{j_{0}}} \ll B^{n}\log(B). \end{multline*}

Pour traiter la partie $ (\beta) $, nous allons devoir distinguer les cas $ x_{j_{0}}<\sqrt{x_{i_{0}}} $ et $ x_{j_{0}}\geqslant \sqrt{x_{i_{0}}} $. Consid\'erons d'abord le cas o\`u $ x_{j_{0}}<\sqrt{x_{i_{0}}} $ l'in\'egalit\'e $ x_{i_{0}}y_{i_{0}}\leqslant x_{j_{0}}y_{j_{0}} $ nous donne alors $ y_{i_{0}}\leqslant x_{i_{0}}^{-\frac{1}{2}}y_{j_{0}} $. La partie du terme principal correspondant au cas o\`u $ x_{j_{0}}<\sqrt{x_{i_{0}}} $ vaut alors (en majorant $ (\beta) $ par $ 1 $) : \begin{multline*} B^{n}\sum_{x_{i_{0}},y_{j_{0}}}\frac{\poids(x_{i_{0}},y_{j_{0}})}{x_{i_{0}}y_{j_{0}}^{2}}\sum_{q\leqslant Q}\sum_{\substack{(a_{i})_{i\neq i_{0}}\in (\ZZ /q\ZZ)^{n} \\  (a_{j}')_{j \neq j_{0}}\in (\ZZ /q\ZZ)^{n}}}q^{-3n+1} S_{q}(\aa,\aa')(\0) \underbrace{\sum_{\substack{y_{i_{0}}\equiv a_{i_{0}}'(q)\\ y_{i_{0}}\leqslant x_{i_{0}}^{-\frac{1}{2}}y_{j_{0}}  }}\poids(y_{i_{0}})}_{=O(x_{i_{0}}^{-\frac{1}{2}}y_{j_{0}})} \\ \ll B^{n}\sum_{x_{i_{0}},y_{j_{0}}}\frac{\poids(x_{i_{0}},y_{j_{0}})}{x_{i_{0}}^{\frac{3}{2}}y_{j_{0}}} \ll B^{n}\log(B) \end{multline*}

Consid\'erons \`a pr\'esent la partie correspondant \`a $ x_{j_{0}}\geqslant \sqrt{x_{i_{0}}} $. Dans ce cas, le terme $ (\beta) $ vaut $ O(x_{i_{0}}^{-\frac{1}{2}}) $, et on obtient une contribution :

\begin{multline*}B^{n}\sum_{x_{i_{0}},y_{j_{0}}}\poids(x_{i_{0}},y_{j_{0}})\frac{\log(x_{i_{0}})}{x_{i_{0}}^{\frac{3}{2}}y_{j_{0}}^{2}}\sum_{q\leqslant Q}\sum_{\substack{(a_{i})_{i\neq i_{0}}\in (\ZZ /q\ZZ)^{n} \\  (a_{j}')_{j \neq j_{0}}\in (\ZZ /q\ZZ)^{n}} }q^{-3n+1}S_{q}(\aa,\aa')(\0) \underbrace{\sum_{y_{i_{0}}\equiv a_{i_{0}}'(q)}\poids(y_{i_{0}})}_{=O(q^{-1}y_{j_{0})}} \\ \ll B^{n}\sum_{x_{i_{0}},y_{j_{0}}}\frac{\poids(x_{i_{0}},y_{j_{0}})}{x_{i_{0}}^{\frac{3}{2}}y_{j_{0}}} \ll B^{n}\log(B).  \end{multline*}

Le terme principal est ainsi donn\'e par \begin{multline*}B^{n}\sum_{x_{i_{0}},y_{j_{0}}}\frac{\poids(x_{i_{0}},y_{j_{0}})}{x_{i_{0}}^{n}y_{j_{0}}^{n+1}}\sum_{q\leqslant Q}\sum_{\substack{(a_{i})_{i\neq i_{0}}\in (\ZZ /q\ZZ)^{n} \\  (a_{j}')_{j \neq j_{0}}\in (\ZZ /q\ZZ)^{n}} }q^{-3n}S_{q}(\aa,\aa')(\0)  \sum_{y_{i_{0}}\equiv a_{i_{0}}'(q)}\poids(y_{i_{0}})\\ \int_{\RR}\frac{\poidsun(y_{i_{0}},x_{j_{0}})}{x_{j_{0}}}dx_{j_{0}} \prod_{k\neq i_{0},j_{0}}\int_{\RR^{2}}\poids(x_{k})\poids(y_{k})dx_{k}dy_{k}\int_{\RR^{n}}\poids(0,\hat{\uu})d\hat{\uu},\end{multline*}

Il nous reste \`a exprimer la somme : 

\[\sum_{y_{i_{0}}\equiv a_{i_{0}}'(q)}\poids(y_{i_{0}})\int_{\RR}\frac{\poidsun(y_{i_{0}},x_{j_{0}})}{x_{j_{0}}}dx_{j_{0}}=\int_{\RR}\frac{1}{x_{j_{0}}}\sum_{y_{i_{0}}\equiv a_{i_{0}}'(q)}\poids(y_{i_{0}})\poidsun(y_{i_{0}},x_{j_{0}})dx_{j_{0}}\] sous la forme d'une int\'egrale.
Or, toujours par application de la formule d'Euler, et \'etant donn\'e que \[ \poidsun(y_{i_{0}},x_{j_{0}})=\poids(x_{j_{0}})\omeg\left( \frac{x_{j_{0}}}{x_{i_{0}}}-\frac{y_{i_{0}}}{y_{j_{0}}}\right), \] on a \begin{multline*}\sum_{y_{i_{0}}\equiv a_{i_{0}}'(q)}\poids(y_{i_{0}})\poidsun(y_{i_{0}},x_{j_{0}})=\frac{1}{q}\int_{\RR}\poids(y_{i_{0}})\poidsun(y_{i_{0}},x_{j_{0}})dy_{i_{0}} \\ +\poids(x_{j_{0}})\int_{\RR}(z_{i_{0}}-\lfloor z_{i_{0}} \rfloor) \frac{\partial}{\partial z_{i_{0}}}(\poids(y_{i_{0}})\omeg\left( \frac{x_{j_{0}}}{x_{i_{0}}}-\frac{y_{i_{0}}}{y_{j_{0}}}\right))dz_{i_{0}} \end{multline*}
avec $ y_{i_{0}}=a_{i_{0}}'+qz_{i_{0}} $. Or, on v\'erifie de la m\^eme mani\`ere que pr\'ec\'edemment que la d\'eriv\'ee $ \frac{\partial}{\partial z_{i_{0}}}\left(\poids(y_{i_{0}})\omeg\left( \frac{x_{j_{0}}}{x_{i_{0}}}-\frac{y_{i_{0}}}{y_{j_{0}}}\right)\right) $ est du type $ O(\frac{q}{y_{j_{0}}}) $ et a un support inclus dans $ [\frac{1}{q}(1-a_{i_{0}}'),\frac{1}{q}(\frac{x_{j_{0}}y_{j_{0}}}{x_{i_{0}}}-a_{i_{0}}')] $. Donc, par changement de variable $ y_{i_{0}}=a_{i_{0}}'+qz_{i_{0}} $ on trouve : 
\begin{multline*}
\left|\int_{\RR}\frac{\poids(x_{j_{0}})}{x_{j_{0}}}\int_{\RR}(z_{i_{0}}-\lfloor z_{i_{0}} \rfloor) \frac{\partial}{\partial z_{i_{0}}}(\poids(y_{i_{0}})\omeg\left( \frac{x_{j_{0}}}{x_{i_{0}}}-\frac{y_{i_{0}}}{y_{j_{0}}}\right))dz_{i_{0}}dx_{j_{0}}\right| \\  \ll \int_{1}^{x_{i_{0}}}\frac{1}{x_{j_{0}}}\int_{1}^{\frac{x_{j_{0}}y_{j_{0}}}{x_{i_{0}}}}\frac{1}{y_{j_{0}}}dy_{i_{0}}dx_{j_{0}} \ll \frac{1}{y_{j_{0}}}\int_{1}^{x_{i_{0}}}\frac{1}{x_{j_{0}}}\frac{x_{j_{0}}y_{j_{0}}}{x_{i_{0}}}dx_{j_{0}}=O(1)
\end{multline*}
Ceci donne donc une contribution \begin{multline*} B^{n}\sum_{x_{i_{0}},y_{j_{0}}}\frac{\poids(x_{i_{0}},y_{j_{0}})}{x_{i_{0}}y_{j_{0}}^{2}}\sum_{q\leqslant Q}\sum_{\substack{(a_{i})_{i\neq i_{0}}\in (\ZZ /q\ZZ)^{n} \\  (a_{j}')_{j \neq j_{0}}\in (\ZZ /q\ZZ)^{n}} }q^{-3n}S_{q}(\aa,\aa')(\0) \\ \ll B^{n}\sum_{x_{i_{0}},y_{j_{0}}}\frac{\poids(x_{i_{0}},y_{j_{0}})}{x_{i_{0}}y_{j_{0}}^{2}} \ll B^{n}\log(B).\end{multline*}

Donc, finalement le terme principal vaut : 
\begin{multline*}B^{n}\sum_{x_{i_{0}},y_{j_{0}}}\frac{\poids(x_{i_{0}},y_{j_{0}})}{x_{i_{0}}^{n}y_{j_{0}}^{n+1}}\sum_{q\leqslant Q}\sum_{\substack{(a_{i})_{i\neq i_{0}}\in (\ZZ /q\ZZ)^{n} \\  (a_{j}')_{j \neq j_{0}}\in (\ZZ /q\ZZ)^{n}} }q^{-3n-1}S_{q}(\aa,\aa')(\0) \\ \int_{\RR^{3n}}\frac{\underline{w}_{\varepsilon,1}(\xx,\yy)}{|x_{j_{0}}|}\poids(0,\hat{\uu})d\xx'd\yy'd\hat{\uu}.\end{multline*}
o\`u $ \xx' $ (resp. $ \yy' $) d\'esignent les variables $ \xx $ et $ \yy $ priv\'ees de $ x_{i_{0}} $ et $ y_{j_{0} } $. On effectue un changement de variables $ \xx=x_{i_{0}}\ss $ et $ \yy=y_{j_{0}}\tt $ et on obtient ce terme sous la forme :

\begin{multline*}B^{n}\sum_{x_{i_{0}},y_{j_{0}}}\frac{\poids(x_{i_{0}},y_{j_{0}})}{x_{i_{0}}y_{j_{0}}}\sum_{q\leqslant Q}\sum_{\substack{(a_{i})_{i\neq i_{0}}\in (\ZZ /q\ZZ)^{n} \\  (a_{j}')_{j \neq j_{0}}\in (\ZZ /q\ZZ)^{n}} }q^{-3n-1}S_{q}(\aa,\aa')(\0) \\ \int_{\RR^{3n}}\frac{\underline{w}_{\varepsilon,1}(\ss,\tt)}{|s_{j_{0}}|}\poids(0,\hat{\uu})d\ss'd\tt'd\hat{\uu}.\end{multline*}

\begin{rem}
Par sym\'etrie, on obtient exactement les m\^emes r\'esultats avec la partie du terme principal correspondant \`a la fonction poids $ \underline{w}_{\varepsilon,2}(\xx,\yy) $, et le terme principal obtenu est alors 
\begin{multline*}B^{n}\sum_{x_{i_{0}},y_{j_{0}}}\frac{\poids(x_{i_{0}},y_{j_{0}})}{x_{i_{0}}y_{j_{0}}}\sum_{q\leqslant Q}\sum_{\substack{(a_{i})_{i\neq i_{0}}\in (\ZZ /q\ZZ)^{n} \\  (a_{j}')_{j \neq j_{0}}\in (\ZZ /q\ZZ)^{n}} }q^{-3n-1}S_{q}(\aa,\aa')(\0) \\ \int_{\RR^{3n}}\frac{\underline{w}_{\varepsilon,2}(\ss,\tt)}{|t_{i_{0}}|}\poids(0,\hat{\uu})d\ss'd\tt'd\hat{\uu}.\end{multline*}
\end{rem}

On note ensuite \[ \tilde{\chi}_{i_{0},j_{0}}^{(1)}(\ss',\tt',\uu)=\left\{\begin{array}{lllll}
1 & \mbox{si} & |\ss'|\in [\frac{1}{x_{i_{0}}},1],|\tt'|\in [\frac{1}{y_{j_{0}}},1],|\hat{\uu}|\in [\frac{1}{P}, 1] & & \\ & \mbox{et} & \left| \sum_{k\neq j_{0}}\frac{x_{k}y_{k}u_{k}}{x_{j_{0}}y_{j_{0}}}\right|= \left| \sum_{k\neq j_{0}}\frac{s_{k}t_{k}u_{k}}{s_{j_{0}}}\right| \in [\frac{1}{P},1] \; \mbox{et} \;   |t_{i_{0}}|\leqslant |s_{j_{0}}|&   &\\ 0  &\mbox{sinon}&
\end{array}\right. \]
On constate que $ \underline{w}_{\varepsilon}(\ss,\tt)\poids(\uu)\ra  \chi_{i_{0},j_{0}}(\ss',\tt',\hat{\uu}) $ lorsque $\varepsilon\ra 0 $, et qu'ainsi on a finalement  \[
 \int_{\RR^{3n}}\frac{\underline{w}_{\varepsilon,1}(\ss,\tt)}{|s_{j_{0}}|}\poids(0,\hat{\uu})d\ss'd\tt'd\hat{\uu}\\=\sigma_{i_{0},j_{0}}^{(1)}+O(\varepsilon) \] o\`u 
\[ \tilde{\sigma}_{i_{0},j_{0}}^{(1)}=\int_{\RR^{3n}}\frac{\tilde{\chi}_{i_{0},j_{0}}^{(1)}(\ss',\tt',\hat{\uu})}{|s_{j_{0}}|}d\ss' d\tt' d\hat{\uu}. \]

Nous voudrions pouvoir remplacer la constante $ \tilde{\sigma}_{i_{0},j_{0}}^{(1)} $ par \[ \sigma_{i_{0},j_{0}}^{(1)}=\int_{\RR^{3n}}\frac{\chi_{i_{0},j_{0}}^{(1)}(\ss',\tt',\hat{\uu})}{|s_{j_{0}}|}d\ss' d\tt' d\hat{\uu}, \] o\`u \[ \chi_{i_{0},j_{0}}^{(1)}(\ss',\tt',\hat{\uu})=\left\{\begin{array}{lllllll}
1 & \mbox{si} \; |\ss'|,|\tt'|,|\hat{\uu}|\leqslant 1 \;  \mbox{et} \;  \left| \sum_{k\neq j_{0}}\frac{s_{k}t_{k}u_{k}}{s_{j_{0}}}\right| \leqslant 1 \; \mbox{et} \; |t_{i_{0}}|\leqslant |s_{j_{0}}|& & & & &\\ 0  &\mbox{sinon} & &  &
\end{array}\right.  \]

On remarque que 

\[ (\chi_{i_{0},j_{0}}^{(1)}-\tilde{\chi}_{i_{0},j_{0}})^{(1)}(\ss',\tt',\uu)=\left\{\begin{array}{lcrrc}
1 & \mbox{si} & |t_{i_{0}}|\leqslant |s_{j_{0}}| &&\\ & \mbox{et} & (\min|s_{i}|< \frac{1}{x_{i_{0}}} && \\  & \mbox{ou} & \min|t_{j}|< \frac{1}{y_{j_{0}}}&&\\ & \mbox{ou} & \min|u_{k}|< \frac{1}{P}& &\\ & \mbox{ou} & \left| \sum_{k\neq j_{0}}\frac{s_{k}t_{k}u_{k}}{s_{j_{0}}}\right|<\frac{1}{P})&&\\ 0  &\mbox{sinon} &&&
\end{array}\right. \]

On v\'erifie alors que  \[  \sigma_{i_{0},j_{0}}^{(1)}-\tilde{\sigma}_{i_{0},j_{0}}^{(1)}=\int_{\RR^{3n}}\frac{(\chi_{i_{0},j_{0}}^{(1)}-\tilde{\chi}_{i_{0},j_{0}})^{(1)}(\ss',\tt',\uu)}{|s_{j_{0}}|}d\ss' d\tt' d\hat{\uu}\ll \min\{ \frac{1}{x_{i_{0}}},\frac{1}{y_{j_{0}}},\frac{1}{P}\}, \]
et on a donc \begin{multline*} B^{n}\sum_{x_{i_{0}},y_{j_{0}}}\frac{\poids(x_{i_{0}},y_{j_{0}})}{x_{i_{0}}y_{j_{0}}}\sum_{q\leqslant Q}\sum_{\substack{(a_{i})_{i\neq i_{0}}\in (\ZZ /q\ZZ)^{n}\\  (a_{j}')_{ j \neq j_{0}}\in (\ZZ /q\ZZ)^{n}} }q^{-3n-1}S_{q}(\aa,\aa')(\0)(\sigma_{i_{0},j_{0}}^{(1)}-\tilde{\sigma}_{i_{0},j_{0}})^{(1)} \\ =O(B^{n}\log(B)). \end{multline*}

\begin{rem}
De la m\^eme mani\`ere, on a que \[ \int_{\RR^{3n}}\frac{\underline{w}_{\varepsilon,2}(\ss,\tt)}{|t_{i_{0}}|}\poids(0,\hat{\uu})d\ss'd\tt'd\hat{\uu}\\=\sigma_{i_{0},j_{0}}^{(2)}+O(\varepsilon)+O\left(\min\{ \frac{1}{x_{i_{0}}},\frac{1}{y_{j_{0}}},\frac{1}{P}\} \right) \] o\`u 
\[ \sigma_{i_{0},j_{0}}^{(2)}=\int_{\RR^{3n}}\frac{\chi_{i_{0},j_{0}}^{(2)}(\ss',\tt',\hat{\uu})}{|t_{i_{0}}|}d\ss' d\tt' d\hat{\uu}, \] et \[ \chi_{i_{0},j_{0}}^{(2)}(\ss',\tt',\uu)=\left\{\begin{array}{lrrrrrc}
1 & \mbox{si} & |\ss'|,|\tt'|,|\hat{\uu}|\leqslant 1 & \mbox{et} & \left| \sum_{k\neq i_{0}}\frac{s_{k}t_{k}u_{k}}{t_{i_{0}}}\right| \leqslant 1 & \mbox{et} & |t_{i_{0}}|\geqslant |s_{j_{0}}| \\ 0 & &\mbox{sinon} &  &
\end{array}\right. \]

\end{rem}

On note alors \[ \sigma_{i_{0},j_{0}}=\sigma_{i_{0},j_{0}}^{(1)}+\sigma_{i_{0},j_{0}}^{(2)}.\]

Par cons\'equent, on aura un terme principal donn\'e par :
\[ \sigma_{i_{0},j_{0}}B^{n}\sum_{x_{i_{0}},y_{j_{0}}}\frac{\poids(x_{i_{0}},y_{j_{0}})}{x_{i_{0}}y_{j_{0}}}\sum_{q\leqslant Q}\sum_{\substack{(a_{i})_{i\neq i_{0}}\in (\ZZ /q\ZZ)^{n}\\  (a_{j}')_{ j \neq j_{0}}\in (\ZZ /q\ZZ)^{n}} }q^{-3n-1}S_{q}(\aa,\aa')(\0)\left(1+O(\varepsilon)\right). \]

\subsection{D\'emonstration de la conjecture}\label{dernieresection}

Nous avons donc d\'emontr\'e que pour tous $ i_{0},j_{0} $ on a un terme principal du type \begin{multline*}
 \sigma_{i_{0},j_{0}}B^{n}\sum_{x_{i_{0}},y_{j_{0}}}\frac{\poids(x_{i_{0}},y_{j_{0}})}{x_{i_{0}}y_{j_{0}}}\sum_{q\leqslant Q}\sum_{\substack{(a_{i})_{i\neq i_{0}}\in (\ZZ /q\ZZ)^{n} \\  (a_{j}')_{j\neq j_{0}}\in (\ZZ /q\ZZ)^{n}} }q^{-3n-1}S_{q}(\aa,\aa')(\0)\left(1+O(\varepsilon)\right). 
\end{multline*} 
On peut r\'e\'ecrire ce terme sous la forme : 
\begin{equation}\label{finale} \sigma_{i_{0},j_{0}}B^{n}\sum_{q\leqslant Q }\sum_{\substack{\aa\in (\ZZ /q\ZZ)^{n+1} \\  \aa'\in (\ZZ /q\ZZ)^{n+1}} }q^{-3n-1}S_{q}(\aa,\aa')(\0)\sum_{\substack{x_{i_{0}}\equiv a_{i_{0}} (q)\\y_{j_{0}}\equiv a_{j_{0}}(q)}}\frac{\poids(x_{i_{0}},y_{j_{0}})}{x_{i_{0}}y_{j_{0}}}\left(1+O(\varepsilon)\right).  \end{equation}
Nous allons \`a pr\'esent \'evaluer la somme \[\sum_{\substack{x_{i_{0}}\equiv a_{i_{0}} (q)\\y_{j_{0}}\equiv a_{j_{0}}'(q)}}\frac{\poids(x_{i_{0}},y_{j_{0}})}{x_{i_{0}}y_{j_{0}}}=\sum_{x_{i_{0}}\equiv a_{i_{0}} (q)}\frac{\poids(x_{i_{0}})}{x_{i_{0}}}\sum_{y_{j_{0}}\equiv a_{j_{0}}'(q)}\frac{\poids(y_{j_{0}})}{y_{j_{0}}} \]
\`A $ x_{i_{0}} $ fix\'e, pour $ \varepsilon $ assez petit on a \begin{align*}
 \sum_{y_{j_{0}}\equiv a_{j_{0}}'(q)}\frac{\poids(y_{j_{0}})}{y_{j_{0}}} & =\sum_{\substack{y_{j_{0}}\equiv a_{j_{0}}'(q) \\  1\leqslant x_{i_{0}}y_{j_{0}}\leqslant B^{\frac{2}{3}}}}\frac{1}{y_{j_{0}}}+O(1) \\ & =\frac{1}{q}\int_{1}^{\frac{B^{\frac{2}{3}}}{x_{i_{0}}}}\frac{1}{y_{j_{0}}}dy_{j_{0}}+O\left(\int_{1}^{\frac{B^{\frac{2}{3}}}{x_{i_{0}}}}\frac{1}{y_{j_{0}}^{2}}dy_{j_{0}}\right)+O(1) \\ & =\frac{1}{q}(\frac{2}{3}\log(B)-\log(x_{i_{0}}))+O(1) \end{align*}
 Le terme $ O(1) $ fournit dans \eqref{finale} un terme d'erreur de type $ O(B^{n}\log(B)) $. On a par ailleurs que \begin{multline}\label{l2}
\frac{1}{q}\sum_{x_{i_{0}}\equiv a_{i_{0}} (q)}\frac{\poids(x_{i_{0}})}{x_{i_{0}}}(\frac{2}{3}\log(B)-\log(x_{i_{0}})) \\ =\frac{2\log(B)}{3q^{2}}\int_{1}^{B^{\frac{1}{3}}}\frac{1}{x_{i_{0}}}dx_{i_{0}}+O\left(\frac{1}{q}\int_{1}^{B^{\frac{1}{3}}}\frac{1}{x_{i_{0}}^{2}}dx_{i_{0}}\right)+O(1) \\ -\frac{1}{q^{2}}\int_{1}^{B^{\frac{1}{3}}}\frac{\log(x_{i_{0}})}{x_{i_{0}}}dx_{i_{0}}+O\left(\frac{1}{q}\int_{1}^{B^{\frac{1}{3}}}\frac{\log(x_{i_{0}})}{x_{i_{0}}^{2}}dx_{i_{0}}\right)+O(1) \\  = \frac{1}{q^{2}}\left(\frac{2}{9}\log(B)^{2}-\frac{1}{2}(\frac{1}{9}\log(B)^{2})\right)+O(1)=\frac{1}{q^{2}}\left(\frac{1}{6}\log(B)^{2}\right)+O(1)
\end{multline}
(ici encore, le terme $ O(1) $ fournit dans \eqref{finale} un terme d'erreur de type $ O(B^{n}) $).
Il reste donc \`a \'evaluer : \[ \sum_{q\leqslant Q}q^{-3n-3}\sum_{\substack{\aa\in (\ZZ /q\ZZ)^{n+1} \\  \aa'\in (\ZZ /q\ZZ)^{n+1}} }S_{q}(\aa,\aa')(\0). \] Posons \[ S_{q}(\0)=\sum_{\substack{\aa\in (\ZZ /q\ZZ)^{n+1} \\  \aa'\in (\ZZ /q\ZZ)^{n+1}} }S_{q}(\aa,\aa')(\0). \] On remarque par ailleurs que d'apr\`es le lemme \ref{lemmemoi}, $ S_{q}(\0)\ll q^{5+2n+\epsilon} $, donc \[ \sum_{q\geqslant Q}q^{-3n-3}S_{q}(\0)\ll  \sum_{q\geqslant Q}q^{2-n+\epsilon}=Q^{2-n+\epsilon}=B^{1-\frac{n}{2}+\frac{\epsilon}{2}}, \] et la s\'erie \[ \sum_{q\geqslant 1}q^{-3n-3}S_{q}(\0) \] est absolument convergente. Par cons\'equent : \[ \eqref{finale} =\frac{1}{6}\sigma_{i_{0},j_{0}}B^{n}\log(B)^{2}\left(\sum_{q=1}^{\infty}q^{-3n-3}S_{q}(\0)\right)(1+O(\varepsilon). \]

On v\'erifie par ailleurs que $ q\mt S_{q}(\0) $ est multiplicative, et donc  : \begin{align*}
\sum_{q=1}^{\infty}q^{-3n-3}S_{q}(\0)=\prod_{p \; \premier}\sigma_{p}
\end{align*}
o\`u $ \sigma_{p}=\sum_{t=0}^{\infty}p^{-(3n+3)t}S_{p^{t}}(\0). $
On a ainsi d\'emontr\'e la proposition ci-dessous :
\begin{prop}
Pour tous $ i_{0},j_{0} \in \{ 0,...,n\} $ et $ n\geqslant 35 $ on a : \[ \sum_{i_{0},j_{0}} N(\underline{w}_{\varepsilon,B,i_{0},j_{0}},B)=\frac{2}{9}\underbrace{\left(\prod_{p \; \premier}\sigma_{p}\right)\sigma_{\infty}}_{C}B^{n}\log(B)^{2}(1+O(\varepsilon)). \]
o\`u $ \sigma_{\infty}=\sum_{i_{0},j_{0}}\sigma_{i_{0},j_{0}} $
\end{prop}
On a de la m\^eme mani\`ere que \[ \sum_{i_{0},j_{0}} N(\overline{w}_{\varepsilon,B,i_{0},j_{0}},B)=\frac{1}{6}CB^{n}\log(B)^{2}(1+O(\varepsilon)). \]

Par ailleurs, on v\'erifie facilement que le nombre de points compt\'es par $ N(\overline{w}_{\varepsilon,B,i_{0},j_{0}},B) $ tels que $ x_{\max}=x_{i_{0}}=x_{i_{1}} $ (resp. $ y_{\max}=y_{j_{0}}=y_{j_{1}} $) pour $ i_{0}\neq i_{1} $ (resp. $ j_{0}\neq j_{1} $) est du type $ O(B\log(B)) $. En effet, pour le cas $ x_{\max}=x_{i_{0}}=x_{i_{1}} $ si l'on reprend la somme \eqref{numero} en restreignant la somme sur $ (x_{i})_{i\neq i_{0}} $ aux \'el\'ements $ (x_{i})_{i\neq i_{0}}\in \ZZ^{n} $ tels que $ x_{i_{1}}=x_{i_{0}} $ on trouve : \begin{multline*} B^{n}\sum_{x_{i_{0}},y_{j_{0}}}\frac{\poids(x_{i_{0}},y_{j_{0}})}{x_{i_{0}}^{n}y_{j_{0}}^{n+1}}\sum_{q\leqslant Q}\sum_{\substack{(a_{i})_{i\neq i_{0}}\in (\ZZ /q\ZZ)^{n} \\  (a_{j}')_{j \neq j_{0}}\in (\ZZ /q\ZZ)^{n} } }q^{-n-1}S_{q}(\aa,\aa')(\0) \\ \sum_{\substack{x_{i}\equiv a_{i}(q), i\neq i_{0} \\ x_{i_{1}}=x_{i_{0}} \\ y_{j}\equiv a_{j}'(q), j\neq j_{0}}}\frac{\underline{w}_{\varepsilon,1}(\xx,\yy)}{|x_{j_{0}}|}\int_{\RR^{n}}\poids(0,\hat{\uu})d\hat{\uu}.\end{multline*} On a alors que \[ \sum_{\substack{x_{i}\equiv a_{i}(q), i\neq i_{0} \\ x_{i_{1}}=x_{i_{0}} \\ y_{j}\equiv a_{j}'(q), j\neq j_{0}}}\frac{\underline{w}_{\varepsilon,1}(\xx,\yy)}{|x_{j_{0}}|} \ll \left(\frac{x_{i_{0}}}{q}\right)^{n-2}\left(\frac{y_{j_{0}}}{q}\right)^{n}\log(x_{i_{0}}), \] et le terme principal ci-dessus peut alors \^etre major\'e par : \[ B^{n}\sum_{x_{i_{0}},y_{j_{0}}}\frac{\poids(x_{i_{0}},y_{j_{0}})}{x_{i_{0}}^{2}y_{j_{0}}}\log(x_{i_{0}}) \ll B^{n}\log(B) \] (On proc\`ederait de m\^eme pour la fonction poids $ \underline{w}_{\varepsilon,2} $ en rempla\c{c}ant $ \frac{1}{|y_{i_{0}}|} $ par $ \frac{x_{i_{0}}}{y_{j_{0}}}\frac{1}{|x_{j_{0}}|} $, via $ |x_{j_{0}}y_{j_{0}}|\leqslant|x_{i_{0}}y_{i_{0}}| $). Le cas $ y_{\max}=y_{j_{0}}=y_{j_{1}} $ se traite de mani\`ere analogue. On en d\'eduit l'encadrement \eqref{encadrement}, et on a alors \begin{multline*} N'(B)=\card\{ (\xx,\yy,\zz)\in \ZZ_{0}^{n+1}\times \ZZ_{0}^{n+1}\times \ZZ_{0}^{n+1}\; |  \; x_{\max} >0, \; y_{\max}>0,\;  \\ \xx.\yy.\zz=0 ,\;   H(\xx,\yy,\zz)\leqslant B, \; \max_{i}|x_{i}|\max_{j}|y_{j}|\leqslant B^{\frac{2}{3}} \; \max_{i}|x_{i}|\leqslant B^{\frac{1}{3}} \} \\ =\frac{1}{6}CB^{n}\log(B)^{2}(1+o(1)),\end{multline*}.

Nous allons voir \`a pr\'esent que la donn\'ee de ce terme $ N'(B) $ nous permet de d\'eduire la valeur asymptotique du cardinal de $ E(B) $ o\`u
\begin{multline*} E(B)= \card\{(\xx,\yy,\zz)\in \ZZ_{0}^{n+1}\times \ZZ_{0}^{n+1}\times \ZZ_{0}^{n+1} \; | \; x_{\max}>0, \; y_{\max}>0 ,\; z_{\max}>0 \\ \xx.\yy.\zz=0, \;H(\xx,\yy,\zz) \leqslant B\}.  \end{multline*} En effet, on remarque qu'un point de $ E(B) $ appartient \`a l'un (au moins) des ensembles 
\begin{multline*} E_{1}(B)=\card\{ (\xx,\yy,\zz)\in \ZZ_{0}^{n+1}\times \ZZ_{0}^{n+1}\times \ZZ_{0}^{n+1}\; |  \; x_{\max} >0, \; y_{\max}>0,\; z_{\max}>0\\  \xx.\yy.\zz=0 ,\;   H(\xx,\yy,\zz)\leqslant B, \; \max_{i}|x_{i}|\max_{j}|y_{j}|\leqslant B^{\frac{2}{3}},\; \max_{i}|x_{i}|\leqslant B^{\frac{1}{3}} \},\end{multline*}
\begin{multline*} E_{2}(B)=\card\{ (\xx,\yy,\zz)\in \ZZ_{0}^{n+1}\times \ZZ_{0}^{n+1}\times \ZZ_{0}^{n+1}\; |   \; x_{\max} >0, \; y_{\max} >0, \; z_{\max}>0,\;  \\ \xx.\yy.\zz=0 ,\;   H(\xx,\yy,\zz)\leqslant B, \; \max_{j}|y_{j}|\max_{k}|z_{k}|\leqslant B^{\frac{2}{3}}, \; \max_{j}|y_{j}|\leqslant B^{\frac{1}{3}}\},\end{multline*}
\begin{multline*} E_{3}(B)=\card\{ (\xx,\yy,\zz)\in \ZZ_{0}^{n+1}\times \ZZ_{0}^{n+1}\times \ZZ_{0}^{n+1}\; |   \; x_{\max} >0,  \;y_{\max}>0, \;  z_{\max}>0,\;\\ \xx.\yy.\zz=0 ,\;   H(\xx,\yy,\zz)\leqslant B, \; \max_{k}|z_{k}|\max_{i}|x_{i}|\leqslant B^{\frac{2}{3}}, \; \max_{k}|z_{k}|\leqslant B^{\frac{1}{3}}\}\end{multline*}
(chacun de ces ensemble \'etant de m\^eme cardinal $ \frac{1}{2}N'(B) $). Nous allons \'etablir que le nombre de points appartenant \`a deux des ensembles $ E_{l}(B) $ est du type $ O(B^{n}\log(B)) $. Nous allons le faire pour les points $ (\xx,\yy,\zz) $ appartenant \`a $ E_{1,2}(B)=E_{1}(B)\cap E_{2}(B) $. Pour de tels points, notons $ x_{i_{0}}=\max_{i}|x_{i}| $, $ y_{j_{0}}=\max_{j}|y_{j}| $ et $ z_{k_{0}}=\max_{k}|z_{k}| $. \\

Le nombre de points contenus dans cette intersection peut \^etre major\'e par $ \sum_{i_{0},j_{0}}N(\underline{\tilde{w}}_{\varepsilon,B,i_{0},j_{0}},B) $ o\`u \[ \underline{\tilde{w}}_{\varepsilon,B,i_{0},j_{0}}(\xx,\yy,\zz)=\underline{w}_{\varepsilon,B}(x_{i_{0}},y_{j_{0}})\underline{w}_{\varepsilon}(\xx,\yy) \tilde{\poids}(\zz) \] avec $ \tilde{\poids}(\zz)=\poids(\zz) \prod_{k=0}^{n}\underline{\omega}_{\varepsilon}\left(1-\frac{y_{j_{0}}|z_{k}|}{B^{\frac{2}{3}}}\right)$ (qui prend en compte la condition suppl\'ementaire $ \max_{j}|y_{j}|\max_{k}|z_{k}|\leqslant B^{\frac{2}{3}} $). Il est clair que les r\'esultats obtenus dans la partie \ref{PartieTE} restent valables pour cette nouvelle fonction poids. On s'int\'eresse donc au terme principal

\begin{equation*}
\sum_{x_{i_{0}},y_{j_{0}}}\poids(x_{i_{0}},y_{j_{0}})\sum_{\substack{(x_{i})_{i\neq i_{0}} \\ (y_{j})_{j\neq j_{0}} }}\poids(\xx,\yy)\sum_{q=1}^{\infty}q^{-n-1}S_{q}(\xx,\yy)(\0)\tilde{I}_{q}(\xx,\yy)(\0), \end{equation*} o\`u l'on a not\'e :
 \begin{equation*} \tilde{I}_{q}(\xx,\yy)(\0)=\int_{\RR^{n+1}}\frac{\tilde{\poids}(\zz)}{Q^{2}}h\left(\frac{q}{Q},\frac{\xx.\yy.\zz}{Q^{2}}\right)d\zz.       \end{equation*}

On effectue par la suite un changement de variables $ z_{k}=\frac{B^{\frac{2}{3}}}{y_{j_{0}}} $, ce terme principal devient 

\begin{equation*}
B^{\frac{2}{3}n-\frac{1}{3}}\sum_{x_{i_{0}},y_{j_{0}}}\frac{\poids(x_{i_{0}},y_{j_{0}})}{y_{j_{0}}^{n+1}}\sum_{\substack{(x_{i})_{i\neq i_{0}} \\ (y_{j})_{j\neq j_{0}} }}\poids(\xx,\yy)\sum_{q=1}^{\infty}q^{-n-1}S_{q}(\xx,\yy)(\0)\tilde{I}(\0), \end{equation*} avec 

\begin{equation*} \tilde{I}(\0)=\int_{\RR^{n+1}}\tilde{\poids}(\uu)h\left(\frac{q}{Q},\frac{\xx.\yy.\uu}{B^{\frac{1}{3}}y_{j_{0}}}\right)d\zz.       \end{equation*} Puis, en supposant $ |x_{i_{0}}y_{i_{0}}|\leqslant |x_{j_{0}}y_{j_{0}}| $ comme pr\'ec\'edemment, et en rempla\c{c}ant la variable $ u_{j_{0}} $ par $ t=\frac{\xx.\yy.\uu}{B^{\frac{1}{3}}y_{j_{0}}} $ on a \[ \tilde{I}(\0)=\frac{B^{\frac{1}{3}}}{|x_{j_{0}}|}\int_{\RR}J(t)h\left(\frac{q}{Q},t\right)dt \] o\`u 

\[ J(t)=\int_{\RR^{n}}\tilde{\poids}(0,\hat{\uu})d\hat{\uu}. \] Rappelons que l'on a montr\'e (cf. Lemme \ref{TE}) que \[ \int_{\RR}J(t)h\left(\frac{q}{Q},t\right)dt\ll 1, \] le terme principal est donc du type : \begin{multline*} B^{\frac{2}{3}n}\sum_{x_{i_{0}},y_{j_{0}}}\frac{\poids(x_{i_{0}},y_{j_{0}})}{y_{j_{0}}^{n+1}}\sum_{\substack{(x_{i})_{i\neq i_{0}} \\ (y_{j})_{j\neq j_{0}} \\ |y_{i_{0}}|\leqslant \frac{x_{j_{0}}y_{j_{0}}}{x_{i_{0}}} }}\frac{\poids(\xx,\yy)}{|x_{j_{0}}|}\sum_{q=1}^{\infty}q^{-n-1}S_{q}(\xx,\yy)(\0) \\ \ll B^{\frac{2}{3}n}\sum_{x_{i_{0}},y_{j_{0}}}\frac{\poids(x_{i_{0}},y_{j_{0}})}{y_{j_{0}}^{n+1}}x_{i_{0}}^{n-1}y_{j_{0}}^{n} \ll B^{\frac{2}{3}n}\sum_{x_{i_{0}},y_{j_{0}}}\frac{x_{i_{0}}^{n-1}}{y_{j_{0}}} \ll B\log(B)\end{multline*} (en rappelant que $ x_{i_{0}}\leqslant B^{\frac{1}{3}} $ sur le support de $ \poids(x_{i_{0}},y_{j_{0}}) $).

On a donc que \[ \card(E_{1,2}(B))=O(B^{n}\log(B)), \]
et on en d\'eduit que \begin{align*}
N(B)=\card E(B) & = \card E_{1}(B)+\card E_{2}(B)+\card E_{3}(B)+O(B^{n}\log(B)) \\ &=\frac{3}{2}N'(B)+O(B^{n}\log(B)).
\end{align*}
Par cons\'equent, nous avons \'etabli : \[  N(B)=\frac{1}{2}\prod_{p \; \premier}\sigma_{p}\frac{\sigma_{\infty}}{2}B^{n}\log(B)^{2}(1+o(1)). \]

Remarquons que \begin{align*}
 \sigma_{i_{0},j_{0}}^{(1)}& =\int_{\RR^{3n}}\frac{\chi_{i_{0},j_{0}}^{(1)}(\ss',\tt',\hat{\uu})}{|s_{j_{0}}|}d\ss' d\tt' d\hat{\uu} \\ & =\sum_{k_{0}\neq j_{0}}\int_{|u_{k_{0}}|=\max | u_{k}|}\frac{\chi_{i_{0},j_{0}}^{(1)}(\ss',\tt',\hat{\uu})}{|s_{j_{0}}|}d\ss' d\tt' d\hat{\uu} 
\end{align*}
En effectuant un changement de variables $ u_{k}=|u_{k_{0}}|v_{k} $ pour tout $ k\neq k_{0} $ on obtient alors (si $ \vv' $ d\'esigne $ \vv $ priv\'e de la coordonn\'ee en $ k_{0} $) :
\begin{align*}
 \sigma_{i_{0},j_{0}}^{(1)}& =\sum_{k_{0}\neq j_{0}}\int_{-1}^{1}|u_{k_{0}}|^{n-1}du_{k_{0}}\int_{\RR^{3n-1}}\frac{\chi_{i_{0},j_{0}}^{(1)'}(\ss',\tt',\hat{\vv})}{|s_{j_{0}}|}d\ss' d\tt' d\hat{\vv}'  =\frac{2}{n}\sigma_{i_{0},j_{0}}^{(1)'}
\end{align*}
o\`u \[ \sigma_{i_{0},j_{0}}^{(1)'}=\sum_{k_{0}\neq j_{0}}\int_{\RR^{3n-1}}\frac{\chi_{i_{0},j_{0}}^{(1)'}(\ss',\tt',\hat{\vv})}{|s_{j_{0}}|}d\ss' d\tt' d\hat{\vv}' \] et \[ \chi_{i_{0},j_{0}}^{(1)'}(\ss',\tt',\hat{\vv})=\left\{\begin{array}{lllllll}
1 & \mbox{si} \; |\ss'|,|\tt'|,|\hat{\vv}|\leqslant 1 \;  \mbox{et} \;  \left| \sum_{k\neq j_{0}}\frac{s_{k}t_{k}v_{k}}{s_{j_{0}}}\right| \leqslant 1 \; \mbox{et} \; |t_{i_{0}}|\leqslant |s_{j_{0}}|& & & & &\\ 0  &\mbox{sinon} & &  &
\end{array}\right. \] Par les m\^emes calculs on \'ecrit $ \sigma_{i_{0},j_{0}}^{(2)} $ sous la forme \[ \sigma_{i_{0},j_{0}}^{(2)}= \frac{2}{n}\sigma_{i_{0},j_{0}}^{(2)'} \] o\`u \[ \sigma_{i_{0},j_{0}}^{(2)'}=\sum_{k_{0}\neq i_{0}}\int_{\RR^{3n-1}}\frac{\chi_{i_{0},j_{0}}^{(2)'}(\ss',\tt',\hat{\vv})}{|t_{i_{0}}|}d\ss' d\tt' d\hat{\vv}'. \]
On a alors \[  \sigma_{\infty}=\frac{2}{n}\sigma_{\infty}' \] o\`u \[ \sigma_{\infty}'=\sum_{i_{0},j_{0}}(\sigma_{i_{0},j_{0}}^{(1)'}+\sigma_{i_{0},j_{0}}^{(2)'}). \] On peut donc \'ecrire : \[  N(B)=\frac{1}{2n}\prod_{p \; \premier}\sigma_{p}\sigma_{\infty}'{2}B^{n}\log(B)^{2}(1+o(1)). \]

Nous pouvons finalement calculer  le cardinal \begin{multline*} \mathcal{N}(B)=\card\{(\xx,\yy,\zz)\in \ZZ_{0}^{n+1}\times \ZZ_{0}^{n+1}\times \ZZ_{0}^{n+1} \; | \; \\ (x_{0},...,x_{n}), (y_{0},...,y_{n}), (z_{0},...,z_{n}) \; \primitifs \;,  \xx.\yy.\zz =0, \; H(\xx,\yy,\zz)\leqslant B\}. \end{multline*}

On remarque en effet que si $ N_{d,e,f}(B) $ d\'esigne
\begin{multline*}\card\{(d\xx,e\yy,f\zz)\in d\ZZ_{0}^{n+1}\times e\ZZ_{0}^{n+1}\times f\ZZ_{0}^{n+1} \; | \; \\  \xx.\yy.\zz =0, \; H(d\xx,e\yy,f\zz)\leqslant B\}=N(B/def) \end{multline*} et
 \begin{multline*} \mathcal{N}_{k,l,m}(B)=\card\{(k\xx,l\yy,m\zz)\in k\ZZ_{0}^{n+1}\times l\ZZ_{0}^{n+1}\times m\ZZ_{0}^{n+1} \; | \; \\ (x_{0},...,x_{n}), (y_{0},...,y_{n}), (z_{0},...,z_{n}) \; \primitifs \;,  \xx.\yy.\zz =0, \; H(k\xx,l\yy,m\zz)\leqslant B\} \\ =\mathcal{N}(B/klm ) \end{multline*} (pour $ d,e,f,k,l,m \in \NN $), alors on a 
  \[ N_{d,e,f}(B)=\sum_{d|k}\sum_{e|l}\sum_{f|m}\mathcal{N}_{k,l,m}(B). \] Par inversions de M\"{o}bius successives appliqu\'ees \`a $ (d,e,f)=(1,1,1) $, on obtient : \begin{align*} \mathcal{N}(B)=\mathcal{N}_{(1,1,1)}(B) & =\sum_{k\in \NN^{\ast}}\mu(k)\sum_{l\in \NN^{\ast}}\mu(l) \sum_{m\in\NN^{\ast}}\mu(m)N_{k,l,m}(B)\\ & =\sum_{k,l,m \in \NN^{\ast}}\mu(k)\mu(l)\mu(m)N(B/klm) \\ & \sim_{B\ra \infty}\frac{1}{2n}\prod_{p \; \premier}\sigma_{p}\sum_{k,l,m \in \NN^{\ast}}\frac{\mu(k)\mu(l)\mu(m)}{k^{n}l^{n}m^{n}}\sigma_{\infty}'B^{n}\log(B)^{2}.
 \end{align*} 
 On remarque que \[ \sum_{k,l,m \in \NN^{\ast}}\frac{\mu(k)\mu(l)\mu(m)}{k^{n}l^{n}m^{n}}=\left(\sum_{k\in \NN^{\ast}}\frac{\mu(k)}{k^{n}}\right)^{3}, \] et que \[\sum_{k\in \NN^{\ast}}\frac{\mu(k)}{k^{n}}=\prod_{p \; \premier}\left(1-\frac{1}{p^{n}}\right). \] 
On a donc finalement d\'emontr\'e le r\'esultat suivant : 
\begin{prop}\label{conclusion}
Pour tout $ n\geqslant 35 $, on a : \[\mathcal{N}(B)=\frac{1}{2n}\prod_{p \; \premier}\sigma_{p}'\sigma_{\infty}'B^{n}\log(B)^{2}(1+o(1)), \]
lorsque $ B \ra \infty $, o\`u l'on a not\'e $ \sigma_{p}'= \left(1-\frac{1}{p^{n}}\right)^{3}\sigma_{p} $
\end{prop}

Remarquons ce r\'esultat est en accord avec les conjectures \'enonc\'ees dans \cite{BT} et \cite{Pe}. En effet, commen\c{c}ons par d\'ecrire les singularit\'es de l'hypersurface $ X $. Ce sont les $ (\xx,\yy,\zz)\in \PP^{n}\times  \PP^{n}\times \PP^{n} $ tels que $ x_{i}y_{i}=y_{i}z_{i}=x_{i}z_{i}=0 $ pour tout $ i \in \{0,...,n\} $. On d\'efinit l'ensemble \begin{multline*} \mathcal{E}=\{ (I,J,K)\subset \{0,...,n\}^{3}\; |\;  I,J,K\neq \{0,...,n\}, \\  \; (I\cap J)\cup (J\cap K) \cup (I\cap K)=\{0,...,n\}\}.   \end{multline*}
\`a chaque triplet $ (I,J,K) $, on associe la vari\'et\'e $ V(I,J,K)=\{ (\xx,\yy,\zz)\in \PP^{n}\times  \PP^{n}\times \PP^{n} \; | \; x_{i}=y_{j}=z_{k}=0, \; \forall (i,j,k)\in I\times J\times K\} $ de dimension $ 3n-(\card(I)+\card(J)+\card(K)) $. On a alors que le lieu singulier $ X_{s} $ de l'hypersurface $ X $ est donn\'e par \[ X_{s}=\bigcup_{(I,J,K)\in \mathcal{E}}V(I,J,K), \] et on a que pour tout $ (I,J,K)\in \mathcal{E} $ \[ V(I,J,K)\setminus \bigcup_{(I,J,K)\varsubsetneq (I',J',K')\in \mathcal{E}}V(I',J',K') \] est lisse. \\

 Calculons \ a pr\'esent le cardinal de $ \mathcal{E} $ : commen\c{c}ons par remarquer que l'ensemble  $ \mathcal{F}=\{(I,J,K)\subset \{0,...,n\}^{3}\; |\;  (I\cap J)\cup (J\cap K) \cup (I\cap K)=\{0,...,n\} \}$ est de cardinal $ 4^{n+1} $ : en effet, par hypoth\`ese, chaque indice doit appartenir \`a au moins deux des ensembles $ I,J,K $, ce qui donne $ 4 $ possibilit\'es pour chaque indices $ i $ : $ i \in I,J $ et $ i\notin K $, ou  $ i \in J,K $ et $ i\notin I $, ou 
$ i \in I,K $ et $ i\notin J $ ou $ i \in I,J,K $. Pour trouver le cardinal de $ \mathcal{E} $, il faut retrancher \`a ce cardinal le nombre $ \card(\mathcal{G}) $, o\`u  \[ \mathcal{G}= \{  (I,J,K) \in  \mathcal{F} \; | \; I=\{0,...,n\} \; \ou \;J=\{0,...,n\}\; \ou \; K=\{0,...,n\} \}.\] Le nombre de triplets de $ \mathcal{F} $ tels que $ I=\{0,...,n\} $ est $ 3^{n+1} $ (par un calcul analogue \`a celui effectu\'e pr\'ec\'edemment), puis en proc\'edant de m\^eme avec les cas  $ J=\{0,...,n\} $ et $ K=\{0,...,n\} $, on trouve que le cardinal de $ \mathcal{G} $ vaut $ 3.3^{n+1} $ auquel on doit retrancher le nombre de triplets $ I,J,K $ de $ \mathcal{F} $ dont au moins deux des ensembles $ I,J,K $ sont $ \{0,...,n\} $ tout entier. Par un calcul analogue, on peut voir que ce nombre est $ 3.2^{n+1}-1 $. En r\'esum\'e, on a donc : \[ \card(\mathcal{E})=4^{n+1}-3.3^{n+1}+3.2^{n+1}-1=N_{0}. \]

On obtient une d\'esingularisation de la vari\'et\'e $ X $ en effectuant un \'eclatement en chaque $ V(I,J,K) $. Plus pr\'ecis\'ement, on commence par \'eclater $ X $ en chaque sous-vari\'et\'e de dimension $ 0 $. Puis on \'eclate la vari\'et\'e obtenue en les transform\'es stricts des sous-vari\'et\'es singuli\`eres $ V(I,J,K) $ de dimension $ 1 $ (qui sont alors lisses), etc... On note alors $ \pi : \tilde{X} \ra X $ cette d\'esingularisation. \\

On a alors que : \[ \Pic(\tilde{X})\simeq \Pic(X) \oplus \bigoplus_{(I,J,K)\in \mathcal{E}}\ZZ.\tilde{D}_{I,J,K} \simeq \ZZ^{3}\oplus \ZZ^{N_{0}}, \] ($ D_{I,J,K} $ d\'esignant le diviseur exceptionnel issu de l'\'eclatement en $ V(I,J,K) $ et que \[ C^{1}_{\eff}(\tilde{X})\simeq C^{1}_{\eff}(X) + \sum_{(I,J,K)\in \mathcal{E}}\RR^{+}.D_{I,J,K}. \]

Pour appliquer les r\'esultats \'enonc\'ees dans \cite{BT}, il convient de montrer que $ X $ est \`a singularit\'es canoniques, c'est-\`a-dire que, si $ K_{\tilde{X}} $ (resp. $ K_{X} $) d\'esigne le diviseur canonique de $ \tilde{X} $ (resp. $ X $) il existe des rationnels $ a_{I,J,K}\geqslant 0 $ tels que \[ K_{\tilde{X}}=\pi^{\ast}(K_{X})+\sum_{(I,J,K)\in \mathcal{E}}a_{I,J,K}D_{I,J,K}. \]
Pour cela nous allons utiliser \cite[Th\'eor\`eme 7.9]{K} : on consid\`ere un hyperplan g\'en\'eral de coordonn\'ees en $ \xx $ dans $ \PP^{n}\times \PP^{n}\times \PP^{n} $ passant par l'une des singularit\'es les plus petites (disons par exemple le point $ P=((1:0:...:0),(0:1:0:..:0)(0:0:1:0:...:0) $ et on note $ S $ le diviseur $ H\cap X $. D'apr\`es \cite[Th\'eor\`eme 7.9]{K}, si $ S $ est \`a singularit\'es canoniques, alors il en va de m\^eme pour $ X $. Par d\'efinition, l'hyperplan $ H $ admet une \'equation du type \[ \sum_{\substack{i=1}}^{n}a_{i}x_{i}=0, \] avec l'un des $ a_{i} $ (disons $ a_{n} $) non nul. Le diviseur $ S $ peut alors \^etre d\'ecrit comme l'hypersurface de $ \PP^{n-1}\times \PP^{n}\times \PP^{n} $ d'\'equation : \[ x_{0}y_{0}z_{0}+...+x_{n-1}y_{n-1}z_{n-1}+\left(\sum_{i=1}^{n}\frac{a_{i}}{a_{n}}x_{i}\right)y_{n}z_{n}, \] et le probl\`eme revient alors \`a montrer que cette hypersurface est \`a singularit\'e canoniques. En effectuant les m\^emes op\'erations avec des hyperplans en $ \yy $ et en $ \zz $, on se ram\`ene au cas d'une hypersurface de $ \PP^{n-1}\times \PP^{n-1}\times \PP^{n-1} $ du type : \[  x_{0}y_{0}z_{0}+...+x_{n-1}y_{n-1}z_{n-1}+\left(\sum_{i\neq 0}\frac{a_{i}}{a_{n}}x_{i}\right)\left(\sum_{j\neq 1}\frac{b_{j}}{b_{n}}y_{j}\right)\left(\sum_{k\neq 2}\frac{c_{k}}{c_{n}}z_{k}\right). \] En r\'eit\'erant ce proc\'ed\'e $ n-2 $ fois on v\'erifie que l'on se ram\`ene \`a une surface de $ \PP^{1}\times \PP^{1}\times \PP^{1} $ d'\'equation du type
\[ \alpha x_{0}y_{2}z_{0}+ \beta x_{1}y_{1}z_{0}+\gamma x_{1}y_{2}z_{2}+\delta x_{1}y_{2}z_{0}, \] ($ ((x_{0}:x_{1}),(y_{1}:y_{2}),(z_{0}:z_{2})) $ d\'esignant un triplet de coordonn\'ees homog\`enes de $ \PP^{1}\times \PP^{1}\times \PP^{1} $). On v\'erifie que cette surface admet une unique singularit\'e $ P_{0}=((1:0),(0:1),(0:1)) $ de multiplicit\'e $ 2 $. On obtient une d\'esingularisation de cette hypersurface en effectuant un \'eclatement de $ S $ en $ P_{0} $. On note $ Y=\PP^{1}\times \PP^{1}\times \PP^{1}  $  et $ \pi : \tilde{Y}\ra Y $ (resp. $ \pi : \tilde{S}\ra S $) l'\'eclatement de $ Y $ (resp. $ S $) en $ P_{0} $. La vari\'et\'e $ S $ est alors une hypersurface lisse de $ \tilde{Y} $.  Puisque $ Y $ est lisse et que $ \codim_{Y}(P_{0})=3 $, on a d'apr\`es \cite[Exercice 8.5]{Ha} : \begin{equation}\label{sing1} K_{\tilde{Y}}=\pi^{\ast}K_{Y}+2E, \end{equation} $ E $ d\'esignant le diviseur exceptionnel issu de l'\'eclatement en $ P_{0} $. D'autre part, on a \begin{equation}\label{sing2} \pi^{\ast}S=\tilde{S}+2E\end{equation} et la formule d'adjonction donne : \begin{equation}\label{sing3} K_{\tilde{S}}=(K_{\tilde{Y}}+\tilde{S})|_{\tilde{S}},\; \; \;   K_{S}=(K_{Y}+S) |_{S}.\end{equation} En regroupant les r\'esultats \eqref{sing1}, \eqref{sing2} et \eqref{sing3}, on obtient \[ K_{\tilde{S}}=\pi^{\ast}K_{S}, \] donc $ \pi : \tilde{S}\ra S $ est cr\'epant, donc $ S $ est \`a singularit\'es canoniques et ainsi $ X $ est \`a singularit\'es canoniques.\\

Nous pouvons \`a pr\'esent appliquer les r\'esultats de la section $ 3 $ de \cite{BT} : en effet si l'on note $ L=\OO_{Y}(1,1,1) $, on a alors $ \omega_{X}^{-1}=\OO(n,n,n)=L^{\otimes n} $, et puisque d'autre part on a $  \omega_{\tilde{X}}\simeq \pi^{\ast}\omega_{X}\otimes\bigotimes_{(I,J,K)\in \mathcal{E}} \mathcal{L}(a_{I,J,K}D_{I,J,K}) $ avec $ a_{I,J,K}\geqslant 0 $, la vari\'et\'e $ X $ est une vari\'et\'e $ \QQ $-Fano canonique (cf. \cite[D\'efinition 2.2.10]{BT}), ce qui permet alors de conclure (voir \cite[Exemple 2.3.7]{BT}) que La vari\'et\'e $ X $ est $ L $-primitive. Les r\'esultats de la section $ 3 $ de \cite{BT}, sugg\`erent alors que la formule asymptitique attendue pour $ \mathcal{N}(B) $ est \[ \mathcal{N}(B)=c_{L}(X)B^{\alpha_{L}(X)}\log(B)^{\beta_{L}(X)-1}(1+o(1)) \] (voir \cite[page 323]{BT}), et le r\'esultat obtenu en \ref{conclusion} semble \^etre compatible avec cette formule. En effet on a $ \alpha_{L}(X)=n $, $ \beta_{L}(X)=3 $ (cf. \cite[Remarque 2.3.12]{BT}) et   \[ c_{L}(X)=\frac{\gamma_{L}(X)}{\alpha_{L}(X)(\beta_{L}(X)-1)!}\delta_{L}(X)\tau_{L}(X) \] o\`u \[ \delta_{L}(X)=\card H^{1}(\Gal(\bar{\QQ}/\QQ,\Pic(X))=1, \] \[ \tau_{L}(X)=n^{3}\sigma_{\infty}'\prod_{p \; \premier}\sigma_{p}' \] (voir par exemple \cite{S2} pour les calculs), et \[ \gamma_{L}(X)= \int_{C^{1}_{\eff}(X)^{\vee}}e^{-\langle\omega_{X}^{-1},y\rangle}dy\] o\`u  \begin{align*}
C^{1}_{\eff}(X)^{\vee} & =\{y \in \Pic(X)\otimes\RR^{\vee} \; | \; \langle x,y\rangle\geqslant 0, \; \forall x\in C^{1}_{\eff}(X)\} \\ & \simeq \{(y_{1},y_{2},y_{3}) \in \RR^{3} \; | \; x_{1}y_{1}+x_{2}y_{2}+x_{3}y_{3}\geqslant 0, \; \forall x_{1},x_{2},x_{3}\geqslant 0 \}, \end{align*} donc \[\gamma_{L}(X)= \int_{(\RR^{+})^{3}}e^{-(ny_{1}+ny_{2}+ny_{3})}dy_{1}dy_{2}dy_{3}=\left(\int_{0}^{\infty}e^{-ny}dy\right)^{3}=\frac{1}{n^{3}}  \]

et on a ainsi  \begin{align*}
 c_{L}(X) & =\frac{1}{2n}\frac{1}{n^{3}}n^{3}\sigma_{\infty}'\prod_{p \; \premier}\sigma_{p}' \\ & =\frac{1}{2n}\sigma_{\infty}'\prod_{p \; \premier}\sigma_{p}'.
\end{align*}
Le r\'esultat obtenu est donc bien en accord avec les conjectures \'enonc\'ees dans \cite{BT}.

\printnomenclature


\begin{thebibliography}{[GH]}
\bibitem[B-B]{BB}
V. Blomer, J. Br\"{u}dern,
{\it Counting in hyperbolic sikes : the diophantine analysis of multihomogeneous diagonal equations}, arXiv : 1402.1122v1.
\bibitem[B-T]{BT}
V. V. Batyrev, Yu. Tschinkel,
{\it Tamagawa numbers of polarized algebraic varieties}, Ast\'erisque. \textbf{251} (1998) 299-340.
\bibitem[Da]{D}
H. Davenport, 
{\it Analytic methods for Diophantine equations and Diophantine inequalities}, $ 2^{eme} $ \'edition, Cambridge University Press, (2005).
\bibitem[Ha]{Ha}
R. Hartshorne,
{\it Algebraic Geometry.} Graduate Texts in Mathematics. \textbf{52}, Springer-Verlag, New York (1977) 
\bibitem[HB1]{HB1}
D. R. Heath-Brown,
{\it A new form for the circle method, and its application to quadratic forms}, J. Reine Angew. Math. \textbf{481}. (1996) 149-206.
\bibitem[HB2]{HB2}
D. R. Heath-Brown,
{\it The circle method and diagonal cubic forms   }, R. Soc. Lond. Philos. Trans. Ser. A \textbf{356}. (1998) 673-699.
\bibitem[HB3]{HB3}
D. R. Heath-Brown,
{\it Cubic forms in 14 variables}. Invent. Math. \textbf{170}. (2007) 199-230.
\bibitem[HB4]{HB4}
D. R. Heath-Brown,
{\it The density of rational points on curves and surfaces}, Ann. Math. \textbf{155} (2002) 553-598. 
\bibitem[Ho]{Ho}
C. Hooley,
{\it On Waring's problem}, Acta. Math. \textbf{157} (1986) 49-97.

\bibitem[Hu]{Hu}
L. K. Hua,
{\it Sur une somme exponentielle }, C. R. Acad. Sci. Paris, \textbf{210}. (1940) 520-523.
\bibitem[K]{K}
J. Koll\'ar,
{Singularities of pairs}, Proc. Symp. Pure Math. Amer. Math. Soc. \textbf{63}  (1997) 221-286.
\bibitem[Pe]{Pe}
E. Peyre,
{\it Hauteurs et mesures de Tamagawa sur les vari\'et\'es de Fano}, Duke Math. J. \textbf{79}. (1995) 101-218.
\bibitem[Ro]{R}
M. Robbiani,
{\it On the number of rational points of bounded height on smooth bilinear hypersurfaces in biprojective space}, J. London Math. Soc. \textbf{63}. (2001) 33-51.
\bibitem[Sch1]{S1}
D. Schindler,
{\it Bihomogeneous forms in many variables}, arXiv. 1301.6516.
\bibitem[Sch2]{S2}
D. Schindler,
{\it Manin's conjecture for certain biprojective hypersurfaces}, arXiv. 1307.7069.
\bibitem[St]{S}
E. Stein,
{\it Harmonic analysis}, Princeton Mathematical Series \textbf{43}. (Princeton univeresity press, 1993).








\end{thebibliography}
\end{document}